\documentclass[11pt,a4,fleqn,english]{amsart}

\usepackage[left=2.35cm,right=2.35cm,top=4cm,bottom=4cm]{geometry}
\usepackage{graphicx}
\usepackage{graphics}
\usepackage{graphicx}

\usepackage{bbm}
\usepackage{dsfont}

\RequirePackage{amsthm,amsmath}

\usepackage{amsmath}   
\usepackage{amsfonts,amssymb,color}
\newcommand\del[1]{}

\usepackage{url}      

\usepackage{babel}
\usepackage{latexsym}

\newcommand\OS{\mbox{\rm Os--}\!\!}
\newcommand\Hyp{\mbox{Hyp}}

\newtheorem{notation}{Notation}[section]
\newtheorem{thm}{Theorem}[section]
\newtheorem{rem}{Remark}[section]
\newtheorem{prop}{Proposition}[section]
\newtheorem{cor}{Corollary}[section]
\newtheorem{ex}{Example}[section]
\newtheorem{defn}{Definition}[section]

\newcommand{\CU}{\mathcal{U}}

\newcommand{\Leb}{\mbox{Leb}}

\newcommand{\ta}{{\tilde a}} 
\newcommand{\deltaa}{{\delta}} 
\newcommand{\gr}{{\langle}} 
\newcommand{\gl}{{\rangle}} 
\newcommand{\ggxi}{\gr |\xi|\gl}
\newcommand{\ggx}{\gr |x|\gl}

\newcommand{\lgxi}{\gr |\xi|\gl}
\renewcommand{\sharp}{\circ}

\newcommand{\CB}{{\mathcal{B}}}

\newcommand{\cal}{C}

\newcommand{\tinv}[1]{\tfrac{1}{#1}}

\newcommand\sou[1]{}

\numberwithin{equation}{section}
\newcommand\red[1]{{\color{red}#1}}

\newcommand{\bcase}{\begin{cases}}
\newcommand{\ecase}{\end{cases}}
\newcommand{\pmat}{\begin{pmatrix}}
\newcommand{\epmat}{\end{pmatrix}}

\newcommand{\levy}{L\'evy }
\newcommand{\supp}{{\mbox{supp}}}

\newcommand{\barray}{\begin{array}{rcl}}
\newcommand{\earray}{\end{array}}

\newcommand{\If}{{\;\mbox{ if }}}

\newcommand{\lqq}{\lefteqn}
\newcommand{\dom}{{\;\mbox{dom}}}

\newcommand{\la} {{\langle}}
\newcommand{\ra} {{\rangle}}

\newcommand{\CBB} {{\mathcal{B}}}

\newcommand{\CP} {{\mathcal{P}}}

\newcommand{\CN} {{\mathcal{N}}}

\newcommand{\CSS} {{\mathcal{S}}}

\newcommand{\lk}{\left}
\newcommand{\rk}{\right}

\newcommand{\ep} {\varepsilon }

\newcommand{\be} {\begin{enumerate}}
\newcommand{\ee} {\end{enumerate} }

\newcommand{\CF}{{ \mathcal{ F } }}

\newcommand{\CA}{{ \mathcal{ A } }}

\newcommand{\CC}{{\mathbb{C}}}

\newcommand{\RR}{{\mathbb{R}}}

\newcommand{\NN}{\mathbb{N}}

\newcommand{\PP}{{\mathbb{P}}}

\newcommand{\EE}{ \mathbb{E} }
\newcommand{\TT}{{\rm I \kern -0.2em T}}

\newcommand{\DEQS}{\begin{eqnarray*}}
\newcommand{\EEQS}{\end{eqnarray*}}
\newcommand{\DEQSZ}{\begin{eqnarray}}
\newcommand{\EEQSZ}{\end{eqnarray}}

\begin{document}

\title[The Analyticity of Markovian semigroup]
{On Markovian semigroups of L\'evy driven SDEs, symbols and pseudo--differential operators}
%
\author{Pani W.~Fernando}
 \address{Department of Mathematics, Faculty of Applied Sciences, University of Sri Jayewardenepura, Gangodawila, Nugegoda,
     Sri Lanka.}
  \email{bpw@sci.sjp.ac.lk}
\author{E.~Hausenblas}
        \address{Department of Mathematics and Informationtechnology,
                 Montanuniversitaet Leoben,
 Austria.}

\email{erika.hausenblas@unileoben.ac.at}

\author{Fahim Kistosil}
        \address{Department of Mathematics and Informationtechnology,
                 Montanuniversitaet Leoben,
 Austria.}
\email{....}

\date{\today}

%
\maketitle
\begin{abstract}
We analyse analytic properties of nonlocal transition semigroups associated with a class of stochastic differential equations (SDEs) in $\RR^d$ driven by pure jump--type L\'evy processes.
First, we will show under which conditions the semigroup will be analytic on the  Besov space $B_{p,q}^ m(\RR^d)$ with $1\le p, q<\infty$ and $m\in\RR$.
Secondly, we present some applications by proving the strong Feller property and give weak error estimates for approximating schemes of the SDEs over the Besov space $B_{\infty,\infty}^ m(\RR^d)$.
\end{abstract}

\section{Introduction}
\label{intro}
The purpose of the article is to show smoothing properties for the Markovian semigroup generated by stochastic differential equations driven by pure jump--type L\'evy processes.
To be more precise, let $L=\{L(t):t\ge 0\}$ be a family of \levy processes.
Let us  consider the stochastic differential equations of the form
\DEQSZ\label{eq1}
\lk\{\barray
dX^ x(t) &=& b(X^ x(t-))\, dt +\sigma(X^ x(t-)) dL(t)
\\
X^ x(0)&=&x,\quad x\in\RR^ d,
\earray\rk.
\EEQSZ
where $\sigma:\RR^ d\to\RR^ d$ and $b:\RR^ d \to\RR^ d $ are Lipschitz continuous. Under this assumption, the existence and uniqueness of a solution to equation
\eqref{eq1} is well established, see for e.g. \cite[p. 367, Theorem 6.2.3]{applebaum}.
Let $(\CP_{t})_{t\ge0}$ be the Markovian semigroup associated to $X$ defined
by
\DEQSZ\label{semigroup1}
\lk( \CP_t f\rk) (x) := \EE \lk[ f(X^ x(t))\rk],\quad t\ge 0,\quad x\in\RR^d.
\EEQSZ
Then, it is known that  $(\CP_{t})_{t\ge0}$ is a Feller semigroup (see \cite[Theorem 6.7.2]{applebaum}) and
its infinitesimal generator is given by
$$
Au(x) = \int_\RR e^ {ix^T\xi} a(x,\xi)\hat u(\xi)\, d\xi\quad u\in \CSS(
\RR^ d),
$$
where the symbol $a$ is defined by
\DEQSZ\label{defsym}
a(x,\xi) := -\lim_{t \downarrow 0} \frac 1t \EE \lk[ e^ {i (X^ x(t)-x)^ T \xi} -1\rk] ,\quad x\in\RR^ d.
\EEQSZ
In \cite{BG1}, the authors investigate the analytic properties of the Markovian semigroup generated by an SDE driven by a L\'evy process (see Theorem 2.1 in \cite{BG1}). These type of results are used to solve several applications which arise in fields related to probability theory such as nonlinear filtering theory \cite{copula}, or stochastic numerics (see section \eqref{secondapp}). In this article, we put a further step and investigate under which constraints the corresponding Markovian semigroup $(\CP_t)_{t\ge 0}$ driven by an SDE with pure jump noise forms an analytic semigroup in the Besov spaces $B_{p,q}^ m(\RR^d)$ with $1\le p, q<\infty$ and $m\in\RR$. Here, we used Besov-spaces due to two reasons. First, Besov spaces are quite general; one covers on one side the space of continuous functions and the on the other side, e.g.\ the scale of Hilbert spaces $L^2(\RR^d)$ and $H^s_2(\RR^d)$, $s\in\RR$
(see \cite[2.3.5]{triebel2} or \cite[p.\ 14]{Runst+Sickel_1996}). Even, if we exclude in our results the case where $q=\infty$ or $p=\infty$, by embedding Theorems (see \cite[p.\ 30-31]{Runst+Sickel_1996}), one get easily good estimates for
$B_{\infty,\infty}^ s(\RR^d)$, $s\not\in\NN$, a space which coincides with $C^s_b(\RR^d)$. 
In this way, we can use the analyticity property of the Markovian semigroup $(\CP_t)_{t\ge 0}$ in Besov spaces to obtain the strong Feller property of $(\CP_t)_{t\ge 0}$. The strong Feller property of the Markovian semigroup associated with $\RR^d$-valued SDEs plays an important role in the long time behaviour or to show the uniqueness of an invariant measure of SDEs. In this way, we obtain our first motivation from studying the regularity of the Markovian semigroup $(\CP_t)_{t\ge 0}$ (e.g. see Corollary \eqref{sftheo1} and Corollary \eqref{sfcor2}) associated with equation \eqref{eq1}. In particular, we were interested in getting weak assumptions on the coefficients $b$ and $\sigma$.
The second motivation comes from studying the Monte-Carlo error of an approximation of an SDE driven by \levy noise. To be more precise, it enables us to obtain an explicit estimate of the distance between the semigroup  associated with the original problem \eqref{eq1} and the semigroup  associated with certain approximations of the original problem.

In \cite[Theorem 2.2]{ishikawa} and \cite{picard}, the authors get some results on the density of the solution of an SDE driven by a L\'evy process.
These estimates are uniform in space and are related to our results, see Corollaries \ref{densitytheo1} and \ref{density}.
%
In \cite{bally17}, they represent their main result as the propagation of the regularity of the Markovian semigroups
induced by the solution process of an SDE driven by a  Brownian motion and a L\'evy process. 
In particular, they show that for all $k\in\NN$ there  exists a constant $C>0$ depending on the operator $a$ and $T>0$ such that
$$
\sup_{0<t\le T}\|\CP_t f\|_{W^k_{\infty}} \le C\|f\|_{W^k_\infty},
$$
for all $f\in C_b^k(\RR^d)$. Here $\|f\|_{k,\infty}$ is the supremum norm of $f$ and its first $k$ derivatives. In the case of $k=0$, this means that the semigroup $(\CP_t)_{t\ge 0}$ is a Feller semigroup.
K\"uhn, \cite{kuehn}, investigates the Feller property of the Markovian semigroup for unbounded diffusion coefficients. In \cite{marinelli}, the analyticity of the Markovian semigroup
$(\CP_t)_{t\ge0}$ is proven for SDEs with only additive noise; the noise has to have a very special form.
 Notice also that in \cite{Dong}, authors derived a Bismuth-Elworthy-Li type formula for the L\'evy process in the Hilbert space setting. In \cite{ishi2} the authors got nice estimates on the density of solution processes driven by a pure L\'evy processes. They investigated the case in $\RR^d$, $d>1$,
 where the L\'evy measure is only supported by the axes and have different exponents. In addition, they  give some short time estimates for the density.

\medskip
The paper is organised as follows. In section \ref{symbol}, we give a short review of the symbols associated with the SDEs driven by L\'evy processes and introduce some notations. The main result of this article, i.e.\ the invertibility of the pseudo-differential operator,  with the detailed proof is postponed to section \ref{nsec}, followed
by section \ref{sec3}, where we study under which conditions on the symbol, the semigroup $(\CP_t)_{t\ge 0}$ is an analytic semigroup in a general Besov  space $B_{p,q}^ m(\RR^d)$, $1\le p,q<\infty$.
The motivation for our main results, i.e. two applications to solution processes of stochastic differential equations,  are prioritised and  presented in section  \ref{firstapp}  and  \ref{secondapp}.
Thus, as the first application, we verify in section \ref{firstapp} under which constrains the semigroup $(\CP_t)_{t\ge 0}$ is strong Feller.
In section \ref{secondapp}, we calculate the rate of convergence for the Monte Carlo error for SDEs driven by a pure L\'evy process.
The theoretical result is also verified by some  numerical experiments.
Finally, in the appendix, we give a short overview of pseudo-differential operators.


\begin{notation}
For a multiindex $\alpha=(\alpha_1,\alpha_2,\ldots,\alpha_n)\in\NN^n$ let $|\alpha|=\alpha_1+\cdots+\alpha_n$ and $\alpha!={\alpha_1}!\cdots \alpha_n!$.
For  an element $\xi\in\RR^n$, let $\xi^\alpha$ be defined by $\xi_1^{\alpha_1}\xi_2^{\alpha_2}\cdots \xi_n^{\alpha_n}$.
Moreover for a function $f:\RR^d\to \mathbb{C}$ we write $\partial_x ^ \alpha f(x)$ for
$$ {\partial ^ \alpha \over \partial {x_1}\partial {x_2}\cdots \partial {x_d}} f(x).
$$
In addition, let us define the brackets $\langle\,\cdot\,\rangle:\RR\ni \xi \mapsto \langle\xi\rangle^\rho:=(1+|\xi|^2 )^\frac \rho 2\in\RR$.
Following inequality, also called {\sl Peetres inequality}, is used on several places
$$
\langle x+y\rangle^ s \le c_s \langle x\rangle^ s \langle y\rangle^ {|s|},\quad x,y\in \RR^ d , \, s\in \RR.
$$

\medskip

For a multiindex $\alpha=(\alpha_1,\alpha_2,\ldots,\alpha_n)\in\NN^n$ let $|\alpha|=\alpha_1+\cdots+\alpha_n$ and $\alpha!={\alpha_1}!\cdots \alpha_n!$.
For a multiindex $\alpha=(\alpha_1,\alpha_2,\ldots,\alpha_n)\in\NN^n$ and an element $\xi\in\RR^n$ let $\xi^\alpha$ be defined by $\xi_1^{\alpha_1}\xi_2^{\alpha_2}\cdots \xi_n^{\alpha_n}$.

\medskip

Let $X$ be a non empty set and $f,g:X\to[0,\infty)$. We set
$f(x)\lesssim g(x)$,$x\in X$, iff there exists a $C>0$ such that $f(x)\le C g(x)$ for all $x\in X$. Moreover, if $f$ and $g$ depend on a further variable $z\in Z$, the statement
for all $z\in Z$, $ f(x,z)\lesssim g(x,z)$, $x\in X$ means that for every $z\in Z$ there exists a real number $C_z>0$ such that $f(x,z)\le C_z g(x,z)$ for every $x\in X$.
Also we set $f(x)\asymp g(x)$, $x\in X$, iff $f(x)\lesssim g(x)$ and $g(x)\lesssim f(x)$ for all $x\in X$. Finally, we say
$f(x)  \gtrsim g(x)$, $x\in X$, iff $g(x)  \lesssim f(x)$, $x\in X$. Similarly as above, we handle the case if the functions depend on a further variable.

Let $\CSS(\RR^d)$ be the Schwartz space of infinite often differentiable functions  where  all derivatives
decreases faster than any power of $|x|$ as $|x|$ tends to infinity. Let $\CSS'(\RR^d)$ be the dual of $\CSS(\RR^d)$.

If $m\in\NN$ we define
$$
\cal^m_b(\RR^d ):= \{ f\in \cal^0_b(\RR^d ): D^\alpha f\in \cal^0_b(\RR^d ), |\alpha|\le m\}
$$
endowed with the norm
$$
|f|_{\cal_b^m}:= \sum_{|\alpha|\le m} |D^\alpha f|_{\cal^0_b}.
$$
Let $s\in\RR\setminus \NN$, then we put $s=[s]+\{s\}$, where $[s]$ is an integer and $0\le \{s\}<1$.
Then 
$$
C_b^{s}(\RR^d ) := \lk\{ f\in C_b^{[s ]}(\RR^d ): \sum_{|\alpha|=[s ]} \sup_{x\not =y} \frac{\lk| D^\alpha f(x)-D^\alpha f(y)\rk| }{|x-y|^{\{s \}}}<\infty\rk\}
$$
equipped with the norm
$$
|f|_{C_b^{s }}:= |f|_{C_b^{[s]}}+ \sum_{|\alpha|=[s ]} \sup_{x\not =y} \frac{\lk| D^\alpha f(x)-D^\alpha f(y)\rk| }{|x-y|^{\{s \}}}.
$$
In order to define Besov spaces
 as  given in \cite[Definition 2,
pp. 7-8]{Runst+Sickel_1996} let us chose  first a function
$\psi\in{\mathcal{S}}(\mathbb{R}^d)$ such that $0\le \psi(x)\le 1$, $x\in
\mathbb{R}^d$ and
$$
\psi(x) = \left\{ \begin{array}{rcl} 1,&\mbox{ if } & |x|\le 1,\\
0&\mbox{ if } & |x|\ge \frac 32.
\end{array}\right.
$$
Then, let us put \begin{eqnarray*}
 \left\{ \begin{array}{rcl}\phi_0(x) &=&\psi(x), \; x\in \mathbb{R}^d,
\\
\phi_1(x) &=&\psi(\frac x2)-\psi(x), \; x\in \mathbb{R}^d,
\\
\phi_j(x) &=&\phi_1(2 ^{-j+1} x),\; x\in \mathbb{R}^d, \quad
j=2,3,\ldots.
\end{array} \right.
\end{eqnarray*}
Since we need it later on let
\DEQSZ\label{defu1}
\mathcal{U}_1&=&\supp(\phi_1).
\EEQSZ
\begin{defn}
Let $s\in{\mathbb{R}}$,  $0<p\le \infty$ and $f \in {\mathcal{S}}^\prime(\mathbb{R}^d)$. If $0<q <
\infty$  we put
\begin{eqnarray*} \left| f\right|_{B^{s}_{p,q}} &=& \left( \sum_{ j=0} ^\infty  2
^{sjq}\left| {\mathcal{F}} ^{-1} \left[ \phi_j{\mathcal{F}} f\right] \right|_{L ^p} ^q\right)
^\frac 1q = \Vert\Big( 2 ^{sj}\left| {\mathcal{F}}
^{-1} \left[ \phi_j{\mathcal{F}} f\right]\right|_{L ^p}\Big)_{j\in \mathbb{N}} \Vert_{l^q}.
\end{eqnarray*}
\end{defn}
\end{notation}
\section{Symbols, their definitions and properties}\label{symbol}
\label{aD}

In this section, we give a short review of symbols coming up Hoh's and L\'evy's symbols
in dealing with processes generated by L\'evy processes. Besides, we introduce some notations.
Throughout the remaining article, let $L=\{L^x(t):t\ge 0,x\in\RR^ d\}$ be a family of \levy processes $L^ x$, where $L^x$ is a \levy process starting  at $x\in\RR^ d$. Then $L$ generates a Markovian semigroup $( \CP_t)_ {t\ge 0}$ on $C_b (\RR^ d)$
by
$$ \CP_t f (x) := \EE f(L^x(t)),\quad f\in C_b (\RR^ d).
$$
Let $A$ be
 the infinitesimal generator of $(\CP_t)_{t\ge 0}$ acting on $\cal^{2}_b(\RR^ d)$ defined by
\DEQSZ\label{inf_g}
A\, f:= \lim_{h\to 0} \frac 1 h \lk( \CP_h-\CP_0\rk)\, f, \quad f\in \cal^{2}_b(\RR^ d).
\EEQSZ
Another way of defining $A$ is done by  \levy symbols (see \cite{hoh1}). In particular let 
$$
\psi(\xi)=\frac 1t \ln ( \EE e^{i\la\xi,L(t)\ra} ) \, ,\quad \xi\in\RR^ d.
$$
Observe that we have  (see e.g.\ \cite[P.42]{applebaum} and \cite{sato})
$$\psi(\xi) = \int_{\RR^ d\setminus\{0\}}\lk( e ^{i\la \xi,z\ra}-1-i\la \xi, x\ra\mathbb{I}_{\{|z|\le 1\}}\rk) \, \nu(dz) 
, \quad \xi\in\RR^ d.
$$
If $L$ is a pure jump process with symbol $\psi$, then the
infinitesimal generator defined by \eqref{inf_g} can also be
written as
\DEQSZ\label{opdef}
A\, f & =& - \int_{\RR^ d} e^{i \la \xi, x\ra } \psi(\xi) \hat f (\xi)\, d\xi, \quad f\in \CSS( 
\RR^ d).
\EEQSZ
The operator $A$, usually denoted in the literature by $\psi=\psi(D)$,  is well defined in $\cal^{2}_b(\RR^ d)$, has values
in $\CBB _b(\RR^ d)$ (bounded Borel functions in $\RR^ d$) and satisfies the positive maximum principle
(see e.g.\ \cite[Theorem 4.5.13 ]{Jacob-I}).
Therefore, 
 $A$ generates a Feller semigroup on $\cal^\infty_b(\RR^ d)$
and a sub--Markovian semigroup on $L^2(\RR^ d)$ (see  e.g.\
\cite[Theorem 2.6.9 and Theorem 2.6.10]{Jacob-II}).
To characterise the symbol, we introduce the generalised Blumenthal--Getoor index.
\begin{defn}
Let $L$ be a \levy process with symbol $\psi$ and  $\psi\in \cal^{k}_b(\RR^d\setminus\{0\})$ for some  $k\in\NN_0$. Then
the Blumenthal--Getoor index of order $k$ is defined by
$$
s := \inf_{\lambda>0\atop |\alpha|\le k} \lk\{\lambda: \lim_{\xi\to\infty} {|\partial^ \alpha_\xi \psi(\xi)|\over |\xi|^{\lambda-|\alpha|} }=0\rk\}.
%
$$
\del{Let
$$
s^+ := \inf_{\lambda>0\atop |\alpha|\le k} \lk\{\lambda: \limsup_{|\xi|\to\infty} {|\partial^ \alpha_\xi \psi(\xi)|\over |\xi|^{\lambda-|\alpha|} }=0\rk\},
$$
be the upper and
$$
s^- := \inf_{\lambda>0\atop |\alpha|\le k} \lk\{\lambda: \liminf_{|\xi|\to\infty} {|\partial^ \alpha_\xi  \psi(\xi)|\over |\xi|^{\lambda-|\alpha|} }=0\rk\},
$$
be the lower Blumenthal--Getoor index $s^+$ of order $k$.
}
Here $\alpha$ denotes a  multi-index. If $k=\infty$ then Blumenthal--Getoor index of infinity order is defined by
$$
s := \inf_{\lambda>0\atop \alpha \,\,\mbox{\tiny \rm is a multi--index} } \lk\{\lambda: \lim_{|\xi|\to\infty} {|\partial^ \alpha_\xi \psi(\xi)|\over |\xi|^{\lambda-|\alpha|} }=0\rk\}.
$$
\end{defn}
\begin{rem}\label{essremark}
For a function $\psi:\RR^d\to\RR$, the limit $ \lim_{\xi\to\infty}\psi(\xi)$ is a sloppy formulation and means actual
  $$\sup_{\xi\in \CU_1}\lim_{\lambda\to\infty} \psi(\lambda\xi),
  $$   where $\mathcal{U}_1$ defined in \eqref{defu1}.
  This can be easily seen by analysing the e.g. the proof of boundedness of the corresponding operator and realizing that the estimates comes up in analysing the summands after decomposing the operator in its dyadic partition of the unity.
\end{rem}
\begin{rem}
The Blumenthal--Getoor index of order infinity is defined for the sake of completeness. We are interested in weakening the assumption on the symbol, i.e.,
reducing the order $k$.
\end{rem}
To analyse properties of the Markovian semigroup $(\CP_t)_{t\ge 0}$ and to define the resolvent of the associated operator
$\psi(D)$, the range of the symbol is of importance.

\begin{defn}\label{defrange}
Let $\mathfrak{Rg}(\psi)$ be the essential range of $\psi$, i.e.\
$$
\mathfrak{Rg}(\psi):=\{ y\in \CC\mid  \Leb( \{ s\in\RR^d: |\psi(s)-y|<\ep\})>0 \mbox{ for each $\ep>0$}\}\footnote{Here, $\Leb$ denotes the Lebesgue measure.}.
$$
\end{defn}

Finally, to characterize the spectrum of the associated operator, one can introduce the type of a symbol.
\begin{defn}\label{deftype}
We call a symbol $\psi$
is of type $(\omega,\theta)$, $\omega\in\RR$, $\theta\in (0,\frac \pi
2)$,
iff 
$$
-\mathfrak{Rg}(\psi) \subset \CC\setminus \{\omega\}+\Sigma_{\theta+\frac\pi 2}.
$$
\end{defn}

\begin{rem} \label{rem21} 

If a symbol  $\psi$ is of type $(0,\theta)$, then there  exists a
constant $c>0$ such that
$$
|\Im(\psi(\xi))|\le c\, \Re\psi(\xi),\quad  \xi\in\RR^d.
$$ This is called sector condition of the symbol  $\psi$.
\end{rem}

Before presenting a typical example, we introduce stable processes, compare \cite[Chapter 3]{sato}).
\begin{defn}
A probability measure $\mu$ on $\RR^d$ is infinitely divisible, if for any positive integer $n\in\NN$, there exists a probability measure $\mu_n$ on $\RR^d$ such that $\mu=\mu_n^{( n)\ast}$.\footnote{The symbol $\ast$ denotes the convolution of two probability measures.}
\end{defn}

Observe, also, due to the independent increments of a L\'evy process,
is the distribution function $L(t)$, $t>0$, for any L\'ey
process  an infinite divisible probability measure.
\begin{defn}
An infinite divisible probability measure is stable, if for any $a>0$, there exist  numbers $b>0$ and $c\in\RR^d$ such that
$$
\hat \mu(z)^a=\hat \mu(bz)\, e^{i\la c,z\ra},\quad z\in\RR^d.
$$
The measure $\mu$  is called strictly stable, if for any $a>0$ there exists a $b>0$ such that
$$
\hat \mu(z)^a=\hat \mu(bz),\quad z\in\RR^d.
$$
\end{defn}
\begin{defn}
Let $\{X(t):t\ge 0\}$ be a L\'evy process on $\RR^d$. It is called a stable or strictly stable process, if the distribution for $X(1)$ is a
stable, respectively, a strictly stable infinite divisible measure.
\end{defn}

\begin{ex}
Let $L$ be a one dimensional strictly  $\alpha$--stable process with
\del{In particular, $L$ be a real-valued L\'evy process with initial value $L_0=0$
that satisfies the self-similarity property
$$
 {L_t}/{t^\frac 1 \alpha} \stackrel{d}{=} L_1,\quad \forall t>0.
$$} symbol given by $\phi(\xi)=c\,|\xi|^\alpha$, the parameter $\alpha$ is called the exponent of the process (see \cite[Section 14, p. 77]{sato}).
Let $\sigma$ and $b$ be two Lipschitz continuous functions on $\RR$. Then, for $\alpha>1$,
the symbol $$a(x,\xi):= | \sigma(x)\xi|^ \alpha + i b(x)\xi$$
is of type $(0,\theta)$.
If $\sigma$ is bounded away from zero, then the generalized Blumenthal--Getoor index is $\alpha$.
Let us assume that  $d=1$ and  $\sigma(x)= x$ for $|x|\le 1$, $\sigma\in \cal^ \infty_b(\RR^ d)$ and bounded from zero in $\RR\setminus(-1,1)$.
If $b(x)\le |x|^ \alpha $, then the Blumenthal--Getoor index is again $\alpha$.  
\end{ex}

\begin{rem}\label{remarkC2}
For $\lambda\in\CC\setminus\mathfrak{Rg}(\psi):=
\{\zeta\in \CC: \exists \xi $ with $ \psi(\xi)=\zeta\}$ we have (see Theorem 1.4.2 of \cite{haase})
$$
\| R(\lambda,\psi(D))\|\le {1\over \mbox{dist}(\mathfrak{Rg}(\psi),\lambda)}.
$$
Moreover, the set $\mathfrak{Rg}(\psi)$ equals the spectrum of the generator $A$.
\end{rem}

For several examples of L\'evy processes and their symbols, we refer to \cite{BG1}.
In case there is no dependence on the space variable, one can derive properties of the Markovian semigroup
directly using the range of the symbol. Given
a solution process of an SDE, usually, the associated infinitesimal generator of the Markovian semigroup depends on the space variable $x$.
In particular,
for $x\in\RR^d$ let $X=\{X^x(t):t\ge 0\}$ be a $\RR^d$-valued solution of the SDE  given in \eqref{eq1} and, as before
let $(\CP_{t})_{t\ge0}$ be the associated Markovian semigroup defined in \eqref{semigroup1}.
%
 Let $\psi$ be the \levy symbol of the \levy process $L=\{L(t):t\ge 0\}$.
Then, one can show (see Theorem 3.1 \cite{{schillingSchnurr}}), that the infinitesimal generator of the Markovian semigroup associated to $X^x$ has the symbol $a:\RR^ d \times \RR^ d \to \mathbb{C}$ given by
\DEQSZ\label{schill}
a(x,\xi)& =& \psi( \sigma^ T(x)\,\xi), \quad (x,\xi)\in \RR^ d \times \RR^ d.
\EEQSZ
Let $a_1(x,\xi)$ and $a_2(x,\xi)$ be two given symbols.
Due to the dependence on $x$, the corresponding operators $a_1(x,D)$ and $a_2(x,D)$ do not necessarily commute.
Therefore, many techniques working for operators induced by symbols being independent of the space variable $x$ do not work for operators induced by symbols depending on the space  variable $x$.
Especially, tricks relying on the  Bony's paraproduct gets much more demanding.

\del{\begin{rem}\label{rm1}
The results such as Sectorial condition (Remark \eqref{rem21}), $(\omega,\theta)-$property, Remark \eqref{remarkC2}, etc. are involved with the case where symbol $\psi$ is only depend on $\xi$ variable (i.e. $\psi$ is independent of spacial variable $x$). Since the coefficient function $\sigma$ of $a(x,\xi) = \psi( \sigma^ T(x)\,\xi)$ is bounded as well as bounded away from zero We could  see that those results also valid for the symbol given by  $a(x,\xi) = \psi( \sigma^ T(x)\,\xi)$.
\end{rem}}

\medskip

In our main result Theorem \ref{main2} we show under which conditions on the symbol $\psi$ and on the coefficients $\sigma$ and $b$
the Markovian semigroup $(\CP_t)_{t\ge 0}$ is an analytic semigroup in general Besov spaces $B_{p,q}^s(\RR^d)$.
To be more precise, we show if $\sigma$ is bounded away from zero, $\sigma$ and $b$ are smooth enough, and $\psi$
is of type $(0,\theta)$, $\theta<\frac \pi 2$, is sufficiently smooth,  and has Blumenthal-Geetor index $\delta\in (1,2)$ of sufficiently high order, then
the Markovian semigroup is analytic on $B_{p,q}^s(\RR^d)$ for $p,q\in[1,\infty)$.

\medskip

The choice of Besov spaces is twofold. To explain this, observe first that Besov spaces can be defined via the Fourier transform
(see the paragraph notation or  \cite[Definition 2,
pp. 7-8]{Runst+Sickel_1996}). Since the  operator associated to the symbol $a(x,\xi)$  can  be represented by a kernel of the form
$$
a(x,D)  f(x) = \int_{\RR^d} k(x,x-y) f(y)\, dy, \quad x\in\RR^d,
$$
where the kernel is given by the inverse Fourier transform\footnote{$ \CF_{\xi\to z}[f(x,\xi)](z)=\int_{\RR^d}e^{ -2\pi i \xi z}a(x,\xi)\, d\xi$.}
$$
k(x,z) = \CF_{\xi\to z} \lk[ a(x,\xi)\rk] (z),
$$
 Besov spaces
come up naturally. Secondly, the strong Feller property is defined via the space of continuous functions $C^{s}_b(\RR^d)$,
which is related to the Besov space $B_{\infty,\infty}^{s}(\RR^d)$ for $s\not =0$.
So, it suggests by itself to use Besov spaces and embedding Theorems to prove the strong Feller property for $(\CP_t)_{t\ge 0}$.

\del{
On the other side, s
\begin{thm}\label{main2}
Let us assume that the symbol is of type $(0,\theta)$ and $$\psi\in \CA_{2d+4,d+3;1,0}^ {\delta}(\RR^d\times\RR^d)\cap \mbox{Hyp}_{2d+4,d+3;1,0}^{\delta}\del{\delta}(\RR^d\times\RR^d),$$
 where $1<\delta<2$ is the Blumenthal--Getoor index of order $2d+4$ of $L$.
In addition, let us assume that
\begin{itemize}
\item  $\sigma\in \cal^{d+3}_b(\RR^d)$
\item  and $ b\in \cal^{d+3}_b(\RR^d)$,
\item and
that there exists a number $c>0$ such that
$$
\inf_{x\in\RR^d}|\sigma(x)|\ge c.
$$
\end{itemize}
Then, for all $1\le p,q<\infty$ and $m\in\RR$, the Markovian semigroup $(\CP_t)_{t\ge 0}$ defined in \eqref{semigroup1}
is analytic in $B^m _{p,q}(\RR^d)$.
\end{thm}

}

\section{A short introduction to pseudodifferential operators}

In order to treat pseudo-differential operators, different classes of symbols have been introduced.
Here, we closely follow the definition of \cite{pseudo2}.

\begin{defn} 
Let 
 $\rho,\delta$ two real numbers such that $0\le \rho\le 1$ and $0\le \delta\le 1$.
  Let $S^m_{\rho,\delta}(\RR^ d\times \RR^ d)$ be the set of all functions $a:\RR ^d \times \RR^ d \to \mathbb{C}$, where
 \begin{itemize}
   \item  $a(x,\xi)$ is infinitely often differentiable, i.e.\ $a\in \cal^\infty_b(\RR^d \times\RR^d)$;
   \item
for any two multi-indices $\alpha$ and $\beta$ there exist a constant $C_{\alpha,\beta}>0$  such that 
$$
\lk| \partial ^ \alpha_{\xi'} \partial ^ \beta_x a(x,\xi)\mid_{\xi'=\xi\gamma} \rk|\leq C_{\alpha,\beta}  
\la |\gamma\xi|\ra  ^ {m-\rho|\alpha|}\ggx^{\delta| \beta|},\quad x\in \RR^d,\xi\in \CU_1,\gamma\ge 1 .
$$
 \end{itemize}
\end{defn}
\del{\begin{rem}\label{essremark}
For a function $\psi:\RR^d\to\RR$, the limit $ \lim_{\xi\to\infty}\psi(\xi)$ is a sloppy formulation and means actual
  $$\sup_{\xi\in \CU_1}\lim_{\lambda\to\infty} \psi(\lambda\xi),
  $$   where $\mathcal{U}_1$ defined in \eqref{defu1}.
  This can be easily seen by analysing the e.g. the proof of boundedness of the corresponding operator and realizing that the estimates comes up in analysing the summands after decomposing the operator in its dyadic partition of the unity.
\end{rem}}
We call any function $a(x,\xi)$ belonging to  $\cup_{m\in\RR} S^m_{0,0}(\RR^d ,\RR^d )$ a {\sl symbol}.
For many estimates, one does not need that the function is infinitely often differentiable.
It is often only necessary to know the estimates with respect to $\xi$ and $x$ up to a particular order.
For this reason, one also introduces the following classes.
\begin{defn}(compare \cite[p. 28]{wong})
Let $m\in\RR$.  Let $\CA ^m_{k_1,k_2;\rho,\delta}(\RR ^d,\RR ^d)$ be the set of all functions $a:\RR ^d \times \RR^ d \to \mathbb{C}$, where
 \begin{itemize}
   \item  $a(x,\xi)$ is $k_1$--times differentiable in $\xi$ and $k_2$ times differentiable in $x$;
   \item 
for any two multi-indices $\alpha$ and $\beta$ with $|\alpha|\le k_1$ and $|\beta |\le k_2$,
there exists a  constant $C_{\alpha,\beta}>0$ depending only on $\alpha$ and $\beta$ such that
$$
\lk| \partial ^ \alpha_{\xi'} \partial ^ \beta_x a(x,\xi)\mid_{\xi'=\xi\gamma}  \rk|\le C_{\alpha,\beta}\la |\gamma\xi|\ra   ^ {m-\rho|\alpha|}\ggx^{\delta |\beta|},\quad x\in \RR^ d ,\xi\in \CU_1,\gamma\ge 1 .
$$
\end{itemize}
\end{defn}
\noindent
Moreover, one can introduce a  semi--norm in $\CA^m_{k_1,k_2;\rho,\delta}(\RR^d,\RR ^d)$ by
\DEQS
\lqq{ \| a\| _{\CA_{k_1,k_2;{\rho,\delta}}^m }
} &&
\\
&=& \sup_{|\alpha|\le k_1,|\beta|\le k_2} \, \sup_{(x,\xi)\in \RR^d\times\CU_1\times 
 \RR } \lk| \partial ^ \alpha_\xi \partial ^ \beta_x a(x,\xi) \mid_{\xi=\xi'\gamma} \rk|
\la |\gamma\xi|\ra    ^ {\rho|\alpha|-m}\ggx^{\delta|\beta|}, \quad a\in \CA ^m_{k_1,k_2;\rho,\delta}(\RR^ d\times \RR^ d).
\EEQS
We have seen in the introduction that, given a symbol,  one can define an operator. In case the symbol $\psi$ is a L\'evy symbol, the operator
defined by \eqref{opdef} is the infinitesimal generator of the semigroup of the L\'evy process.
In case one has an arbitrary symbol, the corresponding operator can be defined similarly.

\begin{defn}(compare \cite[p.28, Def. 4.2]{wong})
Let $a(x,\xi)$ be a symbol. 
Then, to  $a(x,\xi)$ corresponds an operator $a(x,D)$ defined by
$$
\big( a(x,D) \,u\big)\, (x): = \int_{\RR^d} e^ {i\la x,\xi\ra } a(x,\xi)\,\hat u(\xi)\, d\xi,\quad x \in \RR^d,\, u\in \CSS( 
\RR^ d)
$$
and is called pseudo--differential operator.
\end{defn}
In order to invert the symbol that associate with the operator, it should satisfy the elliptic property. 

\begin{defn}(compare \cite[p.\ 35]{pseudo})\label{elliptic}
	A symbol $a\in S^{m}_{\rho,\delta}(\RR^ d\times \RR^ d)$ is called globally {\sl elliptic},
	if there exists a number  $r>0$,
	$$
	\gr |\gamma\xi| \gl^ {m} \lesssim \lk|  a(x,\gamma\xi)\rk|,\quad \gamma\ge r, \, \xi\in\CU_1, \, x\in\RR^d .
	$$
\end{defn}
Later on, we will see that we need upper estimates not only for the symbol itself but also for its derivatives.
A more sophisticated definition of the norm is given below.
\begin{defn}(compare \cite[p.\ 35]{pseudo})\label{hypo}
	Let $m,\rho,\delta$ be real numbers with $0\le \delta <\rho\le 1$. The class $\Hyp_{\rho,\delta}^{m}(\RR^ d\times \RR^ d) $ consists of all functions
	$a(x,\xi)$ such that
	\begin{itemize}
		\item $a(x,\xi)\in \cal^\infty_b(\RR^ d\times \RR^d)$;
		\item there exists some $r>0$ such that
	$$
	\gr |\gamma\xi| \gl^ {m} \lesssim \lk|  a(x,\gamma\xi)\rk|,\quad \gamma\ge r, \, \xi\in\CU_1, \, x\in\RR^d ,
	$$
		and for an arbitrary multi-indices $\alpha$ and $\beta$  there exists a constant $C_{\alpha,\beta}>0$ with
		$$
		\lk| \partial_{\xi'}^ \alpha \partial_x^\beta a(x,\xi')\big|_{\xi'=\gamma\xi} \rk| \leq C_{ \alpha, \beta} \gr|\gamma\xi|\gl ^{m-\rho|\alpha|}\gr |x|\gl ^{\delta|\beta|}.
		$$
		for $x\in \RR^d$, $\xi\in\CU_1, \gamma\ge r $.
	\end{itemize}
	In addition, for $k_1,k_2\in\NN_0$, we define the following semi--norm given by
	\DEQS
	\lk\|a\rk\|_{\mbox{\rm \tiny Hyp}^{m}_{k_1,k_2;\rho,\delta}}= \sup_{|\alpha|\le k_1,|\beta|\le k_2}\sup_{x\in\RR^d} \limsup_{\xi\in\CU_1,\gamma\to\infty  
	}\lk|\partial ^\alpha_{\xi'} \partial_x^\beta \lk[ \frac  1 {a(x,\xi)}\Big|_{\xi'=\gamma\xi}\rk] \rk|\, \gr |\gamma\xi|\gl^{m+\rho |\alpha|}
	\gr |x|\gl^{\delta |\beta|}   .
	\EEQS
	
\end{defn}

\section{Analyticity of the Markovian semigroup in general Besov spaces}
\label{sec3}

Given a function  space $\mathbb{X}$ over $\RR^d$ we would be interested under which conditions on the coefficients $\sigma$, $b$ and the symbol $\psi$,
the Markovian semigroup $(\CP_{t})_{t\ge0}$ generates an analytic semigroup on $\mathbb{X}$.
Here, one has first to verify that $(\CP_{t})_{t\ge0}$ generates a strongly continuous semigroup. The necessary and sufficient conditions to fulfill a given semigroup 
is strongly continuous semigroup are given by Hille--Yosida Theorem.
%
Let us assume that $\mathbb{X}$ is a Banach space. 
For an operator $A$, let $\rho(A)$ represent the resolvent set, i.e.\ $\rho(A)=\{\lambda\in\CC: (\lambda I-A) $ is invertible $ \}$ and $\sigma(A)=\CC\setminus \rho(A)$.  Now, if $(A, D(A))$ is closed, densely defined, and for any $\lambda\in\mathbb{C}$ with $\Re \lambda>0$ one has $\lambda\in \rho(A)$ (compare \cite[Theorem 3.5, p.\ 73]{engel}, or \cite[Theorem 1.5.2]{Pazy:83})
and
\DEQSZ\label{estresol}
\| R(\lambda,A)\|_{L(\mathbb{X},\mathbb{X})}\le \frac 1 {\Re \lambda},
\EEQSZ
then $A$ generates a strongly continuous semigroup on $\mathbb{X}$.
Secondly, to show that this strongly continuous semigroup is analytic, one has to show either that
\del{\DEQSZ\label{estresnr}
\| R(\theta+i\tau ,A)\|_{L(\mathbb{X},\mathbb{X})}\le \frac C {|\tau|},\quad \theta>0,\tau\in\RR.
\EEQSZ
or } (compare \cite[Theorem 4.6, p.\ 101]{engel})
\DEQSZ\label{toshowanal}
M:= \sup_{t>0} \| t A \CP_t\|_{L(\mathbb{X},\mathbb{X})}<\infty,
\EEQSZ
or 
\DEQSZ
\label{estres}
\| R(\vartheta+i\tau :A)  \|_{L(\mathbb{X},\mathbb{X})} &\le &\frac {C(A)}{|\tau |}, \quad \vartheta>0,\vartheta,\tau\in\RR.
\EEQSZ

\del{In addition, if $A$ satisfies (compare \cite[Theorem 3.5, p.\ 73]{engel}, or \cite[Theorem 4.5.2]{pazy})
\DEQSZ\label{estresolsec}
\| R(\sigma+i\tau ,A)\|_{L(\mathbb{X},\mathbb{X})}\le \frac 1 {|\tau|},
\EEQSZ}

Let $S(A)=\{\langle x^\ast,Ax\rangle: x\in D(A),x\in \mathbb{X}^\ast, \|x\|=1,\|x^\ast\|=1,\langle x^\ast,x\rangle=1\}$
be the numerical range of an operator $A$. If $\mathbb{X}$ is a Hilbert space and $\sigma$ constant, $S(A)$ can be characterized by the
 $\mathfrak{Rg}(\psi):=\overline{\{a(x,\xi)\in\CC:x,\xi\in \RR^d\}}$, where $a=(x,\xi):=\psi(\sigma^T\xi)$.
Since the range of $\psi$ contains the numerical range $S(A)$
of $A$, we have (see Remark \ref{remarkC2}),
\DEQSZ\label{estresnr}
\| R(\lambda,A)\|_{L(\mathbb{X},\mathbb{X})}\le \frac 1 {\mbox{dist}(\lambda,S(A))}.
\EEQSZ
Hence, for $\mathbb{X}=H^ m_2(\RR^d)$ and $\sigma(x)=\sigma_0$, one can show by  analyzing the numerical range, which is here given by
$$
S(a(x,D))=
\lk\{ \la x,a(x,D)x\ra: x\in \dom(a(x,D)),|x|_{H^m_2(\RR^d)}=1,\la x,x\ra_{H^m_2}
=1\rk\},
$$
and some purely geometric considerations, the analyticity of the 
semigroup $(\CP_t)_{t\ge 0}$ in $\mathbb{X}$.
Here $\la\,\, ,\,\,\ra$ represent the inner product in $H^ m_2(\RR^d)$. In fact, choosing  a complex number $\lambda=\vartheta+i\tau$ with $\vartheta>0$ and $\tau\in\RR$, and
using 
that the symbol $\psi$ is of type $(0,\theta)$,  we obtain by the following series of computations 
(see Theorem 3.9 \cite[Chapter I]{Pazy:83}),

 \DEQS
\lk\|  R(\lambda,a(x,D))\rk\|_{L(H^ m_2(\RR^d),H^ m_2(\RR^d))}&=&\lk\|  R(\vartheta+i\tau,a(x,D))\rk\|_{L(H^ m_2(\RR^d),H^ m_2(\RR^d))}\\
\le\frac 1 {\mbox{dist}(\lambda,\bar{S}(a(x,D)))}
\le \frac 1 {\mbox{dist}(\lambda,\rho(a(x,D)))}
&\le&\frac 1 {\mbox{dist}(\lambda,\rho(a(x,D)))}
 \\= \frac 1 {\mbox{dist}(\vartheta+i\tau,\rho(a(x,D)))}
 &\le &\frac{\cos\theta}{|\tau|}=\frac{C}{|\tau|},
\EEQS
where $C=\cos\theta$. These calculation implies that  $(\CP_t)_{t\ge 0}$ in $\mathbb{X}$ is an analytic semigroup in $\mathbb{X}$.

\medskip

This result can be generalized to arbitrary Besov spaces, for an introduction to Besov spaces we refer to \cite{triebel1}.
The motivation to analyse the analyticity of the Markovian semigroup in Besov spaces comes from the aim to investigate the strong Feller property of the  Markovian semigroup.
Since one has the embedding $C^{s}(\RR^d)\subset B_{\infty,\infty}^s(\RR^d)$ ($s\not =0$), it is nearby switching to Besov spaces.
If one abandon the Hilbert space setting, the numerical range gets more complicated and it is better to use other methods.
In the following Theorem we use the notations introduced in Appendix \ref{pseudo-app}.

\begin{rem}
	Using  the abstract theory of standard books such as e.g.\ \cite{Pazy:83,engel,triebel1}, we can prove the following two Theorems for  the pseudo--differential operator induced by a simple L\'evy process by proving that all assumptions of the corresponding theorems (Theorem 5.2 \cite[Chapter II]{Pazy:83}, Theorem 3.9 \cite[Chapter I]{Pazy:83}, Theorem 2.3.3 \cite[p.48]{triebel1}, etc.) are satisfied. In our case the underlying stochastic process is not a simple L\'evy process, but a solution of a stochastic differential equations. Nevertheless,
we apply the same method  to prove Theorem  \eqref{main22}. In particular, we   prove that all assumptions of the corresponding theorems (Theorem 5.2 \cite[Chapter II]{Pazy:83}, Theorem 3.9 \cite[Chapter I]{Pazy:83}, Theorem 2.3.3 \cite[p.48]{triebel1}, etc.) are satisfied.
\end{rem}

\begin{thm}\label{main22}
Let us assume that the symbol
\begin{itemize}
\item  the symbol $a$ belongs to $ \CA_{2d+4,d+3;1,0}^ {\delta}(\RR^d\times\RR^d)$,  where $1<\delta<2$,
\item the symbol $a$ belongs to $\mbox{Hyp}_{2d+4,d+3;1,0}^{\delta}
(\RR^d\times\RR^d)$,
\item and is of type $(0,\theta)$.
\end{itemize}
Then, for all $1\le p,q<\infty$ and $m\in\RR$, the  operator generates an analytic semigroup  $(\CP_t)_{t\ge 0}$
 in $B^m _{p,q}(\RR^d)$.
\end{thm}

Let $L=\{L(t):t\ge 0\}$ be a family of \levy processes.
We consider the stochastic differential equations of the form
\DEQS 
\lk\{\barray
dX^ x(t) &=& b(X^ x(t-))\, dt +\sigma(X^ x(t-)) dL(t)
\\
X^ x(0)&=&x,\quad x\in\RR^ d,
\earray\rk.
\EEQS
where $\sigma:\RR^ d\to L(\RR^ d, \RR^d)$ and $b:\RR^ d \to\RR^ d $ are  Lipschitz continuous.
Let $(\CP_t)_{t\ge 0}$ the Markovian semigroup of $X$ defined in \eqref{semigroup1}.
Applying Theorem \ref{main22} to the infinitesimal generator of $(\CP_t)_{t\ge 0}$ gives following Theorem.

\begin{thm}\label{main2}
Let us assume that the symbol is of type $(0,\theta)$ and $$\psi\in \CA_{2d+4,d+3;1,0}^ {\delta}(\RR^d\times\RR^d)\cap \mbox{Hyp}_{2d+4,d+3;1,0}^{\delta}\del{\delta}(\RR^d\times\RR^d),$$
 where $1<\delta<2$ is the Blumenthal--Getoor index of order $2d+4$ of $L$.
In addition, let us assume that
\begin{itemize}
\item  $\sigma\in \cal^{d+3}_b(\RR^d)$
\item  and $ b\in \cal^{d+3}_b(\RR^d)$,
\item and
that there exists a number $c>0$ such that
$$
\inf_{x\in\RR^d}|\sigma(x)|\ge c.
$$
\end{itemize}
Then, for all $1\le p,q<\infty$ and $m\in\RR$, the Markovian semigroup $(\CP_t)_{t\ge 0}$ defined in \eqref{semigroup1}
is analytic in $B^m _{p,q}(\RR^d)$.
\end{thm}

\begin{rem} The restriction that $p$ has to be strictly smaller than infinity comes from the fact that the space of Schwarz functions
$\CSS(\RR^d)$ is not dense in $B_{\infty,\infty}^m(\RR^d)$.
\end{rem}
\begin{proof}[Proof of Theorem \ref{main22}:]
For simplicity, let us denote $B_{p,q}^m(\RR^d)$ by $\mathbb{X}$.
Let us assume that the symbol $\psi$ and the coefficients $\sigma$ and $b$ are infinitely often differentiable. We first show that the operator $(\CP_{t})_{t\ge0}$ generates a strongly continuous semigroup on $\mathbb{X}$ by proving the required conditions in the Hille--Yoshida Theorem.
Theorem 2.3.3, p.48 in \cite{triebel1}, gives us that the Schwarz space $\CSS(\RR^d)$ is dense in $\mathbb{X}$. In addition, it is straight forward to show that $\CSS(\RR^d)\subset \dom(a(x,D))$. This immediately gives that  $\dom(a(x,D))$ is dense in $\mathbb{X}$.

Before starting, let us split the operator $a(x,D)$ into two operators in the same way as it is done in Theorem \ref{invert22}.
Let $R\in\NN$ sufficiently large such that
$$R\ge 6 \times  \lk\| a\rk\|_{\tilde{\CA}_{2d+4,d+3;1,0}^{-1,1}}
$$
and  $\lgxi^{\delta}\lesssim |a(x,\xi)|$ for all $x\in\RR^d$ and  $\xi\in \RR^d$ with $|\xi|\ge R$.
In addition, let $\chi\in \cal^\infty_b(\RR_0^+)$ such that
$$
\chi(\xi)=\bcase 0 &\mbox{ if } |\xi|\le 1,\\1 &\mbox{ if } |\xi|\ge 2, \ecase
$$
and put $b(x,\xi):= a(x,\xi)(1-\chi(\xi/R))$ and $\tilde a(x,\xi):=a(x,\xi) \chi(\xi/R)$.
We will show that $\tilde A=\ta(x,D)$ generates an analytic semigroup on $\mathbb{X}$. Due to Theorem 2.1 \cite[Chapter 3.2, p.\ 80]{Pazy:83},
it follows that $A=\tilde A+B$ with $B=b(x,D)$ generates an analytic semigroup on $\mathbb{X}$.

First, we will show that $\lk(\ta(x,D),\dom(\ta(x,D))\rk)$ is closed in $\mathbb{X}$. Let $\{v_n:n\in\NN\}\subset \dom(\ta (x,D))$ be a sequence such that
$\lim_{n\to\infty} v_n= v$ in $\dom(\ta (x,D))$ and $\lim_{n\to\infty}\ta (x,D) v_n= w$ in $\mathbb{X}$. To show that $\lk(\ta (x,D),\dom(\ta (x,D))\rk)$ is closed in $\mathbb{X}$,
we have to show that $\ta (x,D)v=w$. Suppose that $| \ta (x,D) v-w|_{\mathbb{X}}\neq 0$. In particular, there exists a constant $\hat C>0$ such that
 $| \ta (x,D) v-w|_{\mathbb{X}}\ge \hat{C}$. There exist a number $n_0\in \mathbb{N}$ such that for all $n\geq n_0$, we have
$$|  v-v_n|_{\dom(\ta (x,D))}<\frac{\hat{C}}{4\|\ta (x,D)\|_{L(\dom(\ta (x,D)),\mathbb{X})}}
$$
and
$$| \ta (x,D) v_n-w|_{\mathbb{X}}<\frac{\hat{C}}{4}.
$$
Since $\ta (x,D)$ is linear and bounded operator on $\dom(\ta (x,D))$, we have
\DEQS
\lqq{| \ta (x,D) v-w|_{\mathbb{X}}\le| \ta (x,D) v-\ta (x,D) v_n|_{\mathbb{X}}
+| \ta (x,D) v_n-w|_{\mathbb{X}}}
&&
\\
&\le &\|\ta (x,D)\|_{L(\dom(\ta (x,D)),\mathbb{X})}|v-v_n|_{\dom(\ta (x,D))}
+| \ta (x,D) v_n-w|_{\mathbb{X}}<\frac{\hat{C}}{2}.
\EEQS
 Since this  is a contradiction, we conclude that $w=\ta (x,D)v$.

Next,  we  show that there exists a constant $C>0$ such that
\DEQSZ\label{toshow11}
\lk\|  R(\lambda,\ta (x,D))\rk\|_{L(\mathbb{X},\mathbb{X})}\le \frac { C}  {|\lambda|},\quad \lambda \in \Sigma_{ \theta+\frac\pi 2}.
\EEQSZ
Here, we will apply Theorem \ref{invert22} to get the estimate. In order to
do this, first, note that the norm in $\tilde \CA^{-1,1}_{2d+4,d+3,;1,0}(\RR^d\times\RR^d)$ does not depend on $\lambda$. Hence,
$$  \lk\| \lambda+\ta\rk\|_{\tilde \CA_{2d+4,d+3;1,0}^{-1,1}}=  \lk\| \ta\rk\|_{\tilde \CA_{2d+4,d+3;1,0}^{\kappa,1}}, \quad \lambda\in \Sigma_{\theta+\frac \pi 2}.
$$

Next,
we have to estimate the norm of the operator $\lambda + \ta (x,D)$ in
$\mbox{Hyp}_{2d+4,d+3;1,0}^{\delta,\delta}\del{\delta}(\RR^d\times\RR^d)$.
That means, for any multiindices  $|\alpha|\le 2d+4$ and $|\beta|\le d+3$ we have to estimate
$$
\sup_{\lambda \in \Sigma_{\theta+\frac \pi 2} }\, |\lambda|\lk| \partial _ \xi^\alpha \partial _x^\beta \lk[ \frac 1 {\lambda +\ta (x,\xi)} \rk] \rk|\, .
$$
By straightforward calculations one can show that this entity is bounded.
Here, it is essential that $a(x,D)$ satisfies the sectorial condition, i.e.\ that  there exists a $c>0$ such that $|\Im(\tilde a(x,\xi))|\le c|\Re(\tilde a(x,\xi))|$.
We will consider the case where $|\alpha|=|\beta|=0$.
Separating the real and imaginary part we set $\lambda=\lambda_1+i\lambda_2$ and $\tilde a(x,\xi)=\psi_1(x,\xi)+i\psi(x,\xi)$.
Now we have
\DEQS
\frac 1{\lambda + \tilde a(x,\xi)}=\frac {\lambda_1+\psi_1(x,\xi)} { (\lambda_1+\psi_1(x,\xi))^2+(\lambda_2+\psi_2(x,\xi))^2}-i
\frac {\lambda_2+\psi_2(x,\xi)} { (\lambda_1+\psi_1(x,\xi))^2+(\lambda_2+\psi_2(x,\xi))^2}.
\EEQS
In particular, simple calculations give
\DEQS
\lk| \frac 1{\lambda + \tilde a(x,\xi)}\rk|\le
\frac {\sqrt{(\lambda_1+\psi_1(x,\xi))^2 + (\lambda_2+\psi_2(x,\xi))^2  }} { (\lambda_1+\psi_1(x,\xi))^2+(\lambda_2+\psi_2(x,\xi))^2}\le \frac 1{\lambda_2}
,
\EEQS
for $\lambda_1\ge 1$.
Next, we will consider the case where $|\alpha|=|\beta|=1$, that is
 let $\alpha=k$ and $\beta=l$ with $k,l\in\{1,\ldots,d\}$.
Then,
\DEQS
\partial_{x_l}\partial_{\xi_k}\lk[  \frac{1}{\lambda+\ta (x,\xi)}\rk]=
-\frac { \partial^2_{x_l\xi_k}\ta(x,\xi)}{(\lambda+a(x,\xi))^2}+\frac {2 \partial_{x_l}a(x,\xi) \partial_{\xi_k}\ta(x,\xi)}{(\lambda+\ta(x,\xi))^3}
.
\EEQS
For simplicity, we will not separate the real and imaginary part.
In this way we get 
\DEQS
|\lambda|\Big| \partial_{x_l}\partial_{\xi_k}\lk[  \frac{1}{\lambda+\ta(x,\xi)}\rk]\Big|
&\le & |\lambda|
\lk\{ \frac {r^ {-1}}{|\lambda +1|^2}+ \frac {r^ {-1}}{|\lambda +1|^3}\rk\}
\le  C(r),
\EEQS
where in the Definition \ref{hypo}, it is only necessary that there exists some $r>0$ such that  $\la|\xi|\ra^m\lesssim |a(x,\xi)| $ for  $\xi\in\RR^d$ with $|\xi|\ge r$. Similarly, we could get the bound for the general case where for multiindices  satisfying only $|\alpha|\le 2d+4$ and $|\beta|\le d+3$.
By an application of Theorem \ref{invert22} we know that  \eqref{toshow11} is satisfied.
In particular, that there exists a constant $C>0$ such that
\DEQSZ\label{analytic}
\lk\|  R(\lambda,\ta(x,D))\rk\|_{L(\mathbb{X},\mathbb{X})}\le \frac {C}  {|\lambda|},\quad \lambda \in \Sigma_{ \theta+\frac\pi 2},
\EEQSZ
 Finally it remains to show that the semigroup $(\CP_{t})_{t\ge0}$ is analytic over $\mathbb{X}$. Now pick $\lambda=\vartheta+i\tau\in \Sigma_{\theta+\frac{\pi}{2}}$ such that $\vartheta>0$ and $\tau\in\RR$. From the estimate \eqref{analytic}, we easily see that,
 \DEQS
\lk\|  R(\vartheta+i\tau,\ta(x,D))\rk\|_{L(\mathbb{X},\mathbb{X})}\le \frac { C_{d,s}}  {|\tau|},\quad \lambda \in \Sigma_{ \theta+\frac\pi 2},
\EEQS
 Then by applying the Theorem 5.2 \cite[Chapter II]{Pazy:83} we could conclude that  the Markovian semigroup is an analytic semigroup over $\mathbb{X}$.
\end{proof}

\medskip
Analyzing the proof of \cite[p.\ 58]{stein}
 one can see that the condition of the differentiability at the origin can be relaxed. 
Here, it is essential to mention that the proof relies on the Theorem 2.5 \cite[p. 120]{hoermander} (see also Theorem 4.23 \cite{pseudo2}),
from which one can see that the extension of the Theorem 9.7 of \cite{wong} to symbols, whose derivatives have a singularity at $\{0\}$ is possible. Moreover, analyzing line by line of the proof of Theorem 9.7 in \cite{wong}, one can give an estimate of the norm of the operator.

\section{The first application: the strong Feller property}
\label{firstapp}
Let $L=\{L(t):t\ge 0\}$ be a family of \levy processes.
We consider the stochastic differential equations of the form
\DEQSZ\label{eq1sf}
\lk\{\barray
dX^ x(t) &=& b(X^ x(t-))\, dt +\sigma(X^ x(t-)) dL(t)
\\
X^ x(0)&=&x,\quad x\in\RR^ d,
\earray\rk.
\EEQSZ
where $\sigma:\RR^ d\to L(\RR^ d, \RR^d)$ and $b:\RR^ d \to\RR^ d $ are  Lipschitz continuous.
By $C^0_b(\RR^d )$ we denote the set of all real valued and uniformly continuous functions on $\RR^d$ equipped with the supremum--norm.
A Markovian semigroup is Feller, iff
$\CP_tu\in C^0_b(\RR^d)$ for all $u\in  C^0_b(\RR^d)$ for all $t>0$ and
is strongly continuous in zero, i.e. 
$\lim_{t\downarrow 0}|\CP_tu-u|_{C^0_b} = 0$ for every $u\in   C^0_b(\RR^d)$.
The Markovian semigroup $(\CP_t)_{t\ge 0}$ of a process is called strong Feller, iff for all $f\in
 \CB_b(\RR^d)$ and $t>0$, $\CP_tf\in C^0_b(\RR^d)$. In this section we will prove under certain assumptions the strong Feller property of the Markovian semigroup.
Now, we can state our first result.
\begin{thm}\label{sftheo111}
Let $L$ be a square integrable  L\'evy process with Blumenthal--Getoor index $\delta$, $\delta\in(1,2)$, of order $2d+4$.
Let  $\sigma\in \cal^{d+3}_b(\RR^d)$, such that $\sigma$  is bounded away from zero and $ b\in \cal^{d+3}_b(\RR^d)$.
Then, there exists a constant $C>0$ such that
we have for $\gamma\in\RR$, $1\le p< \infty$ and $1 \le q < \infty$
\DEQSZ\label{strongfell}
 \lk| \CP_t u \rk| _{B^ {\gamma}_{p,q}}
\le\frac{C}{t} \lk| u \rk| _{B^ {\gamma-\delta}_{p,q}}
.
\EEQSZ
\del{The assertion follows by the definition of the strong Feller property.}
\end{thm}

The estimate in the above theorem can be used  to prove   the strong Feller property of $(\CP_t)_{t\ge 0}$, which we have done in the following Corollary.
\begin{cor}\label{sftheo1}
Let us assume that
$L$ is a square integrable  L\'evy process with Blumenthal--Getoor index $\delta$, $\delta\in(1,2)$, of order $2d+4$.
Let $\sigma\in \cal^{d+3}_b(\RR^d)$ is bounded away from zero
and $ b\in \cal^{d+3}_b(\RR^d)$.
Then, the process defined by \eqref{eq1sf} is strong Feller. In particular,  we have for all $\gamma\ge 0$ and $n=[\frac \gamma\delta]+1$
we have %
\DEQSZ\label{strongfell}
\lk| \CP_t u \rk| _{\cal^ {\gamma}_b(\RR^d)}
\le \frac {(nC)^n}{{{t}^n}}\, \lk|u\rk| _{L^ {\infty}(\RR^d)},\quad t>0.
\EEQSZ
\end{cor}
Let us define the density $p:[0,\infty)\times \RR^d\times \RR^d\to\RR_0^+$  for the process $X$ by
$$
\PP(X^x(t)\in A)=\int_A p_t(x,y)\, dy ,\quad A\in\CF(\RR^d),\, t>0,\, \mbox{ and }\, x\in\RR^d.
$$
Observe, for any $x,y\in\RR^d$, we have
$$
p_t(x,y)=(\CP_t\delta_x)(y),
$$
By Corollary \ref{sftheo1}, we get also estimates for the density $p$.
\begin{cor}\label{densitytheo1}
Let us assume that
$L$ is a square integrable  L\'evy process with Blumenthal--Getoor index $1<\delta<2$ of order $2d+4$, $\sigma\in \cal^{d+3}_b(\RR^d)$ is bounded away from zero,
and $ b\in \cal^{d+3}_b(\RR^d)$.
Then, the density of the process is $\theta$ times differentiable. In particular, for any $\theta\in\NN$ there exists a number $n=[\frac {\theta+d}\delta]+1$
such that we have for any  $\alpha $ is a multiindex of length $\theta$
$$
\Big|\frac {\partial ^\alpha }{\partial _y^\alpha} \, p_t(x,y) \Big|\le \frac  {C(n,d)} {t^n} ,
$$
\end{cor}

\begin{proof}[Proof of Corollary \ref{sftheo1}:]
Fix  $n\in\NN$ and $p\in[1,\infty)$ such that  $\gamma<n\delta-\frac dp$. Fix $1 \le q < \infty$ arbitrary.
Then, we know, firstly for $\gamma\not\in\NN_0$ (see \cite[p.\ 14]{Runst+Sickel_1996})
$$C_b^{\gamma}(\RR^d)=B^ \gamma_{\infty,\infty}(\RR^d).
$$
 Secondly, we apply the embedding $B^{\gamma+\frac dp}_{p,q}(\RR^d)\hookrightarrow B^ \gamma_{\infty,\infty}(\RR^d)$ (see \cite[Chapter 2.2.3]{Runst+Sickel_1996}, \cite[Section 6.4]{pseudo2}), and, finally,
 we apply Theorem
 \ref{sftheo111} $n$ times to get,
\DEQSZ\nonumber
\lk| \CP_t u \rk| _{\cal^ {\gamma}_b}\le \lk| \CP_t u \rk| _{B^ {\gamma}_{\infty,\infty}}\le \lk| \CP_t u \rk| _{B^ {\gamma+\frac{d}{p}}_{p,q}}=\lk| (\CP_{\frac tn})^n u \rk| _{B^ {\gamma+\frac{d}{p}}_{p,q}}
\le\frac {(nC)^n}{{{t}^n}}\, \lk| u \rk| _{B^ {\gamma+\frac{d}{p}-n\delta}_{p,q}}
.
\EEQSZ
By means of  \cite[Excercise 6.25, Corollary 6.14]{pseudo2},
\DEQSZ\nonumber
L^\infty(\RR^d)\hookrightarrow L^p(\RR^d)\hookrightarrow B_{p,p}^0(\RR^d)\hookrightarrow
B^\kappa_{p,1}(\RR^d), 
\EEQSZ
where $\kappa=\frac{d}{p}-n\delta+\gamma$.
\del{ and $p^\prime=\frac{p}{p-1}$. Now by applying \cite[Lemma 6.5]{pseudo2} and duality property of the Besov spaces gives that
\DEQSZ\nonumber
B^\kappa_{p^\prime,1}(\RR^d)\hookrightarrow L^\infty(\RR^d)\hookrightarrow B^{-\kappa}_{p,\infty}(\RR^d)=B^{\gamma+\frac{d}{p}-n\delta}_{p,\infty}(\RR^d)\hookrightarrow B^{\gamma+\frac{d}{p}-(n-1)\delta}_{p,1}(\RR^d).
\EEQSZ}
By fixing $q=1$ we obtain 
\DEQSZ\nonumber
\lk| \CP_t u \rk| _{\cal^ {\gamma}_b}
\le\frac {(nC)^n}{{{t}^n}}\, \lk| u \rk| _{B^ {\gamma+\frac{d}{p}-n\delta}_{p,1}}\le\frac {(nC)^n}{{{t}^n}}\, \lk| u \rk| _{B^ {\gamma+\frac{d}{p}-n\delta}_{p,1}}\le\frac {(nC)^n}{{{t}^n}}\, \lk| u \rk| _{L^\infty}
.
\EEQSZ
The last line  gives the assertion.

\end{proof}

\begin{proof}[Proof of Corollary \ref{densitytheo1}:]
Fix $p\in(1,\infty)$. We know $\delta_x\in B_{p,\infty}^ {-{\frac d{p'}}}(\RR^d)$,
where $p'$ is the conjugate of $p$ (\cite[Formula B.2]{brzerraz}). Let $\theta\in\NN_0$. A function $u$ is $\theta$ times continuous differentiable, if $u\in C_b^{\theta}(\RR^d)$.
Since $B^{\gamma_1}_{p,q}(\RR^d) \hookrightarrow C_b^{\theta}(\RR^d)$ for $\gamma_1=\theta+\frac dp$, we have to estimate
$|\CP_t \delta_x|_{B^{\gamma_1}_{p,q}}$. Let $n\in\NN$ that large that $n\delta>\theta+d$. Then $\gamma_1-n\delta<-(d-\frac dp)$.
Now, we have
$$
|\CP_t \delta_x|_{B^{\gamma_1}_{p,q}}\le \lk( \frac {C} t\rk)^ n | \delta_x|_{B^{\gamma_1-n\delta}_{p,q}}\le \lk( \frac {C} t\rk)^ n | \delta_x|_{B^{-\gamma_2}_{p,1}},
$$
where $\gamma_2<-(d-\frac dp)$. Since $\delta_x\in B_{p,\infty}^ {-{\frac d{p'}}}(\RR^d)$, the right hand side is bounded.
\end{proof}

\begin{proof}[Proof of Theorem \ref{sftheo111}:]
%
First, note that by Hoh (see \cite{hoh1}), the symbol $\psi$ of the L\'evy process is infinitely often differentiable.
If the coefficient $\sigma$ is independent from the space variable $x$, then one can write the symbol of the semigroup $(\CP_t)_{t\ge 0}$ directly
by $(e^{t\phi(\xi)})_{t\ge 0}$ (see \cite{bahouri})
. If $\sigma$ depends on the space variable $x$, one does not have such a nice representation of the symbol of the semigroup.
We will use the representation of the semigroup $(\CP_t)_{t\ge 0}$ in terms of the contour integrals as
we have already successfully applied in \cite{cox}, \cite{hau} or \cite{engel}. 
Let $\theta'\in (0,\theta)$, $\rho\in (0,\infty)$, and
$$
\Gamma_{\theta'}(\rho,M)= \Gamma^{(1,M)}_{\theta',\rho}+\Gamma^{(2,M)}_{\theta',\rho}+\Gamma^{(3)}_{\theta',\rho},
$$
where $\Gamma^{(1)}_{\theta',\rho}$ and $\Gamma^{(2)}_{\theta',\rho}$ are the rays $r e^{i(\frac{\pi}{2}+\theta')}$ and $r e^{-i(\frac{\pi}{2}+\theta')}$, $\rho\leq r \le M<\infty$, and $\Gamma^{(3)}_{\theta',\rho}= \rho^{-1}e^{i\alpha}$, $\alpha\in[-\frac{\pi}{2}-\theta',\frac{\pi}{2}+\theta']$.
It follows from \cite[Theorem 1.7.7]{Pazy:83} and Fubini's Theorem that for $t>0$ and $v\in B_{p,q}^{\gamma}(\RR^d)$ we have
$$ 
\CP_t u  =\lim_{M\to \infty}
\frac 1{2\pi i}\int_{\Gamma_{\theta'}(\rho,M)} e^{\lambda t}  R(\lambda:a(x,D)) v d\lambda,
$$ 
where $R(\lambda:a(x,D))$ denotes the inverse of $a(x,\lambda,D)=\lambda I+a(x,D)$.
Due to Theorem \ref{main2} and the assumption of Theorem \ref{sftheo111}, we know that $(\CP_t)_{t\ge 0}$ is an analytic semigroup in $B^{\gamma}_{p,q}(\RR^d)$. Therefore,
 for any element $v\in B_{p,q}^{\gamma}(\RR^d)$, the limit exists and is well defined.
 Let $u\in B^{\gamma-\deltaa}_{p,q}(\RR^d)$ and $\{v_n:n\in\NN\}$ be a sequence such that $v_n\in B^{\gamma}_{p,q}(\RR^d)$ and $v_n\to u$ in $B^{\gamma-\deltaa}_{p,q}(\RR^d)$.
\del{In particular, we have by Minkowski inequality
\begin{equation*}
\begin{split}
\lk|\CP_t u \rk|_{B^ {\gamma+ \frac{d}{p}}_{p,q}} &=&\lim_{M\to \infty}
\Big| \frac 1{2\pi i}\int_{\Gamma_{\theta'}(\rho,M)} e^{\lambda t}  R(\lambda:a(x,D)) u d\lambda \Big|_{B^ {\gamma}_{p,q}}
\\
&\le & \frac 1{2\pi i}\int_{\Gamma_{\theta'}(\rho,M)} \Big| e^{\lambda t}  R(\lambda:a(x,D)) u  \Big|_{B^ {\gamma}_{p,q}}d\lambda
.
\end{split}
\end{equation*}}
By a change of variables, we obtain
\DEQS
\lqq{\lim_{M\to \infty}
\lk| \frac 1{2\pi i}\int_{\Gamma_{\theta'}(\rho,M)} e^{\lambda t}  R(\lambda:a(x,D)) v_n d\lambda \rk|_{B^{\gamma}_{p,q}}}
&&
\\
&\le &\lim_{M\to \infty}  \lk|\frac{1}{2\pi it} \int_{\rho}^{M} e^{r e^{-i(\frac{\pi}{2}+\theta^\prime)} }\,  R(\frac{r}{t}e^{-i(\frac{\pi}{2}+\theta^\prime)},a(x,D))\, v_n e^{i(\frac{\pi}{2}+\theta^\prime)} dr \rk|_{B^ {\gamma}_{p,q}}
\\
&+& \lim_{M\to \infty}\lk|\frac{1}{2\pi it} \int_{\rho}^{M} e^{r e^{i(\frac{\pi}{2}+\theta^\prime)} }\, R(\frac{r}{s}e^{i(\frac{\pi}{2}+\theta^\prime)},a(x,D))\,v_n e^{-i(\frac{\pi}{2}+\theta^\prime)} dr\rk|_{B^{\gamma}_{p,q}}\\&+&
 \lk|\frac{1}{2\pi it} \int_{-\frac{\pi}{2}-\theta^\prime}^{\frac{\pi}{2}+\theta^\prime} e^{\rho e^{i\beta} }\, R(\frac{\rho}{s} e^{i\beta},a(x,D)) \,v_n \rho^{-1} e^{i\beta} d\beta\rk| _{B^{\gamma}_{p,q}}.
 \EEQS
 The Minkowski inequality gives
\DEQSZ\label{back01}
\ldots\nonumber
 &\le &
 \frac{1}{2t\pi} \int_{\rho}^{\infty} e^{- r\sin\theta'}\,\lk| R(\frac{r}{t}e^{-i(\frac{\pi}{2}+\theta')},a(x,D))v_n\rk|_{B^{\gamma}_{p,q}}\,dr
 \\\nonumber
&+& \frac{1}{2t\pi} \int_{\rho}^{\infty} e^{- r\sin\theta' }\, \lk| R(\frac{r}{t}e^{i(\frac{\pi}{2}+\theta')},a(x,D))v_n\rk|_{B^{\gamma}_{p,q}}\, dr
\\
&+&\frac{\rho^{-1}}{2t\pi} \int_{-\frac{\pi}{2}-\theta^\prime}^{\frac{\pi}{2}+\theta^\prime}e^{\rho \cos\beta }\, \lk| R(\frac{\rho}{t} e^{i\beta},a(x,D))v_n\rk|_{B^{\gamma}_{p,q}}\, d\beta.
 \EEQSZ
We analyse  the RHS of the estimate above by  analysing the operator $R(\frac{\rho}{t} e^{i\beta},a(x,D))$ by applying Theorem \ref{invert22}. Before doing that, we have to calculate the seminorm of ${\lambda+a(x,\xi)}$ in the space of hypoelliptic operators carefully. In this way, first of all, we require the following estimate.
Similar to  p.\ 11 in \cite{BG1}, we can see that for $\lambda\in\Sigma_{\theta+\frac{\pi}{2}}$,
$$
\langle|\lambda|^{\frac{1}{\delta}}+|\xi|\rangle^\delta\lesssim|\lambda+a(x,\xi)|.
$$
The above result is due to the fact that  $a\in\Hyp_{d+1,0;1,0}^{\delta,\delta}(\RR^d\times\RR^d)$, the idenitity  \eqref{schill}, and since $\sigma$ is bounded away from zero. Therefore, there exists a number $r>0$ such that we know
$$|\lambda+a(x,\xi)|^{-1}\lesssim
\langle|\lambda|^{\frac{1}{\delta}}+|\xi|\rangle^{-\delta}\lesssim\langle|\xi|\rangle^{-\delta},
$$
for all $\xi\in\RR^d$ with $r\le |\xi|$.
In this way  we obtain 
\DEQS
\lk|\frac{ 1}{\lambda+a(x,\xi)}\rk|&\le& \lk|\frac{ 1}{\lambda+\psi(\sigma(x)^T\xi)}\rk|
\le \lk|\frac{ 1}{\lambda+\langle \sigma(x)^T\xi\rangle^{\delta} }\rk|
\le
C(\sigma,\delta)\, \lgxi ^ {-\delta}
.
\EEQS
Let $k\in\{1,\ldots, d\}$. Then
\DEQS
\Big|\partial_{\xi_k} \lk[\frac{1}{\lambda+a(x,\xi)}\rk]\Big|
&=&\Big|\frac{\partial_{\xi_k}a(x,\xi)}{(\lambda+a(x,\xi))^2}\Big|
\le
\lk|\frac {\lgxi^ {\delta-1}}{(\lambda+\lgxi^ {\delta})^2}\rk|\le C(\sigma,\delta)\, \lgxi^ {-\delta-1}.
\EEQS
Next, let $k,l\in\{1,\ldots,d\}$.
Then,
\DEQS
\partial_{\xi_l}\partial_{\xi_k}\lk[  \frac{1}{\lambda+a(x,\xi)}\rk]=
-\frac { \partial^2_{\xi_l\xi_k}a(x,\xi)}{(\lambda+a(x,\xi))^2}+\frac {2 \partial_{\xi_l}a(x,\xi) \partial_{\xi_k}a(x,\xi)}{(\lambda+a(x,\xi))^3}
.
\EEQS
Hence, we have
\DEQS
\Big| \partial_{\xi_l}\partial_{\xi_k}\lk[  \frac{1}{\lambda+a(x,\xi)}\rk]\Big|
&\le & C(\sigma, \delta)
\lk\{ \frac {\lgxi^ {\delta-2}}{(\lambda +\lgxi^ \delta)^2}+ \frac {\lgxi^ {\delta-2}}{(\lambda +\lgxi^ \delta)^2}\rk\}
\le  C(\sigma, \delta)
 {\lgxi^ {-\delta-2}}.
\EEQS
Let $\alpha=(\alpha_1,\cdots,\alpha_k)$ be a multiindex.
By observing the pattern of the above derivative we can identify the general derivative $\partial_\xi^\alpha[\frac{1}{\lambda+a(x,\xi)}]$ and get the following estimate. There exist $C_1,C_2,\cdots  C_{|\alpha|}>0$ depending on $\sigma$ and $\delta$ such that
{\DEQS
\lqq{\lk|\partial_\xi^\alpha\lk[\frac{1}{\lambda+a(x,\xi)}\rk]\rk|\le C_1|\lambda+a|^{-|\alpha|-1}\langle\xi\rangle^{\delta|\alpha|-|\alpha|}}
&&
\\&&
{} +C_2|\lambda+a|^{-|\alpha|}\langle\xi\rangle^{\delta(|\alpha|-1)-|\alpha|}+ C_3|\lambda+a|^{-|\alpha|+1}\langle\xi\rangle^{\delta(|\alpha|-2)-|\alpha|}+\dots+C_{|\alpha|}|\lambda+a|^{-2}\langle\xi\rangle^{\delta-|\alpha|}.
\EEQS
Therefore
\DEQS
\lqq{\lk|\partial_\xi^\alpha\lk[\frac{1}{\lambda+a(x,\xi)}\rk]\rk| \langle\xi\rangle^{-\delta+|\alpha|}    \le C_1|\lambda+a|^{-|\alpha|-1}\langle\xi\rangle^{\delta|\alpha|-\delta}}
&&
\\
&&{}+C_2|\lambda+a|^{-|\alpha|}\langle\xi\rangle^{\delta|\alpha|-2\delta}+ C_3|\lambda+a|^{-|\alpha|+1}\langle\xi\rangle^{\delta|\alpha|-3\delta}+\dots+C_{|\alpha|}|\lambda+a|^{-2}.
\EEQS
Using the fact that there exists some $r>0$ such that we have for all $x\in\RR^d$ with $|\xi|\ge r$
$$|\lambda+a(x,\xi)|^{-1}\lesssim
\langle|\lambda|^{\frac{1}{\delta}}+|\xi|\rangle^{-\delta}\lesssim\langle|\xi|\rangle^{-\delta},
$$
 we obtain 
$$\lk|\partial_\xi^\alpha\lk[\frac{1}{\lambda+a(x,\xi)}\rk]\rk| \langle\xi\rangle^{-\delta+|\alpha| }\le (C_1+C_2+\dots C_{|\alpha|})\langle|\xi|\rangle^{-2\delta}\lesssim \langle|\xi|\rangle^{-2\delta}\le C(\sigma,\delta) R^{-2\delta} .
$$
The last  line shows that    $\lambda+a(x,\xi)\in\Hyp_{d+1,0;1,0}^{\delta,\delta}(\RR^d\times\RR^d)$.

It remains to estimate the norm of the symbol $\lambda+a(x,\xi)$ in $\tilde \CA^{m} _{k_1,k_2;{\rho,\delta}}(\RR^d\times\RR^d)$ with $k_1=2d+4,k_2=d+3$.
Due to the fact that one has to take at least once the derivative with respect to $\xi$,
the constant $\lambda$ has no influence on the norm in $\tilde \CA$.
Since, we have for $a\in \CA ^m_{k_1,k_2;\rho,\delta}(\RR^d\times\RR^d)$
\DEQS
\lqq{ \| \lambda +a\| _{\tilde{\CA}_{k_1,k_2;{\rho,\delta}}^m }
} &&
\\
&=& \sup_{1\le|\alpha|\le k_1,|\beta|\le k_2} \, \sup_{(x,\xi)\in \RR^d
\times \RR^d } \lk| \partial ^ \alpha_x \partial ^ \beta_\xi (\lambda+a(x,\xi)) \rk|
\ggxi  ^ {\rho|\beta|-m}\ggx^{\delta|\alpha|}, \quad
\EEQS
where  $k_1=2d+4,k_2=d+3$ and $\rho=1,\delta=0$, we can conclude that $\lambda+a(x,\xi)\in\CA^{-1}_{2d+4,d+3;1,0}$.
}

\medskip

Going back to \eqref{back01} we can conclude by our discussion before 
\DEQS
\lqq{\lim_{M\to \infty}
\lk| \frac 1{2\pi i}\int_{\Gamma_{\theta'}(\rho,M)} e^{\lambda t}  R(\lambda:a(x,D)) v_n d\lambda \rk|_{B^{\gamma}_{p,q}}}
&&
\\
&\le &
 \frac{C(\sigma,\delta)}{2t\pi} \int_{\rho}^{\infty} e^{- r\sin\theta'}\,|v_n|_{B^{\gamma-\delta}_{p,q}}dr
+ \frac{C(\sigma,\delta)}{2t\pi} \int_{\rho}^{\infty} e^{- r\sin\theta' }\,|v_n|_{B^{\gamma-\delta}_{p,q}}  dr
\\
&&{}+\frac{C(\sigma,\delta)\rho^{-1}}{2t\pi} \int_{-\frac{\pi}{2}-\theta^\prime}^{\frac{\pi}{2}+\theta^\prime}e^{\rho \cos\beta }\,|v_n|_{B^{\gamma-\delta}_{p,q}} \, d\beta
\\
&\le &
\frac{C(\sigma,\delta)}{2t\pi} \,|v_n|_{B^{\gamma-\delta}_{p,q}}.
 \EEQS
 Taking the limit $n\to\infty$, we get
 \DEQS
\lim_{M\to \infty}
\lk| \frac 1{2\pi i}\int_{\Gamma_{\theta'}(\rho,M)} e^{\lambda t}  R(\lambda:a(x,D)) u d\lambda \rk|_{B^{\gamma}_{p,q}}
&\le &
\frac{C(\sigma,\delta)}{2t\pi} \,|u|_{B^{\gamma-\delta}_{p,q}},
 \EEQS
 which is the assertion.

\end{proof}

The following Corollary is a consequence of  Theorem \eqref{sftheo1}.
\begin{cor}\label{smooth1}
Let $L$ be  a  square integrable L\'evy process with Blumenthal--Getoor index $\delta\in(1,2)$ of order $2d+4$. Let  $\sigma\in \cal^{d+3}_b(\RR^d)$ be  bounded away from zero and $ b\in \cal^{d+3}_b(\RR^d)$. Let  $m(D)$ be  a pseudo--differential operator such that $m(\xi)\in S^\kappa_{1,0}(\RR^ d\times \RR^ d)$ with $0\le \kappa\leq 1$.
Then, there exists a constant $C>0$ such that  for any $0<\gamma<\frac{\delta-\kappa}{4}$, $\gamma\not\in\NN$, and $t>0$  we have
\DEQSZ\nonumber
\lk| \CP_t \, m(D)\,  u \rk| _{\cal^ {\gamma}_b(\RR^d)}
\le \frac {C}{{{t}}}\, \lk|u\rk| _{L^ {\infty}(\RR^d)}.
\EEQSZ
\end{cor}
\begin{proof}
\del{Note, first that we have for Sobolev embedding for $p,q\in(1,\infty)$ with $\gamma+\frac dp>0$ 
and $1<q<\infty$
(see \cite[Identities on p.14 and Chapter 2.2.3]{Runst+Sickel_1996})
$$
\lk| \CP_t \, m(D)\,  u \rk| _{\cal^ {\gamma}_b(\RR^d)}
\lesssim \lk|\CP_t m(D)u \rk|_{B^{\gamma+ \frac{d}{p}}_{p,q}}.
$$}
The proof is a combination of the proof of Theorem \ref{sftheo1} and Corollary \ref{sftheo1}.
Due to this reason, we include only the essential steps of the  proof.
 We have already shown that
 $$\lambda+a(x,\xi)\in\Hyp_{d+1,0;1,0}^{\delta,\delta}(\RR^d\times\RR^d)\cap\CA^{-1}_{2d+4,d+3;1,0}.
  $$
  On the other hand it is straightforward  that ${m(\xi)}\in\Hyp_{d+1,0;1,0}^{\delta,\delta}(\RR^d\times\RR^d)\cap\CA^{-1}_{2d+4,d+3;1,0}$. As already observed in the proof of Theorem \ref{sftheo1}, we have the following representation
  of the semigroup 
 \begin{equation*}
\begin{split}
\lk|\CP_t m(D)u \rk|_{B^{\gamma+ \frac{d}{p}}_{p,q}} =\lim_{M\to \infty}
\lk| \frac1{2\pi i}\int_{\Gamma_{\theta'}(\rho,M)} e^{\lambda t}  R(\lambda:a(x,D))m(D) u d\lambda \rk|_{B^{\gamma+ \frac{d}{p}}_{p,q}}.
\end{split}
\end{equation*}
Similarly as in the proof of Theorem \ref{sftheo1}
we can write 
\DEQS
\lqq{\lim_{M\to \infty}
\lk| \frac 1{2\pi i}\int_{\Gamma_{\theta'}(\rho,M)} e^{\lambda t}  R(\lambda:a(x,D))\,m(D)\, u \,d\lambda \rk|_{B^ {\gamma+ \frac{d}{p}}_{p,q}}}
&&
\\
&\le &\lim_{M\to \infty}  \lk|\frac{1}{2\pi it} \int_{\rho}^{M} e^{r e^{-i(\frac{\pi}{2}+\theta^\prime)} }\,  R(\frac{r}{t}e^{-i(\frac{\pi}{2}+\theta^\prime)},a(x,D))\, m(D)\, u\, e^{i(\frac{\pi}{2}+\theta^\prime)} dr \rk|_{B^{\gamma+ \frac{d}{p}}_{p,q}}
\\
&+& \lim_{M\to \infty}\lk|\frac{1}{2\pi it} \int_{\rho}^{M} e^{r e^{i(\frac{\pi}{2}+\theta^\prime)} }\, R(\frac{r}{s}e^{i(\frac{\pi}{2}+\theta^\prime)},a(x,D))\, m(D)\,u\, e^{-i(\frac{\pi}{2}+\theta^\prime)} dr\rk|_{B^{\gamma+ \frac{d}{p}}_{p,q}}\\&+&
 \lk|\frac{1}{2\pi it} \int_{-\frac{\pi}{2}-\theta^\prime}^{\frac{\pi}{2}+\theta^\prime} e^{\rho e^{i\beta} }\, R(\frac{\rho}{s} e^{i\beta},a(x,D)) \, m(D)\,u\, \rho^{-1} e^{i\beta} d\beta\rk| _{B^{\gamma+ \frac{d}{p}}_{p,q}}
 \\
 &\le &
 \frac{1}{2t\pi} \int_{\rho}^{\infty} e^{- r\sin\theta'}\,\lk| R(\frac{r}{t}e^{-i(\frac{\pi}{2}+\theta')},a(x,D))\,  m(D)\,u\rk|_{B^{\gamma+ \frac{d}{p}}_{p,q}}\,dr
 \EEQS
 \DEQS
\lqq{+ \frac{1}{2t\pi} \int_{\rho}^{\infty} e^{- r\sin\theta' }\, \lk| R(\frac{r}{t}e^{i(\frac{\pi}{2}+\theta')},a(x,D))\, m(D)\,u\rk|_{B^{\gamma+ \frac{d}{p}}_{p,q}}\, dr}
&&
\\
&+&\frac{\rho^{-1}}{2t\pi} \int_{-\frac{\pi}{2}-\theta^\prime}^{\frac{\pi}{2}+\theta^\prime}e^{\rho \cos\beta }\, \lk| R(\frac{\rho}{t} e^{i\beta},a(x,D))\, m(D)\,u\rk|_{B^{\gamma+ \frac{d}{p}}_{p,q}}\, d\beta.
 \EEQS
 Note again  that the  semi-norms $$\lk\|\lambda+a\rk\|_{\tiny \Hyp_{d+1,0;1,0}^{\delta,\delta}(\RR^d\times\RR^d)} \quad \mbox{and}\quad \lk\|\lambda+a\rk\|_{\CA^{-1}_{2d+4,d+3;1,0}}$$
  do  not depend on $\lambda$.
In this way, by separating $m(D)$ and $R(\frac{r}{t}e^{i(\frac{\pi}{2}+\theta')},a(x,D))$
and applying Theorem \eqref{invert22} we get
\DEQS
 \lqq{\dots\lesssim\frac{1}{2t\pi} \int_{\rho}^{\infty} e^{- r\sin\theta'}\lk|\, m(D)\,u\rk|_{B^{\gamma+\frac{d}{p}-\delta}_{p,q}}\,dr }
 &&
 \\
 &+& \frac{1}{2t\pi} \int_{\rho}^{\infty} e^{- r\sin\theta^\prime }\lk|\, m(D)\,u\rk|_{B^{\gamma+\frac{d}{p}-\delta}_{p,q}}\,dr
+\frac{\rho^{-1}}{2t\pi} \int_{-\frac{\pi}{2}-\theta^\prime}^{\frac{\pi}{2}+\theta^\prime}e^{\rho \cos\beta }\lk|m(D)\,u\rk|_{B^{\gamma+\frac{d}{p}-\delta}_{p,q}}\,d\beta,
\\
&\lesssim &\lk[\frac{1}{t\pi} \int_{\rho}^{\infty} e^{- r\sin\theta^\prime }\,dr + \frac{1}{2t\pi} \int_{-\frac{\pi}{2}-\theta^\prime}^{\frac{\pi}{2}+\theta^\prime}e^{\rho \cos\beta }\,d\beta\rk] \lk|\, m(D)\,u\rk|_{B^{\gamma-\delta}_{p,q}}\le \frac{C}{t} \lk|\, m(D)\,u\rk|_{B^{\gamma+\frac{d}{p}-\delta}_{p,q}}.
 \EEQS
Since $m(\xi)\in S^\kappa_{1,0}(\RR^ d\times \RR^ d)$
\del{Again applying Theorem \eqref{invert22}
and noting that
${m(\xi)}\in\Hyp_{d+1,0;1,0}^{\delta,\delta}(\RR^d\times\RR^d)\cap\CA^{-1}_{2d+4,d+3;1,0}$,}
we get 
\DEQS
\frac{C}{t} \lk|\, m(D)\,u\rk|_{B^{\gamma+\frac{d}{p}-\delta}_{p,q}}&\le &\frac{C}{t} %
\lk\|m\rk\|_{ S^\kappa_{1,0}}
\lk|u\rk|_{B^{\gamma+\frac{d}{p}-\delta+\kappa}_{p,q}}
\\
&\lesssim & \frac{C}{t}\lk|u\rk|_{B^{\gamma+\frac{d}{p}-\delta+\kappa}_{p,q}}.
\EEQS

\del{On the other hand due to the Sobolev embedding theorem (Theorem 6.15) in \cite{pseudo2}, we have
$$\lk| \CP_t u \rk| _{B^ {\gamma}_{\infty,\infty}(\RR^d)}\le \lk| \CP_t u \rk| _{B^ {\gamma+\frac{d}{p}}_{p,q}(\RR^d)},$$
with $\frac{4d}{\delta-\kappa}< p< \infty$ and $1 <q < \infty$. That is
\DEQSZ\nonumber
\lk| \CP_t \, m(D)\, u \rk| _{\cal^ {\gamma}_b(\RR^d)}=\lk| \CP_t \, m(D)\, u \rk| _{B^ {\gamma}_{\infty,\infty}(\RR^d)}\le \lk| \CP_t \, m(D)\,) u \rk| _{B^ {\gamma+\frac{d}{p}}_{p,q}(\RR^d)} \le\frac{C}{t} \lk| u \rk| _{B^ {\gamma+\frac{d}{p}-\delta+\kappa}_{p,q}(\RR^d)}
\le \frac {C}{{{t}}}\, \lk|u\rk| _{L^ {\infty}(\RR^d)}.
\EEQSZ
}
Now we are following the same argument of the proof of Corollary \ref{sftheo1} to complete the argument.
Fix  $\gamma<(n-1)\delta-\kappa-d$ and let $p\ge 1$ such that $\frac dp<(n-1)\delta-\kappa-\gamma$. Fix $1 \le q < \infty$ arbitrary.
Then, we know, firstly (see \cite[p.\ 14]{Runst+Sickel_1996})
$$\cal^{\gamma}(\RR^d)=B^ \gamma_{\infty,\infty}(\RR^d),\quad \gamma\not\in\NN.
$$
 Secondly, we apply the embedding $B^{\gamma+\frac dp}_{p,q}(\RR^d)\hookrightarrow B^ \gamma_{\infty,\infty}(\RR^d)$ (see \cite[Chapter 2.2.3]{Runst+Sickel_1996}, \cite[Section 6.4]{pseudo2}), and, finally,
 we apply Theorem
 \ref{sftheo111} $n$ times to get,
\DEQSZ\nonumber
\lk| \CP_t \, m(D)\, u \rk| _{\cal^ {\gamma}_b}\le \lk| \CP_t \, m(D)\, u \rk| _{B^ {\gamma}_{\infty,\infty}(\RR^d)}\le \lk| \CP_t \, m(D)\, u \rk| _{B^ {\gamma+\frac{d}{p}}_{p,q}}\\=\lk| (\CP_{\frac tn})^n \, m(D)\, u \rk| _{B^ {\gamma+\frac{d}{p}}_{p,q}}
\le\frac {(nC)^n}{{{t}^n}}\, \lk| u \rk| _{B^ {\gamma+\frac{d}{p}-n\delta+\kappa}_{p,q}}
.
\EEQSZ
By means of  \cite[Excercise 6.25, Corollary 6.14]{pseudo2},
\DEQSZ\nonumber
B^\theta_{p^\prime,1}(\RR^d)\hookrightarrow B^{\frac{d}{p^\prime}}_{p^\prime,1}(\RR^d)\hookrightarrow \cal^0(\RR^d)\hookrightarrow L^\infty(\RR^d),
\EEQSZ
where $\theta=n\delta-\frac{d}{p}-\gamma+\kappa$ and $p^\prime=\frac{p}{p-1}$. Now, applying \cite[Lemma 6.5]{pseudo2} and duality property of the Besov spaces gives that
\DEQSZ\nonumber
B^\theta_{p^\prime,1}(\RR^d)\hookrightarrow L^\infty(\RR^d)\hookrightarrow B^{-\theta}_{p,\infty}(\RR^d)=B^{\gamma+\frac{d}{p}-n\delta+\kappa}_{p,\infty}(\RR^d)\hookrightarrow B^{\gamma+\frac{d}{p}-(n-1)\delta+\kappa}_{p,1}(\RR^d).
\EEQSZ
Finally by fixing $q=1$ we  get
\DEQSZ\nonumber
\lk| \CP_t  \, m(D)\, u \rk| _{\cal^ {\gamma}_b}
\le\frac {(nC)^n}{{{t}^n}}\, \lk| u \rk| _{B^ {\gamma+\frac{d}{p}-n\delta+\kappa}_{p,1}}\le\frac {(nC)^n}{{{t}^n}}\, \lk| u \rk| _{B^ {\gamma+\frac{d}{p}-(n-1)\delta+\kappa}_{p,1}}\le\frac {(nC)^n}{{{t}^n}}\, \lk| u \rk| _{L^\infty}
.
\EEQSZ
This completes the proof.
\del{
Let $\theta'\in (0,\theta)$, $\rho\in (0,\infty)$, and
$$
\Gamma_{\theta'}(\rho,M)= \Gamma^{(1,M)}_{\theta',\rho}+\Gamma^{(2,M)}_{\theta',\rho}+\Gamma^{(3)}_{\theta',\rho},
$$
where $\Gamma^{(1,M)}_{\theta',\rho}$ and $\Gamma^{(2,M)}_{\theta',\rho}$ are the rays $r e^{i(\frac{\pi}{2}+\theta')}$ and $r e^{-i(\frac{\pi}{2}+\theta')}$, $\rho\leq r \le M<\infty$, and $\Gamma^{(3)}_{\theta',\rho}= \rho^{-1}e^{i\alpha}$, $\alpha\in[-\frac{\pi}{2}-\theta',\frac{\pi}{2}+\theta']$.
Let us denote the resolvent, i.e.\ the inverse of $a(x,\lambda,D)=\lambda I+a(x,D)$ by $R(\lambda:a(x,D))$.
Note, that the symbol of $R(\lambda:a(x,D))$ cannot be written in a closed form.
It follows from \cite[Theorem 1.7.7]{Pazy:83} and Fubini's Theorem that for $t>0$ and $u\in L^\infty(\RR^d)$
\begin{equation*}
\begin{split}
\CP_t u  =\lim_{M\to \infty}
\frac 1
{2\pi i}\int_{\Gamma_{\theta'}(\rho,M)} e^{\lambda t}  R(\lambda:a(x,D)) u d\lambda.
\end{split}
\end{equation*}
Denote $\tilde \CP_t=\CP_t m(D)=\text{Op}[\CP(t,x,\xi) m(\xi)]$ and $\tilde R(\lambda:a(x,D))=R(\lambda:a(x,D)) m(D)=\text{Op}[\tilde R(\lambda,x,\xi)]=\text{Op}[R(\lambda,x,\xi) m(\xi)]$.
Then, since $(\CP_t)_{t\ge 0}$ is analytic in $B^{\gamma}_{p,q}(\RR^d)$, the limit exists for all $u\in B_{p,q}^{\gamma}(\RR^d)$  and we have 
\begin{equation*}
\begin{split}
\tilde \CP_t u  =\lim_{M\to\infty} \frac 1{2\pi i}\int_{\Gamma_{\theta'}(\rho,M)} e^{\lambda t}\tilde R(\lambda:a(x,D))u d\lambda
\end{split}
\end{equation*}
In particular, we have
\begin{equation*}
\begin{split}
\lk|\tilde \CP_t u \rk|_{B^\gamma_{p,q}} =\lim_{M\to \infty}
\lk| \frac 1{2\pi i}\int_{\Gamma_{\theta'}(\rho,M)} e^{\lambda t}  \tilde R(\lambda:a(x,D)) u d\lambda \rk|_{B^\gamma_{p,q}}.
\end{split}
\end{equation*}
The Minkowski inequality gives
\begin{equation*}
\begin{split}
\lk|\tilde \CP_t u \rk|_{B^\gamma_{p,q}} =\lim_{M\to \infty}
 \frac 1{2\pi i}\int_{\Gamma_{\theta'}(\rho,M)}  e^{\lambda t}  \lk|\tilde R(\lambda:a(x,D)) u  \rk|_{B^\gamma_{p,q}} d\lambda.
\end{split}
\end{equation*}
Now, in the  next step we have to calculate the norm of $|\tilde R(\lambda:a(x,D)) u|_{L(L^\infty,  B^\gamma_{p,q}(\RR^d)}$ depending on $\lambda$.
Here we use Theorem \ref{invert22} to calculate
\DEQS
|\tilde R(\lambda:a(x,D)) u|_{L(L^\infty,  B^\gamma_{p,q}(\RR^d)}.
\EEQS
In the first step, we calculate the norm of $\lambda+a(x,\xi)$ in $\Hyp_{\rho,\delta}^{0,0}(\RR^ d,\RR^ d) $ in dependence of $\lambda$.
First, note that
for $\alpha=(\alpha_1,\alpha_2,\ldots,\alpha_d)$
$$
\partial^{\alpha}_\xi\frac 1 {\lambda + a(x,\xi,t)}
= \partial^{\alpha}_\xi
$$
\begin{equation*}
\begin{split}
\lk|\partial^{\alpha}_\xi\tilde \CP(x,\xi,t)\rk| \le\ \frac{1}{2t\pi} \int_{\rho}^{\infty} e^{- r\sin\theta'}\,\lk| \partial^{\alpha}_\xi\tilde R(\frac{r}{t}e^{-i(\frac{\pi}{2}+\theta')},x,\xi)\rk|\,dr \\
+ \frac{1}{2t\pi} \int_{\rho}^{\infty} e^{- r\sin\theta' }\, \lk|\partial^{\alpha}_\xi \tilde R(\frac{r}{t}e^{i(\frac{\pi}{2}+\theta')},x,\xi)\rk|\, dr
+\frac{\rho^{-1}}{2t\pi} \int_{-\frac{\pi}{2}-\theta^\prime}^{\frac{\pi}{2}+\theta^\prime}e^{\rho \cos\beta }\, \lk|\partial^{\alpha}_\xi\tilde R(\frac{\rho}{t} e^{i\beta},x,\xi)\rk| \, d\beta
\\\le \frac{1}{t\pi} \int_{\rho}^{\infty} e^{- r\sin\theta'}\sum_{|\alpha|\le d,\varrho\le\alpha}C^\alpha_\varrho\lk|\partial^{\alpha-\varrho}_\xi R(\frac{r}{t}e^{-i(\frac{\pi}{2}+\theta')},x,\xi)\rk|\lk|\partial^{\varrho}_\xi m(\xi)\rk|\,dr\\+\frac{\rho^{-1}}{2t\pi} \int_{-\frac{\pi}{2}-\theta^\prime}^{\frac{\pi}{2}
+\theta^\prime} \sum_{|\alpha|\le d,\varrho\le\alpha}C^\alpha_\varrho\lk|\partial^{\alpha-\varrho}_\xi R(\frac{\rho}{t} e^{i\beta},x,\xi)\rk|\lk|\partial^{\varrho}_\xi m(\xi)\rk|\, d\beta
\\\le \frac{1}{t\pi} \int_{\rho}^{\infty} e^{- r\sin\theta'}\sum_{|\alpha|\le d,\varrho\le\alpha}C^\alpha_\varrho\gr |\xi|+|\frac{r}{t}|^ {\frac 1 \delta} \gl ^{-\delta-|\alpha-\varrho|}\gr |\xi| \gl ^{\kappa-|\varrho|}\,dr\\+\frac{\rho^{-1}}{2t\pi} \int_{-\frac{\pi}{2}-\theta^\prime}^{\frac{\pi}{2}
+\theta^\prime} \sum_{|\alpha|\le d,\beta\le\alpha}C^\alpha_\varrho\gr |\xi|+|\frac{\rho}{t}|^ {\frac 1 \delta} \gl ^{-\delta-|\alpha-\varrho|}\gr |\xi| \gl ^{\kappa-|\varrho|}\, d\beta
\\ \le C_1\frac{\gr |\xi| \gl ^{-\delta-|\alpha|+\kappa}}{t\pi} \int_{\rho}^{\infty} e^{- r\sin\theta'}\,dr+C_2\frac{\gr |\xi| \gl ^{-\delta-|\alpha|+\kappa}\rho^{-1}}{2t\pi} \int_{-\frac{\pi}{2}-\theta^\prime}^{\frac{\pi}{2}+\theta^\prime}e^{\rho \cos\beta }  \, d\beta=\frac{C}{t}\gr |\xi| \gl ^{-\delta-|\alpha|+\kappa}.
\end{split}
\end{equation*}

We would slightly modify the proof of Theorem \eqref{sftheo1} and construct the proof of this corollary.
Let $\theta'\in (0,\theta)$, $\rho\in (0,\infty)$, and
$$
\Gamma_{\theta',\rho}= \Gamma^{(1)}_{\theta',\rho}+\Gamma^{(2)}_{\theta',\rho}+\Gamma^{(3)}_{\theta',\rho},
$$
where $\Gamma^{(1)}_{\theta',\rho}$ and $\Gamma^{(2)}_{\theta',\rho}$ are the rays $r e^{i(\frac{\pi}{2}+\theta')}$ and $r e^{-i(\frac{\pi}{2}+\theta')}$, $\rho\leq r <\infty$, and $\Gamma^{(3)}_{\theta',\rho}= \rho^{-1}e^{i\alpha}$, $\alpha\in[-\frac{\pi}{2}-\theta',\frac{\pi}{2}+\theta']$.
Denote $\tilde \CP_t=\CP_t m(D)=\text{Op}[\CP(t,x,\xi) m(\xi)]$ and $\tilde R(\lambda:a(x,D))=R(\lambda:a(x,D)) m(D)=\text{Op}[\tilde R(\lambda,x,\xi)]=\text{Op}[R(\lambda,x,\xi) m(\xi)]$.
It follows from \cite[Theorem 1.7.7]{Pazy:83} and Fubini's Theorem that for $t>0$ and $u\in L^\infty(\RR^d)$
\begin{equation*}
\begin{split}
\tilde \CP_t u  =\frac 1{2\pi i}\int_{\Gamma_{\theta'}(\rho)} e^{\lambda t}\tilde R(\lambda:a(x,D))u d\lambda
=\frac 1{2\pi i}\int_{\Gamma_{\theta'}(\rho)} e^{\lambda t}\int_{\RR^ d} e^{i \xi x }\tilde R(\lambda,x,\xi) \hat u (\xi)\, d\xi d\lambda\\
=\int_{\RR^ d} e^{i \xi x }\biggl[\tinv{2\pi i}\int_{\Gamma_{\theta'}(\rho)} e^{\lambda t}\tilde R(\lambda,x,\xi) d\lambda\biggr] \hat u (\xi)\, d\xi.
\end{split}
\end{equation*}
Therefore $$\tilde \CP(x,\xi,t)=\tinv{2\pi i}\int_{\Gamma_{\theta'}(\rho)} e^{\lambda t}\tilde R(\lambda,x,\xi) d\lambda.$$
Now we would like to estimate $$\partial^{\alpha}_\xi\tilde \CP(x,\xi,t)m(\xi)=\tinv{2\pi i}\int_{\Gamma_{\theta'}(\rho)} e^{\lambda t} \partial^{\alpha}_\xi\tilde R(\lambda,x,\xi) d\lambda,$$ where $\alpha=(\alpha_1,\alpha_2,\ldots,\alpha_d)$.
First consider,
\begin{equation*}
\begin{split}
\partial^{\alpha}_\xi\tilde\CP(x,\xi,t) =\tinv{2\pi i}\int_{\Gamma_{\theta'}(\rho)} e^{\lambda t} \partial^{\alpha}_\xi \tilde R(\lambda,x,\xi) d\lambda
  = \frac{1}{2\pi it} \int_{\rho}^{\infty} e^{r e^{-i(\frac{\pi}{2}+\theta')} }\, \partial^{\alpha}_\xi \tilde R(\frac{r}{t}e^{-i(\frac{\pi}{2}+\theta')},x,\xi)\,  e^{i(\frac{\pi}{2}+\theta')} dr \\
+ \frac{1}{2\pi it} \int_{\rho}^{\infty} e^{r e^{i(\frac{\pi}{2}+\theta')} }\, \partial^{\alpha}_\xi \tilde R(\frac{r}{s}e^{i(\frac{\pi}{2}+\theta')},x,\xi)\, e^{-i(\frac{\pi}{2}+\theta')} dr
+\frac{1}{2\pi it} \int_{-\frac{\pi}{2}-\theta^\prime}^{\frac{\pi}{2}+\theta^\prime} e^{\rho e^{i\beta} }\, \partial^{\alpha}_\xi \tilde R(\frac{\rho}{s} e^{i\beta},x,\xi) \, \rho^{-1} e^{i\beta} d\beta
\end{split}
\end{equation*}
By taking to the account that
$R(\lambda,x,\xi)\in\Hyp_{1,0;\frac 1\delta}^{-\delta,-\delta}(\RR^d \times \RR^ d\times \Sigma_{ \theta+\frac\pi 2})
$ (see the Definition B.8 in \cite{BG1} for the definition of the symbol class $\Hyp_{\rho,\delta;\gamma}^{m,m_0}(U\times \RR^ d\times\Lambda ) $)
and $m(\xi)\in S^\kappa_{1,0}(\RR^ d\times \RR^ d)$ together with Leibniz's formula , we get,
\begin{equation*}
\begin{split}
\lk|\partial^{\alpha}_\xi\tilde \CP(x,\xi,t)\rk| \le\ \frac{1}{2t\pi} \int_{\rho}^{\infty} e^{- r\sin\theta'}\,\lk| \partial^{\alpha}_\xi\tilde R(\frac{r}{t}e^{-i(\frac{\pi}{2}+\theta')},x,\xi)\rk|\,dr \\
+ \frac{1}{2t\pi} \int_{\rho}^{\infty} e^{- r\sin\theta' }\, \lk|\partial^{\alpha}_\xi \tilde R(\frac{r}{t}e^{i(\frac{\pi}{2}+\theta')},x,\xi)\rk|\, dr
+\frac{\rho^{-1}}{2t\pi} \int_{-\frac{\pi}{2}-\theta^\prime}^{\frac{\pi}{2}+\theta^\prime}e^{\rho \cos\beta }\, \lk|\partial^{\alpha}_\xi\tilde R(\frac{\rho}{t} e^{i\beta},x,\xi)\rk| \, d\beta
\\\le \frac{1}{t\pi} \int_{\rho}^{\infty} e^{- r\sin\theta'}\sum_{|\alpha|\le d,\varrho\le\alpha}C^\alpha_\varrho\lk|\partial^{\alpha-\varrho}_\xi R(\frac{r}{t}e^{-i(\frac{\pi}{2}+\theta')},x,\xi)\rk|\lk|\partial^{\varrho}_\xi m(\xi)\rk|\,dr\\+\frac{\rho^{-1}}{2t\pi} \int_{-\frac{\pi}{2}-\theta^\prime}^{\frac{\pi}{2}
+\theta^\prime} \sum_{|\alpha|\le d,\varrho\le\alpha}C^\alpha_\varrho\lk|\partial^{\alpha-\varrho}_\xi R(\frac{\rho}{t} e^{i\beta},x,\xi)\rk|\lk|\partial^{\varrho}_\xi m(\xi)\rk|\, d\beta
\\\le \frac{1}{t\pi} \int_{\rho}^{\infty} e^{- r\sin\theta'}\sum_{|\alpha|\le d,\varrho\le\alpha}C^\alpha_\varrho\gr |\xi|+|\frac{r}{t}|^ {\frac 1 \delta} \gl ^{-\delta-|\alpha-\varrho|}\gr |\xi| \gl ^{\kappa-|\varrho|}\,dr\\+\frac{\rho^{-1}}{2t\pi} \int_{-\frac{\pi}{2}-\theta^\prime}^{\frac{\pi}{2}
+\theta^\prime} \sum_{|\alpha|\le d,\beta\le\alpha}C^\alpha_\varrho\gr |\xi|+|\frac{\rho}{t}|^ {\frac 1 \delta} \gl ^{-\delta-|\alpha-\varrho|}\gr |\xi| \gl ^{\kappa-|\varrho|}\, d\beta
\\ \le C_1\frac{\gr |\xi| \gl ^{-\delta-|\alpha|+\kappa}}{t\pi} \int_{\rho}^{\infty} e^{- r\sin\theta'}\,dr+C_2\frac{\gr |\xi| \gl ^{-\delta-|\alpha|+\kappa}\rho^{-1}}{2t\pi} \int_{-\frac{\pi}{2}-\theta^\prime}^{\frac{\pi}{2}+\theta^\prime}e^{\rho \cos\beta }  \, d\beta=\frac{C}{t}\gr |\xi| \gl ^{-\delta-|\alpha|+\kappa}.
\end{split}
\end{equation*}
\\

Now let
$$
\tilde k_j(x,z,t) = \lk(\CF^ {-1} \tilde \CP_j(x,\cdot,t)\rk)(z), \quad t\ge 0, x,z\in\RR^d,
$$
where $\tilde \CP_j(x,\xi,t)=\tilde \CP(x,\xi,t)\phi_j(\xi).$
A straightforward computation gives that $\tilde k_j(\,\cdot\,,z\,\cdot\,t)$ is the kernel of the operator $\tilde \CP_j(\,\cdot\,,D,t)=\text{Op}[\tilde \CP_j(.,\xi,t)]$, i.e.\
$$\tilde\CP_j(x,D,t) u =\int_{\RR^d} \tilde k_j(x,x-z,t)u(z)\, dz, \quad t\ge 0, x,z\in\RR^d.
$$
The aim is to calculate the $L^ \infty$--norm of $\tilde\CP_j(x,D,t) u$.
From above we know that for all $u\in L^ \infty(\RR^d)$ 
\DEQSZ\label{heresss1}
\sup_{x\in{\RR^d}} \lk|\tilde \CP_j(x,D,t) u (x)\rk| &\le&  |u|_{L^\infty(\RR^d)}\sup_{x\in{\RR^d}} \int_{\RR^d} \lk| \tilde k_j(x,x-z,t) \rk|\, dz\,
\EEQSZ
If $ \lk| \tilde k_j(x,x-z,t) \rk|  \lesssim z^ {-d-1}$, then the RHS is integrable.
From elementary calculations we know for any multiindex $\alpha$ (see Lemma 5.14 \cite{pseudo2})
\DEQS
(i 2 \pi z) ^ {\alpha} \tilde k_j(x,z,t) = \lk( \CF^ {-1} \partial _\xi^ \alpha\tilde \CP_j(x,\xi,t)\rk)(z) .
\EEQS
Now we get by straightforward calculations for  $t>0$
\DEQS
\lqq{ \int_{\RR^d} \lk| \partial _\xi^ \alpha \tilde\CP_j(x,\xi,t)\rk|\, d\xi }
&&\\
\hspace{1cm}&=& \int_{2^ {j-1} \le |\xi|\le 2^ {j+1}} \lk| \partial _\xi^ \alpha \tilde\CP_j(x,\xi,t)\rk|\, d\xi
\le \int_{2^ {j-1} \le |\xi|\le 2^ {j+1}} \lk| \partial _\xi^ \alpha \tilde\CP(x,\xi,t)\rk|\, d\xi
\\ &\le& \frac{C}{t}\int_{2^ {j-1} \le |\xi|\le 2^ {j+1}} \gr \xi \gl^{-\delta-|\alpha|+\kappa}\, d\xi
\\&\lesssim &\frac{C}{t} 2^ {-(\delta-\kappa+|\alpha|)j}2^{jd}.
\EEQS
\del{
Consider (see section 4.2 in \cite{pseudo2} for more details about the symbol of the composition of operators)
\DEQSZ\label{esti-1}
\lk|\int_{\RR^d} \tilde\CP_j(x,\xi,t)\, d\xi\rk| &=&\lk|\int_{\RR^d}\int_{\RR^d}\int_{\RR^d}\phi_j(\eta) e^{-ix.\eta}e^{ix.\xi}\CP(x,\xi,t)\, d\eta\, dx\, d\xi\rk|
\nonumber
\\ &=& \lk|\int_{\RR^d}\int_{\RR^d}\int_{\RR^d}\phi_1(\lambda\eta) e^{-ix.\eta}e^{ix.\xi}\CP(x,\xi,t)\, d\eta\, dx\, d\xi\rk|,
\EEQSZ
where $\lambda=2^{-j+1}$. By using substitution method with $\lambda\eta=\hat{\eta}$, $\frac{x}{\lambda}=\hat{x}$ and $\lambda\xi=\hat{\xi}$ accordingly, we get
\DEQS
\lk|\int_{\RR^d} \CP_j(x,\xi,t)\, d\xi\rk| &=&\lambda^{-d}\lk|\int_{\RR^d}\int_{\RR^d}\int_{\RR^d}\phi_1(\hat{\eta}) e^{-ix.\frac{\hat{\eta}}{\lambda}}e^{ix.\xi}\CP(x,\xi,t)\, d\hat{\eta}\, dx\, d\xi\rk|
\\ &=&\lambda^{-d}\lambda^{d}\lk|\int_{\RR^d}\int_{\RR^d}\int_{\RR^d}\phi_1(\hat{\eta}) e^{-i\lambda\hat{x}.\frac{\hat{\eta}}{\lambda}}e^{i\lambda\hat{x}.\xi}\CP(\lambda\hat{x},\xi,t)\, d\hat{\eta}\, d\hat{x}\, d\xi\rk|
\\ &=&\lambda^{-d}\lk|\int_{\RR^d}\int_{\RR^d}\int_{\RR^d}\phi_1(\hat{\eta}) e^{-i\hat{x}.\hat{\eta}}e^{i\lambda\hat{x}.\frac{\hat{\xi}}{\lambda}}\CP(\lambda\hat{x},\frac{\hat{\xi}}{\lambda},t)\, d\hat{\eta}\, d\hat{x}\, d\hat{\xi}\rk|
\\ &=&\lambda^{-d}\lk|\int_{\RR^d}\int_{\RR^d}\int_{\RR^d}\phi_1(\hat{\eta}) e^{-i\hat{x}.\hat{\eta}}e^{i\hat{x}.\hat{\xi}}\CP(\lambda\hat{x},\frac{\hat{\xi}}{\lambda},t)\, d\hat{\eta}\, d\hat{x}\, d\hat{\xi}\rk|
\\ &\le& 2^{jd}2^{d}\Leb(\mathbf{B}_2(0)) \sup_{\hat{\eta}\in \mathbf{B}_2(0)}\lk|\phi_1(\hat{\eta}) \rk|\lk|\int_{\RR^d}\int_{\RR^d} e^{-i\hat{x}.(\hat{\eta}-\hat{\xi})}\CP(\lambda\hat{x},\frac{\hat{\xi}}{\lambda},t)\, d\hat{x}\, d\hat{\xi}\rk|,
\EEQS
now again applying substitution method with $\tilde{\xi}=\hat{\eta}-\hat{\xi}$
\DEQSZ\label{esti-2}
\ldots&=&2^{jd}2^{d}\Leb(\mathbf{B}_2(0)) \sup_{\hat{\eta}\in \mathbf{B}_2(0)}\lk|\phi_1(\hat{\eta}) \rk|\lk|\int_{\RR^d}\int_{\RR^d} e^{-i\hat{x}.\tilde{\xi}}\CP(\lambda\hat{x},\frac{\hat{\eta}-\tilde{\xi}}{\lambda},t)\, d\hat{x}\, d\tilde{\xi}\rk|
\\ &=& C2^{jd}\sup_{\hat{\eta}\in \mathbf{B}_2(0)}\lk|\phi_1(\hat{\eta}) \rk|\lk|\int_{\RR^d}\int_{\RR^d} e^{-i\hat{x}.\tilde{\xi}}\CP(\lambda\hat{x},\frac{\hat{\eta}-\tilde{\xi}}{\lambda},t)\, d\hat{x}\, d\tilde{\xi}\rk|
\nonumber
\\ &=& C2^{jd}\lk|\text{Os}-\int_{\RR^d}\int_{\RR^d} e^{-i\hat{x}.\tilde{\xi}}\CP(\lambda\hat{x},\frac{\hat{\eta}-\tilde{\xi}}{\lambda},t)\, d\hat{x}\, d\tilde{\xi}\rk|
\nonumber
\\ &\le& C2^{jd}\|\CP\|_{\mathcal{A}_{0,d}^0},
\EEQSZ
where the notation $\text{Os}-$ represent the corresponding oscillatory integral (see theorem 4.9 in \cite{pseudo2}) and the norm $\|.\|_{\mathcal{A}_{\tau,d}^m}$ is the norm defined for the space of amplitudes $\mathcal{A}_{\tau}^m(\RR^d\times\RR^d)$ (see definition 4.8 in \cite{pseudo2}).
}
Hence, 
\DEQS
\lk|\tilde k_j(x,z,t)\rk| & \le& \,| z| ^ {-|\alpha|} \int_{\RR^d}  \lk| \partial _\xi^ \alpha \tilde\CP_j(x,\xi,t)e  ^{2\pi i\xi x} \rk|\, d\xi
\\ &\le
& \,| z| ^ {-|\alpha|} \int_{\RR^d}  \lk| \partial _\xi^ \alpha \tilde\CP_j(x,\xi,t) \rk|\, d\xi\le   \frac{C}{t} 2^ {-(\delta-\kappa+|\alpha|)j}2^{jd}\, | z| ^ {-|\alpha|}.
\EEQS

Going back to \eqref{heresss}, we get for $u\in L^ \infty(\RR^d)$ and $|\alpha|=(d+1)>d$,
\DEQSZ
\sup_{x\in{\RR^d}} \lk| \tilde\CP_j(x,D,t) u (x)\rk| &\le&  |u|_{L^\infty(\RR^d)}\sup_{x\in{\RR^d}} \int_{\RR^d} \lk| \tilde k_j(x,x-z,t) \rk|\, dz\,
\nonumber
\\&\le &|u|_{L^\infty(\RR^d)}\frac{C}{t} 2^ {-(\delta-\kappa+d+1)j}2^{jd}\sup_{x\in{\RR^d}} \int_{\RR^d}| z| ^ {-|\alpha|}\, dz
\nonumber
\\&\lesssim &\frac{C}{t} 2^ {-(\delta-\kappa+1)j}|u|_{L^\infty(\RR^d)}.
\EEQSZ
Now by combining the Remark \eqref{opbou} together with Lemma 6.22 in \cite{pseudo2},
\DEQSZ\label{twooperators1}
\lk\|\phi_k(D) \tilde\CP_j(x,D,t) \rk\|_{L(L^\infty,L^\infty)}\lesssim  \frac{K_{d}}{t}2^ {\min(k,j)(-(\delta-\kappa)) }2^ {-|j-k|l },
\EEQSZ
for all $j,k\in \mathbb{N}_0$ and $l\in \mathbb{N}$. In other words,
\DEQSZ\label{twooperators21}
\sup_{x\in{\RR^d}} \lk| \phi_k(D)\tilde\CP_j(x,D,t) u (x)\rk|\lesssim  \frac{K_{d}}{t}2^ {\min(k,j)(-(\delta-\kappa)) }2^ {-|j-k|l }|u|_{L^\infty(\RR^d)}.
\EEQSZ
We know that $$u=\sum_{j=0}^\infty u_j,$$ where $u_j=\phi_j(D)u$ and $u\in L^\infty(\RR^d)$. On the other hand we know that $$\tilde\CP(x,D,t)u=\sum_{j=0}^\infty\tilde\CP(x,D,t)\phi_j(D)u=\sum_{j=0}^\infty\tilde\CP_j(x,D,t)u,$$
(see result (5.15) in Remark 5.13, \cite{pseudo2}). Therefore by using the fact that $\supp \phi_j \cap \supp \phi_k=\emptyset$ for $|j-k|>1$,
$$\tilde\CP(x,D,t)u=\sum_{j=0}^\infty\tilde\CP_j(x,D,t)u=\sum_{j=0}^\infty\tilde\CP_j(x,D,t)\hat{u}_k,$$
where $\hat{u}_k=u_{k-1}+u_{k}+u_{k+1}$ and we have set $u_{-1}\equiv0$ (see proof of Theorem 6.19 in \cite{pseudo2}). Now using $\eqref{twooperators21}$,
$$2^{\gamma k}\lk| \phi_k(D)\tilde\CP_j(x,D,t) u\rk|_{L^\infty(\RR^d)}\lesssim  \frac{K_{d}}{t}2^{\gamma k}2^ {\min(k,j)(-(\delta-\kappa))  }2^ {-|j-k|l }|u|_{L^\infty(\RR^d)},$$ where
$$
2^ {\gamma k+\min(k,j)(-(\delta-\kappa)) -|j-k|l } = \lk\{ \begin{array}{rcl} 2^ {\gamma k-j(\delta-\kappa)-|j-k|l} <2^{-(l-\delta+\kappa)|j-k|},&\If & k> j,\\
2^ {-(\delta-\kappa-\gamma)k-|j-k|l},&\If & k\le j.
\end{array}\rk.
$$
Now the assertion follows by the definition of the norm
$B_{\infty,\infty}^{\gamma}(\RR^d)$. In particular for $l>3$, we have  for $u\in L^\infty(\RR^d)\cap B_{\infty,\infty}^{\gamma}(\RR^d)$ such that
\DEQS
\lk|  \tilde\CP_t u\rk|_{B^{\gamma}_{\infty,\infty}(\RR^d)} &=& \sup_{k\in\NN}  2 ^{\gamma k}\lk| \phi_k(D)\tilde\CP(x,D,t) u\rk|_{L^\infty(\RR^d)}
\\
\le  \sup_{k\in\NN}  2 ^{\gamma k}\sum_{j=0}^\infty\lk\|\phi_k(D)\tilde \CP_j(x,D,t) \rk\|_{L(L^\infty,L^\infty)}|\hat{u}_k|_{L^\infty(\RR^d)}
&\le& \frac{K_{d}}{t}\sup_{k\in\NN}  2 ^{\gamma k}\sum_{j=0}^\infty 2^{-|j-k|}  |u|_{L^\infty(\RR^d)}
\\
 = \frac{K_{d}}{t}|u|_{L^\infty(\RR^d)}\le \frac{K_{d}}{t}|u|_ {B_{\infty,\infty}^{\gamma}(\RR^d)}
.\EEQS
}

\end{proof}

If $L$ is an $\alpha $ stable process, the problem appears that only the moments up to $p<\alpha$ are bounded. Therefore,
the symbol is in higher dimensions not uniformly differentiable up to order $d+1$ in any neighborhood of $\xi=0$. However,
if $\alpha>0$, then this problem can be solved.

\begin{cor}\label{sfcor2}
Let $L$ be  a  L\'evy process   $L$ with Blumenthal--Getoor index $1<\delta<2$ of order $2d+4$. Let  $\sigma\in C^{d+3}_b(\RR^d)$  be  bounded away from zero and $ b\in C^{d+3}_b(\RR^d)$.
Then, the Markovian semigroup $(\CP_t)_{t\ge 0}$ of the process defined by \eqref{eq1sf} is strong Feller.
\end{cor}
\del{Similarly, we get also estimates for the density.
\begin{cor}\label{density}
Let us assume that
$L$ is a  L\'evy process with Blumenthal--Getoor index $1<\delta<2$ of order $2d+4$, $\sigma\in C^{d+3}_b(\RR^d)$ is bounded away from zero,
and $ b\in C^{d+3}_b(\RR^d)$.
Then, the density of the process is $n$ times differentiable. In particular, we have
$$
\Big|\frac {\partial ^\alpha }{\partial _y^\alpha} \, p_t(x,y) \Big|\le \frac  {C(n,d)} {t^n} ,
$$
where $\alpha $ is a multiindex of length $\theta$ and $n=[\frac {\theta+d}\delta]+1$.
\end{cor}
\begin{proof}
Note, that we can write $p_t(x,y)=\la \CP_t\delta_x,\delta_y\ra$.
We know (see \cite[p. ]{runst}) that $\delta_x\in B_{r,r}^{-d(1-\frac 1r)}(\RR^d)$ for any $r\in(1,\infty)$. Hence,
Therefore, we have to give an estimate of
$$
|\CP_t\delta_x|_{B_{q,q}^{|\alpha|+d(1-\frac 1p)}}
$$
where $\frac 1p+\frac 1q\le1$. Since $d(1-\frac 1q)+d(1-\frac 1p)+\theta=d+\theta$ and the Blumenthal-Getoor index is $\delta$, the assertion is shown. 

\end{proof}}
\del{

Since we need it for the proof of the corollaries, we list here two important properties.
\begin{rem}\label{inthoelder}(compare \cite[p.\ 201, Section 2.7.2, Theorem1]{triebelint})
For $t_1,t_2\in\NN$ and $t_1<t_2$ we have
$$
\lk( C_b^{t_1}(\RR^d ), C_b^{t_2}(\RR^d )\rk)_{\theta,\infty} = C_b^{t}(\RR^d ),
$$
where $t=(1-\theta)t_1+\theta t_2$.
\end{rem}
\begin{rem}\label{intungl}(compare \cite[p.\ 25, Section 1.3.1, Theorem 1 -(g)]{triebelint})
For any $\theta\in (0,1)$, $1<q\le \infty$ and couple of interpolation spaces $A_1$ and $A_2$, there exists a constant $c=c(1)>0$ such that
$$
\lk| a\rk|_{(A_1,A_2)_{\theta,q}}\le c |a|_{A_1}^\theta |a|_{A_2}^{1-\theta}.
$$
\end{rem}
\begin{rem}
One has to use the Besov spaces, since $C_b^m(\RR^d)$ does not have the weak multiplier property (see \cite[Chapter 2.2.3, p. 40, Remark 1 and Remark 2]{triebel1983}).
\end{rem}
}
\medskip
\del{
\begin{proof}[Proof of Corollary \ref{sfcor1}]
Due to the assumption on the symbol we know $a(x,\xi):=\psi(\sigma(x)^ T \xi)\in S_{...}^ {...}(\RR^d\times\RR^d)$.
Moreover, due to Theorem \ref{main2} we know for all $\delta\in\RR$  $(\CP_t)_{t\ge 0}$ is an analytic semigroup on $B_{\infty,\infty}^ \delta(\RR^d)$.
Hence, we have
\DEQS
\lk| a(x,D) \CP_t \phi\rk|_{B_{\infty,\infty}^\delta }\le \frac 1t \lk| \phi\rk|_{B_{\infty,\infty}^\delta }.
\EEQS
On the other side, straightforward calculations show that
$$
\frac { a(x,\xi)}{\la \xi\ra ^\alpha } \in S_{d+1,1}^{0,0}(\RR^d,\RR ^d ).
$$
Hence
\DEQS
&\lk| (I-\la D\ra ^\alpha) \CP_t \phi\rk|_{B_{\infty,\infty}^\delta }=\lk|  (I-\la D\ra ^\alpha)a(x,D)^{-1} \, a(x,D) \CP_t \phi\rk|_{B_{\infty,\infty}^\delta }
\\
\le& \lk\|  (I-\la D\ra ^\alpha)a(x,D)^{-1} \rk\|_{L(B_{\infty,\infty}^\delta,B_{\infty,\infty}^\delta)}
\, \lk| a(x,D)   \CP_t \phi\rk|_{B_{\infty,\infty}^\delta }
 \le
\frac 1t \lk| \phi\rk|_{B_{\infty,\infty}^\delta }.
\EEQS
Interpolation gives
\DEQS
\lk| (I-\la D\ra ^\alpha)^\gamma \CP_t \phi\rk|_{B_{\infty,\infty}^\delta }
& \le &
\frac 1{t^\gamma} \lk| \phi\rk|_{B_{\infty,\infty}^\delta }.
\EEQS
In addition, since $B_{\infty,\infty}^\delta (\RR^d)$ has the weak multiplier property and $\psi$ has generalized Blumenthal--Getoor index $\alpha$, we have
\DEQS
\lk| (I-\la D\ra ^\gamma) \CP_t \phi\rk|_{B_{\infty,\infty}^\delta }
& \le &
\frac 1{t^\frac \gamma\alpha } \lk| \phi\rk|_{B_{\infty,\infty}^\delta }.
\EEQS
Finally, note that due to the equivalence of the norms $B_{\infty,\infty}^s(\RR^d)$ and $C_b^{\delta}(\RR^d)$ (see \cite[p.\ 14, Proposition]{Runst+Sickel_1996})
and the continuous embedding of $L^\infty(\RR^d)$ in $B_{\infty,\infty}^{-\gamma }(\RR^d)$,
we know that
\DEQS
\lk| \CP_t \phi\rk|_{\cal^ 0_b}\le \lk| \CP_t \phi\rk|_{B_{\infty,\infty}^\gamma }
\le \frac C {t ^{2\gamma} } \lk|\phi\rk|_{B_{\infty,\infty}^{-\gamma }}\le \frac C {t ^{2\gamma} } \lk|\phi\rk|_{L^\infty},
\EEQS
which gives the assertion.
\end{proof}
}
\begin{proof}[Proof of Corollary \ref{sfcor2}]

In order to deal with the large jumps we decompose the \levy process into one \levy process recollecting the jumps smaller than one and one \levy process, recollecting the jumps larger than one.
 Therefore, let $\nu_0$ be the L\' evy measure defined by
$$
\nu_0:\CB(\RR^d)\ni U \mapsto \nu(U\cap Z_0),
$$
and
$\nu_1$ be the L\' evy measure defined by
$$
\nu_1:\CB(\RR^d)\ni U \mapsto \nu(U\cap Z_1),
$$
where $Z_0=\{ z\in \RR^ d : |z|\le 1\}$ and $Z_1=\{ z\in \RR^ d : |z|> 1\}$.
Since the proof of this theorem mainly rely on the analysis of the decomposition of the small and large jumps it is important to decompose also the probability space $\mathfrak{A}=(\Omega,\CF,\{\CF_t\}_{t\in [0,T]},\PP)$. Let  $\tilde \eta_0$ be a compensated Poisson random measure on $(Z_0\times \RR_+,\mathcal{B}(Z_0)\otimes \mathcal{B}({\mathbb{R}_+}))$ over
$\mathfrak{A} ^0=(\Omega^0,\CF^0,\{\CF^0_t\}_{t\in [0,T]},\PP^0)$ with intensity measure $\nu_0$ where
$$\CF^0=\sigma \{ \eta(B ,[0,s]): B \in \CB(Z_0), s\in[0,T] \}
$$
and for $ 0\leq t\leq T$
$$\CF^0_t=\sigma \{ \eta(B ,[0,s]): B \in \CB(Z_0), s\in[0,t] \}.
$$
Furthermore, let $\tilde\eta_1$ be a compensated Poisson random measure on $(Z_1\times \RR_+,\mathcal{B}(Z_1)\otimes \mathcal{B}({\mathbb{R}_+}))$ over $\mathfrak{A} ^1=(\Omega^1,\CF^1,\{\CF^1_t\}_{t\in [0,T]},\PP^1)$ with finite intensity measure $\nu_1$ where
$$\CF^1=\sigma \{ \eta(B ,[0,s]): B \in \CB(Z_1), s\in[0,T] \}$$
and for $ 0\leq t\leq T$
$$
\CF^1_t=\sigma \{ \eta(B ,[0,s]): B \in \CB( Z_1), s\in[0,t] \}.
$$
Let $\Omega:=(\Omega^0\times\Omega^1)$, $\mathcal{F}:=\CF^0\otimes \CF^1$, $\CF_t:=\CF^0_t\otimes \CF^1_t$, $P=P^0\otimes P^1$ and $\EE=\EE^0\otimes \EE^1$.
We denote the to $\nu_0$ and $\nu_1$ associated Levy processes by $L_  0 $ and $L_ 1 $.
It is clear by the independent scattered property of a Poisson random measure, that  $L_ 0 $ and $L_1 $ are independent.
Since $\nu_1$ is a finite  measure, $L_1$ can be represented as a sum
of its jumps. In particular,
let $\rho=\nu_1(\RR^d )$, $\{ \tau_n:n\in\NN\}$ be a family of independent exponential distributed random variables with parameter $\rho$,
\DEQSZ\label{stoppingtimes}
 T_n=\sum_{j=1}^n \tau_j,\quad n\in\NN,
\EEQSZ
and   $\{ N(t):t\ge 0\}$ be the counting process defined by
$$ N(t) :=\sum_{j=1}^\infty 1_{[T_j,\infty)}(t),\quad t\ge 0.
$$
Observe, for any $t>0$, $N(t)$ is a Poisson distributed random variable with parameter $\rho t$.
Let $\{ Y_n:n\in \NN\}$ be a family of independent, $\nu_0/\rho$ distributed random variables. Then
the \levy process $L_1$ given by (see \cite[Chapter 3]{cont})
$$
L_ 1(t)=\int_0^t \int_{Z_1} z\,\tilde \eta_1(dz,ds),\quad t\ge 0,
$$
can be represented as
$$
L_ 1(t) = \bcase - z_0 t & \mbox{ for } N(t)=0,\\
\sum_{j=1}^{N(t)} Y_j -z_0 t & \mbox{ for } N(t)>0,
\ecase
$$
where $z_0= \int_{\RR^ d } z\, \nu_1(dz)$.
Let $(\CP_t^ 0)$ the Markovian semigroup of the solution process $X_0^x$ given by
\DEQSZ\label{eq1sfp0}
\lk\{\barray
dX_0^ x(t) &=& b(X_0(t-))\, dt +\sigma(X_0(t-)) dL_0(t)
\\
X_0^ x(0)&=&x,\quad x\in\RR^ d.
\earray\rk.
\EEQSZ
Now, we have for $u(t)=\EE \phi( X(t))$, where $X$ is the solution to the original equation \eqref{eq1sf} and the following identity
\DEQSZ\label{toooshow}
u(t)  = \EE(\CP^ 0_t\phi)(x)+\EE \sum_{i=1}^ {N(t)} \CP_{t-T_i}^ 0  B_{Y_i} u(T^-_i) -
\int_0^ t \CP^ 0_{t-s} D u(s) [ z_0]\, ds ,
\EEQSZ
where $(B_{y} \phi)(x)= \phi ( x+y)-\phi(x)$.
To verify formula \eqref{toooshow}, observe that in the time interval $[0,T_1)$ the solution of $u$ is given by
\DEQS
 u(t)  &=& (\CP^ 0_t\phi)(x)+\int_0^ t \CP^ 0_{t-s} D u(s) [ z_0]\, ds , \quad t\in[0,T_1].
\EEQS
In particular, $u$ solves on the time interval  $[0,T_1)$
\DEQSZ\label{erstes}
\lk\{\barray
\dot u(t) & = & 
a(x,D) u(t) +D u(t) z_0\, 
, \quad t\in[0,T_1),
\\
u(0) &=&  \phi.
\earray\rk.
\EEQSZ
Let us denote the solution of \eqref{erstes} on the first time interval $[0,T_1]$ by $u_1$.
At time $T_1$ the first large jump occurred. Hence, on the time interval $[T_1,T_2)$, $u$ solves
\DEQSZ\label{zweittes}
\lk\{\barray
\dot u(t) & = & 
a(x,D) u(t) + Du(t) z_0\, 
 , \quad t\in(T_1,T_2),
\\
u(T_1) &=& \EE u_1(T_1,\cdot + Y_1).
\earray\rk.
\EEQSZ
Let us denote the solution of \eqref{zweittes} by $u_2$.
The variation of constant formula gives for $t\in (T_1,T_2)$ 
$$
u_2(t)= \CP^0_{t-T_1} u_2(T_1)+ \int_{T_1}^t \CP^0_{t-s} D u_2(s)\, z_0\, ds.
$$
Let us put
$$
u(t):= \bcase u_1(t), & \mbox{ if } t\in [0,T_1),\\ u_2(t), & \mbox{ if } t\in [T_1,T_2).
\ecase
$$
Since
$$
u_1(T_1)= \CP^0_{T_1} \phi+ \int_{0}^{T_1} \CP^0_{T_1-s} D u(s)\, z_0\, ds,
$$
and $u_2(T_1,x)=u_1(T_1, x +Y_1)$, $x\in\RR^d$,
it follows
\DEQS
u(t) &=& \CP^0_{t-T_1 }\CP^0_{T_1}\phi +  \CP^0_{t-T_1 }\int_0^{T_1} \CP^0_{T_1-s} D 
 u(s)\, z_0\, ds
\\ &&{} +  \int_{T_1}^t\CP^0_{t-s} D 
 u(s)\, z_0\, ds
+ \EE\CP^0_{t-T_1 } u(T_1^-, \cdot + Y_1) - \EE\CP^0_{t-T_1 } u(T_1^-, \cdot ) .
\\
&=& \CP^0_{t}\phi + \int_{0}^t \CP^0_{t-s} D 
 u(s)\, z_0\, ds +  \EE\CP^0_{t-T_1 }\lk[ u(T_1^-, \cdot + Y_1) - u(T_1^-, \cdot )\rk]
.
 \EEQS
Repeating  this calculations successively for all time intervals gives formula \eqref{toooshow}.

\medskip

Since, given $N(t)=k$, the random variables $\{ Y_1,Y_2,\ldots, Y_{k}\}$ and
$\{ T_1,T_2,\ldots, T_{k}\}$  are mutually independent and $T_i$, $i=1,\ldots,k$, is uniform distributed on $[0,t]$, it follows that
\DEQS
\EE  \sum_{i=1}^ {N(t)} \CP_{t-s}^ 0  B_{Y_i} u(T_i) &=&  \sum_{k=1}^\infty \PP\lk( N(t)=k\rk)\,
            \EE\lk[\sum_{l=1}^ {N(t)} \EE^1\lk[  \CP_{t-T_l}^ 0  B_{Y_l} u(T_l)\mid N(t)=k\rk]\rk]
            \\
            \lqq{  =     \sum_{k=1}^\infty \PP\lk( N(t)=k\rk)\,
             \EE^1\lk[  \sum_{l=1}^{k} \CP_{t-T_l}^ 0  B_{Y_l} u(T_l)\rk]}
            \\
            \lqq{  =     \sum_{k=1}^\infty \PP\lk( N(t)=k\rk)\,
            \EE^1\lk[   \int_0^ t \CP_{t-s}^ 0  B_{Y} u(s)\rk],}&&
            \EEQS
            where $Y$ is distributed as $\nu_0/\rho$.
Thus, we get for $\gamma\le \delta-1$
\DEQS
|u(t)|_{B^ \gamma _{\infty,\infty}}  = \lk|\CP^ 0_t\phi+C\int_0^ t  \CP_{t-s}^ 0  \CB_{y} u(s)\, ds +
\int_0^ t \CP^ 0_{t-s} D u(s)[ z_0]\, ds\rk|_{B^ \gamma _{\infty,\infty}},
\EEQS
where $(\CB_{y} \phi)(x)= \int_{\RR^d} [\phi ( x+y)-\phi(x)]\, \nu_0(dy)$ and $C$ is a constant depending  on $\rho$ and $\nu_0$.
By applying Minkowski inequality and 
Theorem \ref{sftheo1}, resp.\  Corollary \ref{smooth1}
with $m(D)=D$ and $m(D)=\CB_{y}$, gives for some $p>1$ with $n>\frac{\gamma+1+d}{\delta}+1$
\DEQS
|u(t)|_{B^ \gamma _{\infty,\infty}} \le \frac 1{t} |\phi|_{B^0 _{\infty,\infty}}+\int_0^ t  \lk| \CP^ 0_{t-s}\CB_{y} u(s) \rk|_{B^\gamma _{\infty,\infty}}\, ds +
|z_0|\int_0^ t    \lk| \CP^ 0_{t-s}  Du(s) \rk|_{B^ \gamma _{\infty,\infty}}\, ds
\\\le \frac 1{t} |\phi|_{B^0 _{\infty,\infty}}+K_{1}\int_0^ t (t-s)^{-n} \lk| u(s) \rk|_{B^\gamma _{\infty,\infty}}\, ds +
K_{2}|z_0|\int_0^ t   (t-s)^{-n} \lk|u(s) \rk|_{B^ \gamma _{\infty,\infty}}\, ds
\\\le \frac 1{t} |\phi|_{B^0 _{\infty,\infty}}+K(1+|z_0|)\lk[\int_0^ t (t-s)^{-p}\,ds\rk]^{\frac{1}{p}}
\lk[\int_0^ t \lk| u(s) \rk|^{\frac{p}{p-1}}_{B^\gamma _{\infty,\infty}}\, ds \rk]^{\frac{p-1}{p}}
\\\le \frac 1{t} |\phi|_{B^0 _{\infty,\infty}}+\frac{C_1}{t^{\frac{np-1}{p}}}
\lk[\int_0^ t \lk| u(s) \rk|^{\frac{p}{p-1}}_{B^\gamma _{\infty,\infty}}\, ds \rk]^{\frac{p-1}{p}}.
\EEQS
Rearranging gives 
$$|u(t)|^{\frac{p}{p-1}}_{B^ \gamma _{\infty,\infty}}\le\frac 1{t^{\frac{p}{p-1}}} |\phi|^{\frac{p}{p-1}}_{B^0 _{\infty,\infty}}+ \frac{C_2}{t^{\frac{np-1}{p-1}}}
\,\int_0^ t \lk| u(s) \rk|^{\frac{p}{p-1}}_{B^\gamma _{\infty,\infty}}\, ds .
$$
\del{
where $\gamma =\frac \delta\alpha$. Since $\delta <\alpha$, there exists a $\gamma'>0$ such that
\DEQS
|u(t)|^ q_{B^ \delta _{\infty,\infty}}  = \frac 1{t^{q \delta}} |\phi|^q_{B^0 _{\infty,\infty}}+t ^ {\gamma'}\int_0^ t   \lk| \CB_{y} u(s) \rk|^ q_{B^0 _{\infty,\infty}}\, ds +
t ^ {\gamma'}\int_0^ t   \lk|  u(s) z_0\rk|^ q_{B^0 _{\infty,\infty}}\, ds.
\EEQS
Observe, $|(\CB_{y} \phi)(x)| \le |\int_0^y \phi'(x+r)\, dr|$.
}
A simple application of  Gr\"onwall's Lemma gives
$$
|u(t)|_{B^ \gamma _{\infty,\infty}}\le C(t,p,n) |\phi|_{B^0 _{\infty,\infty}}.
$$
By the definition of $B^0 _{\infty,\infty}(\RR^d)$, it follows that  the process is strong Feller.

\end{proof}

\section{The second application: Error Estimates for Monte-Carlo Simulation}
\label{secondapp}
Given the intensity measure, in most of the cases, one does not know the distribution of $L(t)$ for a fixed time point $t\ge 0$.
However, simulating stochastic differential equations driven by  a  \levy process using the explicit or implicit Euler-Marayuama scheme, one has to simulate the increments $\Delta_\tau^nL:=L(n\tau)-L((n-1)\tau)$ for $\tau>0$ small.
Here, one can apply several strategies to simulate the  random variables $\Delta_\tau^nL$, $n\in\NN$,
to generate a so-called \levy walk given by
%
$$
L:[0,\infty) \ni t\mapsto L(t):= \int_0 ^ t \int_{\RR^d} \zeta\; \tilde{\eta}({\rm d}\zeta, {\rm d}s)
$$
\del{For a given time-step size $\tau>0$, it is a sequence of random variables
\DEQSZ\label{levy-walk}
\Bigl( \Delta_\tau  ^ 0L,  \Delta_\tau ^ 1 L , \Delta_\tau ^ 2 L ,\,\ldots,\,  \Delta_\tau  ^k L ,\,\ldots \,\Bigr)\,,
\EEQSZ
where
$$
 \Delta_\tau ^ k L :=
 L(t_{k+1}) - L(t_{k})
\qquad\forall\, k\in\NN\, .
$$
In general, the distribution of $\Delta_\tau ^ k L  $ is not known, such that
$\Delta _\tau ^ k L$ cannot be simulated directly.}
One way is to cut off the jumps being smaller than a given $\ep$ and to simulate the corresponding compound Poisson process directly.
Now, one has two possibilities, to neglect the small jumps or to replace the small jumps by a Gaussian random variable.
Doing so, one gets a new L\'evy process, denoted by $\hat L_\ep$. This method was introduced by Tsuchiya \cite{tsu}. Asmussen and Rosinki \cite{assmusen}
investigated  the process generated only by the small jumps and characterised the .

To be more precise, 
let us cut off all jumps being in the unit ball with radius $\ep$, denoted in the following by $B_\ep$.
I.e., let $B_\ep=\{ z\in \RR^d: |z|\le\ep\}$, where $|\cdot|$ defined a metric on $\RR^d$.
Then, $\Delta_\tau^nL$ is the sum over $N$ random variables $\{Y_1,\ldots,Y_N\}$, where
$N$ is Poisson distributed with parameter $\nu(\RR^d\setminus B_\ep)$ and the random variables  $\{Y_1,\ldots,Y_N\}$
 are identical and independent distributed with
 $$
 Y_i\stackrel{d}{=}\frac {\nu(\cdot\cap B^c_\ep)}{\nu( B^c_\ep)},\quad i=1,\ldots N.
 $$
Now, one can replace the neglected small jumps by increments of a Wiener process. Here, the rate of convergence for the strong error will not improved. However,
calculating the weak error the quality of the approximation will be improved.
The reason is that the Markovian semigroup of the approximation where the small jumps  are  approximated by a Wiener process is analytic.
To explain the implication of this, let us consider the function $\phi:\RR\ni x\mapsto \mathds{1} _{[a,\infty)}(x)$.
Then, for $t>0$ we know 
$
 \PP\lk(  X^{x_0}_t \ge a\rk)=\EE \mathds{1}_{[a,\infty)}(X_t^{x_0})
$, where $X$ solves the stochastic differential equation ($b:\RR^ d \to\RR^ d $ are  Lipschitz continuous, $\sigma>0$)
\DEQSZ\label{eq1sf}
\lk\{\barray
dX^ {x_0}(t) &=& b(X^ {x_0}(t-))\, dt +\sigma dL(t)
\\
X^ x(0)&=&{x_0},\quad {x_0}\in\RR.
\earray\rk.
\EEQSZ
Let us denote the approximation of $X$, where we replaced the small jumps by a Wiener process, by $\hat X$.
Then, the two functions   $\phi_1:\RR\ni x_0\mapsto \EE \mathds{1} _{[a,\infty)}(X_t^{x_0})$ and  $\phi_2:\RR\ni x_0\mapsto \EE \mathds{1} _{[a,\infty)}(X_t^{x_0})$ are differentiable
and one can use the Taylor approximation to get a nice error estimate. In this way, the analyticity of the Markovian semigroup has a strong impact of the quality
of the approximation. After presenting our main result in this section, we illustrate the result by some numerical simulation.

\medskip

Fix a truncation parameter $0< \epsilon <1$. Let us  define the approximate
L\'{e}vy measure
$$\nu ^ {\ep} :\CB(\RR)\ni \mathcal{C} \mapsto \nu\Bigl( \mathcal{C}\cap\setminus B_\ep^c \Bigr) \,  .
$$
Let  $\hat L_{ \ep} $ be the L\' evy process induced by
truncating the small jumps. In particular,  $\hat L_{ \ep} $ ia a  L\' evy process having
intensity $\nu^{\ep}$.
%
Not to neglect the small jumps, we generate at each time-step $k\in\mathbb{N}$ a
Gaussian random variable $\Delta_{\tau}^{k} W_{\epsilon} $, where
\begin{equation}\label{wien1}
\Delta_{\tau}^{k}  W_{\epsilon} \sim \CN \Bigl(  0,
\Sigma^{2}({\ep}) \tau \Bigr)\,, \qquad \mbox{and} \qquad
\Sigma^2(\ep) =\int_{[-\ep,\ep]^ d} \la y,y\ra \,\nu(dy)
\, .
\end{equation}
Then, the increments of the L\'evy process are  approximated
by \DEQS \lk( \hat{\Delta}_{\tau}^{0} L_{\epsilon,1} +
\Delta_{\tau}^{0}  W_{\epsilon} ,  \hat{\Delta}_{\tau}^{1}
L_{\epsilon,1} +\Delta_{\tau}^{1}
W_{\epsilon},\hat{\Delta}_{\tau}^{2}
L_{\epsilon,1}+\Delta_{\tau}^{2}  W_{\epsilon} ,\, \ldots\, ,
\hat{\Delta}_{\tau}^{k} L_{\epsilon,1} +\Delta_{\tau}^{k}
W_{\epsilon},\,\ldots\, \rk). \EEQS

In the following we give an error estimate of the  two processes $X^x$ and $\hat X^x$, where $X^x$ solves \eqref{eq1sf}
and  $\hat X^x$ solves
\DEQSZ\label{eq1app}
\lk\{\barray
d\hat X^ x_\ep (t) &=& b(\hat X^ x_\ep(t-))\, dt +\sigma(\hat X^ x_\ep(t-)) d\hat L_\ep(t)+\sigma(\hat X^ x_\ep(t-)) d W_\ep(t)
\\
X^ x(0)&=&x,\quad x\in\RR^ d.
\earray\rk.
\EEQSZ
Again let us define the corresponding Markovian semigroups.
Let $(\CP_t)_{t\ge0}$ be the Markovian semigroup of the process \eqref{eq1sf}, i.e.
\DEQSZ\label{msg_exact}
\CP_t\phi(x) &:= &\EE \phi(X^ x(t)),\quad t\ge 0.
\EEQSZ
and
let  $(\hat \CP^ \ep _t)_{t\ge0}$ be the Markovian semigroup of the process \eqref{eq1app}, i.e.
\DEQSZ\label{msg_appepw}
\hat \CP^ \ep _t\phi(x)  &:= & \EE \phi(\hat X_\ep ^ x(t)),\quad t\ge 0.
\EEQSZ
The first proposition shows that the semigroup 
 $(\hat \CP^\ep _t)_{t\ge0}$ of the approximation $\hat X^x$ is  analytic on ${B^m _{p,q}(\RR^d)}$.
\begin{defn}
A probability measure on $\RR^d$ is selfdecomposable, iff for any $b>1$, there exists a probability measure $\rho_b$ on $\RR^d$ such that
$$
\hat \mu(z)=\hat \mu(b^{-1}z)\hat \rho_b(z),\quad z\in\RR^d.
$$
\end{defn}
By Theorem 15.10 of \cite[p. 95]{sato}, we know that there exists
\begin{itemize}
\item a finite measure $\lambda$ on the sphere $\mathbb{S}=\{x\in\RR^d: |x|=1\}$
\item  and a function
measurable $k:S\times \RR^+\to \RR_0^+$, decreasing in the second variable,
\end{itemize}
 such that the L\'evy measure $\nu$ of $L$ has the following
representation
$$
\nu(B) = \int_\mathbb{S} \int_0^\infty 1_B(rx)k(r,x)\frac {dr}r \,\lambda(dx), \quad B\in\CB(\RR^d).
$$

\begin{prop}\label{analytic}
Let $L$ be a L\'evy process, such that $L(1)$ has a selfdecomposable distribution such that there exists a
a finite measure $\lambda$ on $\mathbb{S}$ and a function
measurable $k:\mathbb{S}\times \RR^+\to \RR_0^+$, slowly varying at zero and monoton decreasing for $x\to\infty$ in the second variable
such that
$$
\nu(B) = \int_\mathbb{S} \int_0^\infty 1_B(rx)k(r,x)\frac {dr}{r^{1+\alpha}} \,\lambda(dx), \quad B\in\CB(\RR^d).
$$
Let us assume that there exists some $c_0>0$ such that $\lambda( \{|\la x,e_j\ra |> 1/\sqrt{2}\})\ge c_0$ for $j=1,\ldots,d$, where $e_j=(e_j^1,e_j^2,\ldots ,e_j^d)$ with $e_j^k=0$ iff $j\not=k$ and $e_j^j=1$.
Let us assume that $\sigma\in \cal^{d+3}_b(\RR^d)$ and $ b\in \cal^{d+3}_b(\RR^d)$ and that  $\sigma$ is bounded away from zero.
\begin{itemize}
\item Then,  for all $1\le p,q<\infty$, the Markovian semigroup $(\hat \CP^\ep_t)_{t\ge0}$ is an analytic semigroup in $B^m _{p,q}(\RR^d)$ for all $m\in\RR$. 
\item
Let $q(D)$
 be a pseudo--differential operator with symbol $q(\xi)$, where $q\in S^\kappa_{1,0}(\RR^ d\times \RR^ d)$ with $\kappa\le 2(\delta-1)$.
 Then, we have for $u\in{B^m _{p,q}(\RR^d)}$
 \DEQSZ\label{est1brauche}
 \big|  \hat\CP^\ep_t q(D) u \big|_{B^m _{p,q}}\le \frac Ct |u|_{B^m _{p,q}}.
 \EEQSZ
\end{itemize}
\end{prop}

\begin{thm}\label{error}
Let us assume that $\sigma\in \cal^{d+3}_b(\RR^d)$ and $ b\in \cal^{d+3}_b(\RR^d)$. Let  $\sigma$ be bounded away from zero.
Then,  for $\delta\in(1,2)$ (Blumenthal--Getoor index of L), $r_1,r_2\in(0,1)$ such that $r_1+r_2>1$ and $2r_1>r_2$ with $\delta_1=\frac{\delta r_2}{2}$ and $\delta_2=\delta(r_1-\frac{ r_2}{2})$,
 \del{ then for any $\gamma\in [0,2(\alpha-1))$ there exists a constant $C>0$ such that for any $r\in\RR$}
\DEQS
\big\| \CP_t-\hat \CP^ \ep_t\big\|_{L(B_{\infty,\infty}^{ -\delta_2},B_{\infty,\infty}^ {\delta_1})}\le C  t^{2(\delta-1)( r_1+ r_2)-1}\ep ^ {(2-\delta)} .
\EEQS
\end{thm}

To illustrate Theorem \ref{error} we would like to present the following simulations.
Here, we took as underlying process
\DEQSZ\label{sdestable}
dX_t=-aX^x(t^-)\, dt+dL(t),\quad X^x(0)=x,
\EEQSZ
where $(L_t)_{t\ge 0}$ is a strictly $\alpha$--stable process, $\alpha>1$.
We approximated this process once by cutting off the small jumps and replacing the small jumps by an independent Wiener process
described by \eqref{eq1app}, and secondly, by only cutting off the small jumps, i.e. by
\DEQSZ\label{eq1appnw}
\lk\{\barray
d\bar  X^ x_\ep (t) &=&-a\bar X^ x(t-) dt + d\hat L_\ep(t)
\\
X^ x(0)&=&x_0,\quad x\in\RR^ d.
\earray\rk.
\EEQSZ
Let $(\CP_t)_{t\ge0}$ be the Markovian semigroup of the process \eqref{eq1sf},
let  $(\hat \CP^ \ep _t)_{t\ge0}$ be the Markovian semigroup of the process \eqref{eq1app}, and, finally,
\DEQSZ\label{msg_appepw}
\bar \CP^ \ep _t\phi(x)  &:= & \EE \phi(\bar X_\ep ^ x(t)),\quad t\ge 0.
\EEQSZ
We simulated the Markovian semigroup  for two different functions $\phi_1$ and $\phi_2$
$$
\phi_1:= \mathds{1}_{x\ge 0.5}\quad \mbox{and}\quad  \phi_2(x)=\frac {x^2}{1+x^2}.
$$
In the first pictures we simulated  $\hat  \CP^ \ep _t\phi_i(x_0) $, $\bar \CP^ \ep _t\phi_i(x_0) $ and the difference
 $|\CP_t\phi_i(x_0)- \hat  \CP^ \ep _t\phi_i(x_0) |$ and  $|\CP_t\phi_i(x_0)- \bar \CP^ \ep _t\phi_i(x_0) |$ for $x_0=0$, $x_0=0.45$, and $i=1,2$, $\alpha=1.25, 1.90$.

\medskip
\del{
  \includegraphics[width=0.5\textwidth]{MS_phi_1hatx175.eps}
   \includegraphics[width=0.5\textwidth]{MS_phi_1barx175.eps}

\medskip

   \includegraphics[width=0.5\textwidth]{MS_phi_2hatx175.eps}
   \includegraphics[width=0.5\textwidth]{MS_phi_2barx175.eps}
}
It is observable that,  if $\alpha$ is closed to two, replacing the small jumps by a Wiener process improves the quality of the approximation.
Especially, if the underlying function is not differentiable. In our example, this can be easily seen by taking $x_0$ closed to $0.5$ and $\phi=\phi_2$.
If $\alpha$ is closed to one and $\phi$ is differentiable, there is nearly no improvement. This we could also verify by the following picture, where
we simulated the error between $\CP_1(\phi_i)$ and $\hat \CP_1^\varepsilon(\phi_i)$, respectively, between $\CP_1(\phi_i)$ and $\bar \CP_1^\varepsilon(\phi_i)$
for $i=1,2$ with $x_0=0.45$. Since our starting point is $0$ and $\phi_2$ is not continuously differentiable in zero, we verified quality of the convergence for $\alpha$ closed to $2$ and
starting point closed to $0.5$ for $\phi_2$. Here, we see that the improvement is not as good as for $\phi_1$.

Before presenting the proof, we want to give some remarks.
Note, that $(\CP_t)_{t\ge 0}$ has generator given by the symbol $a(x,\xi)=\psi(\sigma(x)^ T\xi)$ and
$(\hat \CP_t)_{t\ge 0}$ has generator given by the symbol $\hat a_\ep(x,\xi)=\psi_\ep (\sigma(x)^ T\xi)-\Sigma^2(\ep)\la \xi,\sigma^T(x)\xi\ra$,
where
$$
\psi_\ep(\xi) =\int_{\RR^ d\setminus B_\ep } (e^ {i\la y,\xi\ra }-1)\nu(dy),\quad \xi\in\RR^d,
$$
and
$$
\Sigma^2(\ep) =\int_{B_\ep } \la y,y\ra  \, \nu(dy).
$$
In addition, let us note that the L\'evy measure $\nu$ can be written as (see \cite{sato})
$$
\nu(B)=\int_{\mathbb{S}}\lambda(d\eta)\int_0^\infty 1_B(r\xi)k_\xi(r)\frac {dr}r,\quad B\in\CB(\RR^d),
$$
where $\lambda$ is a probability measure on the sphere  $\mathbb{S}$ and $k_\xi(r)$ is nonnegative, measurable in $\xi$ and decreasing in $r$.
\medskip

\begin{proof}[Proof of Proposition \ref{analytic}]
Let us assume that the support of $\nu$ belongs to $\{x\in\RR^d\mid |x|\le 1\}$.
The symbol for the approximation is given by
$$
\psi_\ep(\xi)+\Sigma^2(\ep)\la\xi,\xi\ra,
$$
where
$$
\psi_\ep(\xi):=\int_\mathbb{S} \int_\ep^1 (1-e^{ir\la x,\xi\ra}-i\la x,\xi\ra )\,k(r,x)\frac {dr}{r^{1+\alpha}} \,\lambda(dx)
$$
and
$$
\Sigma^2(\ep)=\int_\mathbb{S} \la x,x\ra\, \int_0^\ep k(r,x)\frac {dr}{r^{1+\alpha}} \,\lambda(dx).
$$
First we will show that there exists some $R>0$ and $c>0$ such that
$$
|\psi_\ep(\xi)+\Sigma^2(\ep)\la\xi,\xi\ra|\ge |\xi|^{\alpha},\quad \xi\in R\CU_1.
$$
Let $R>0$ such that $|\psi(\xi)\ge |\xi|^{\alpha}$ with $\xi\in R\CU_1$.
Substitution gives
\DEQS
\lqq{ \psi_\ep(\xi):=\int_\mathbb{S} \int_\ep^1(1-e^{ir\la x,\xi\ra}-i\la x,\xi\ra)\,k(r,x)\frac {dr}{r^{1+\alpha}} \,\lambda(dx)
}&&
\\
&\ge &
\int_\mathbb{S} \int_\ep^1(1-\cos( r\la x,\xi\ra ))\,k(r,x)\frac {dr}{r^{1+\alpha}} \,\lambda(dx)
\\
&\ge & c \int_\mathbb{S}(\la x,\xi\ra)^\alpha \int_{\ep\la x,\xi\ra }^{\la x,\xi\ra}(1-\cos( r))\frac {dr}{r^{1+\alpha}} \,\lambda(dx)
\EEQS
For  $|\xi|\in(2\pi,\ep^{-1})$ we get   and
\DEQS
 |\psi_\ep(\xi)|\ge  c\, c_0 |\xi|^\alpha \int_{\ep|\xi| }^{|\xi|}(1-\cos( r))\frac {dr}{r^{1+\alpha}}
\ge  \tilde c |\xi|^\alpha  \frac {2\pi-1+\sin(1)}{(2\pi)^{-1-\alpha}}.
\EEQS
For $|\xi|>\ep^{-1}$ we have
$$
\la \Sigma^2(\ep)\xi,\xi\ra\ge c_0\frac {\ep^{2-\alpha}}{2-\alpha}|\xi|^2\ge  c_0\frac {\xi^{\alpha}}{2-\alpha}|\xi|^2.
$$
Hence,
\DEQS
\frac 1{\psi_\ep(\xi)+\Sigma^2(\ep)\la\xi,\xi\ra}\le C |\xi|^{-\alpha}.
\EEQS
It remains to investigate for a multiindex
$$
\partial^{\alpha}_\xi\Big[ \frac 1{\psi_\ep(\xi)+\Sigma^2(\ep)\la\xi,\xi\ra}\Big]
$$
with  derivatives up to order $2d+4$.
 Then, we have for $|\alpha|=1$ and $\xi\in\delta \CU_1$ 
\DEQS
 |\frac{\partial }{\partial \xi_j} \psi_\ep(\xi)|\le | \frac{\partial }{\partial \xi_j} \psi(\xi)|\le \gr |\xi|\gl^{\alpha-1}.
\EEQS
On the other side
$$
|\frac{\partial }{\partial \xi_j}\Sigma(\ep)\la\xi,\xi\ra |\le \frac{\ep^{2-\alpha}}{2-\alpha} |\xi| ,
$$
In this way we get
\DEQS
 \partial^{\alpha}_\xi\Big[ \frac 1{\psi_\ep(\xi)+\Sigma^2(\ep)\la\xi,\xi\ra} \Big]\le \gr\xi\gl^{-2\alpha+\alpha-1}\wedge \big(\frac{\ep^{2-\alpha}}{2-\alpha}  \gr\xi\gl^{-2\alpha+1}\Big).
\EEQS
Iterating gives that 
$\psi_\ep(\xi)+\Sigma^2(\ep)\la\xi,\xi\ra$ belongs to $\mbox{Hyp}^{2(\alpha-1)}_{2d+4,d+3;1,0}(\RR^d\times \RR^d)$. Therefore we can conclude that the Markovian semigroup is analytic.

 \end{proof}

\medskip

\begin{proof}[Proof of Theorem \ref{error}]
Firstly, we have to show that
\DEQSZ\label{firstres}
\lk| a(x,\xi)-a_\ep(x,\xi)\rk|\le \lk|\xi\rk|^ 2 \ep ^ {(2-\delta)}
\EEQSZ
In particular, we have to show for any  multiindex $\gamma$ with $|\gamma|\le d+1$,
we have
$
\lk|\frac {\partial ^ \gamma}{\partial _\xi} \lk[ a(x,\xi)-a_\ep(x,\xi)\rk] \rk|  \le C\, \ep ^ {|\gamma|-\delta}.
$
 That means $a(x,\xi)-a_\ep(x,\xi)\in S^{2}_{1,0}(\RR^ d\times \RR^ d)$. Through out this proof $C$, is denoted as a varying constant.
 Let us start with $\gamma=0$.

\DEQS
\lqq{ \lk| a(x,\xi)-a_\ep(x,\xi) \rk| = \lk|\psi(\sigma(x)^ T\xi)-\psi_\ep (\sigma(x)^ T\xi)+\Sigma(\ep)( \xi,\sigma^T(x)\xi) \rk|}
&&
\\
&=& \biggl| \int_ {\RR^ d\setminus \{0\} } \lk[e^{i(y,\sigma(x)^ T\xi)}-1-i( y,\sigma(x)^ T\xi) \rk]\nu(dy)-\int_ {\RR^ d\setminus B_\ep} \lk[e^{i(y,\sigma(x)^ T\xi)}-1\rk]\nu(dy)
\\
&+&
 \int_{B_\ep } (\xi,\sigma(x)^ T\xi)( y,y) \nu(dy) \biggr|
\\
& =& \biggl| \int_{B_\ep }\lk[e^{i(y,\sigma(x)^ T\xi)}-1-i(y,\sigma(x)^ T\xi)+(\xi,\sigma(x)^ T\xi)( y,y)\rk] \nu(dy) \biggr|
\\
 &\le &  \int_{B_\ep }\biggl|e^{i(y,\sigma(x)^ T\xi)}-1-i( y,\sigma(x)^ T\xi)\biggr| \nu(dy)+ \lk|(\xi,\sigma(x)^ T\xi)\rk|\int_{B_\ep }\lk|y\rk|^2 \nu(dy)
\\
&\le & \int_{B_\ep }\lk|(y,\sigma(x)^ T\xi)\rk|^2 \nu(dy)+ C_1(d)\lk|(\xi,\sigma(x)^ T\xi)\rk|\int_{B_\ep }\lk|y\rk|^2 \nu(dy)
\\
&\le& \lk[\lk|\xi\rk|^2\lk|\sigma(x)^ T\rk|^2+ \lk|\xi\rk|^2\lk|\sigma(x)^ T\rk|\rk]\int_{B_\ep }\lk|y\rk|^2 \nu(dy)\\
&\leq
& C \lk|\xi\rk|^2\int_{B_\ep }\lk|y\rk|^2\nu(dy)\le C \lk|\xi\rk|^2\int_{\mathbb{S}}\lk|\eta\rk|^2\lambda(d\eta)\int_0^\ep r^2\frac{dr}{r^{1+\delta}}\le C \lk|\xi\rk|^2\ep^{2-\delta},
\EEQS
\del{\red{why does we have
$$\int_{[-\ep,\ep]^ d} \lk| y\rk| ^2  \nu(dy)= \lk( \int_{-\ep}^ {\ep} \lk| y\rk| ^{1-\delta} dy\rk)^d
$$
In my opinion this is wrong. Also for the case of independent Levy processes the intensity measure consits of a sum. Please, take a decomposable
multidimensional \levy measure as example  }
}
where 
$\lk|y\rk|= \sqrt{y_1^2+\cdots +y_d^ 2 }$.  The second inequality from the top of the above estimate is due to the first estimate in the proof of the Lemma 15.1.7 in \cite{Davar}. Since $\nu$ is a $\delta$-stable \levy measure, we can apply result 14.7 in p.79 \cite{sato} to get the last estimate.

\medskip

Now, let us consider $\gamma=1$. We have for $1\le j\le d$
 \DEQS
\lqq{ \lk|\frac{d}{d\xi_j}\lk[ a(x,\xi)-a_\ep(x,\xi) \rk]\rk| = \lk|\frac{d}{d\xi_j}\lk[\psi(\sigma(x)^ T\xi)-\psi_\ep (\sigma(x)^ T\xi)+\Sigma(\ep)( \xi,\sigma^T(x)\xi) \rk]\rk|}
&&
\\
& =& \biggl| \int_{B_\ep }\frac{d}{d\xi_j}\lk[e^{i(y,\sigma(x)^ T\xi)}-1-i(y,\sigma(x)^ T\xi)+(\xi,\sigma(x)^ T\xi)( y,y)\rk] \nu(dy) \biggr|
\\
& =& \biggl| \int_{B_\ep }\sum_{k=1}^d\lk[iy_k\sigma_{jk}(x)\lk[e^{i(y,\sigma(x)^ T\xi)}-1\rk]+2\lk[\xi_k\sigma_{jk}(x)\rk]( y,y)\rk] \nu(dy) \biggr|
\\
&\le  &  \int_{B_\ep }\sum_{k=1}^d\Big[\lk|y_k\sigma_{jk}(x)\rk|\lk|e^{i(y,\sigma(x)^ T\xi)}-1-i(y,\sigma(x)^ T\xi)\rk|
\\
&&\vspace{2cm}+\lk|y_k\sigma_{jk}(x)\rk|\lk|(y,\sigma(x)^ T\xi)\rk|+2\lk[\xi_k\sigma_{jk}(x)\Big]( y,y)\rk] \nu(dy)
\\
&\le  &  \int_{B_\ep }\sum_{k=1}^d\lk[\lk|y_k\sigma_{jk}(x)\rk|\lk|(y,\sigma(x)^ T\xi)\rk|^ 2+\lk|y_k\sigma_{jk}(x)\rk|\lk|(y,\sigma(x)^ T\xi)\rk|+2\lk[\xi_k\sigma_{jk}(x)\rk]( y,y)\rk] \nu(dy)
\\
&\le  &  \int_{B_\ep }\sum_{k=1}^d\lk|y_k\sigma_{jk}(x)\rk|\lk[\lk|(y,\sigma(x)^ T\xi)\rk|+\lk|(y,\sigma(x)^ T\xi)\rk|^ 2 \rk]\nu(dy)+ \lk|\xi\rk|\lk|\sigma(x)^ T\rk|\int_{B_\ep }\lk|y\rk|^2 \nu(dy)
\\
&\le& C\sum_{n=2 }^3\int_{B_\ep } \lk|\xi\rk|^n\lk|y\rk|^n\nu(dy)\le C\sum_{n=2 }^3 \lk|\xi\rk|^n\int_{\mathbb{S}}\lk|\eta\rk|^n\lambda(d\eta)\int_0^\ep r^n\frac{dr}{r^{1+\delta}}\le C \sum_{n=2 }^3\lk|\xi\rk|^n\ep^{n-\delta}.
\EEQS
To get the last estimate we applied the similar steps  as in previous estimate.
Now we estimate
$$\lk|\frac{\partial^\gamma}{\partial^{\alpha_1}\xi_1\ldots\partial^{\alpha_d}\xi_d}\lk[ a(x,\xi)-a_\ep(x,\xi) \rk]\rk|
$$ with $2\le|\gamma|\le d+1 $.
Here, we get the following sequence of calculations
\DEQS
\lqq{ \lk|\frac{\partial^\gamma}{\partial^{\alpha_1}\xi_1\ldots\partial^{\alpha_d}\xi_d}\lk[ a(x,\xi)-a_\ep(x,\xi) \rk]\rk|}
&&
\\
& =& \lk|\frac{\partial^\gamma}{\partial^{\alpha_1}\xi_1\ldots\partial^{\alpha_d}\xi_d}\lk[\psi(\sigma(x)^ T\xi)-\psi_\ep (\sigma(x)^ T\xi)+\Sigma(\ep)( \xi,\sigma^T(x)\xi) \rk]\rk|
\\
 &=&\biggl| \int_{B_\ep }\frac{\partial^\gamma}{\partial^{\alpha_1}\xi_1\ldots\partial^{\alpha_d}\xi_d}\lk[e^{i(y,\sigma(x)^ T\xi)}-1-i(y,\sigma(x)^ T\xi)+(\xi,\sigma(x)^ T\xi)( y,y)\rk] \nu(dy) \biggr|
\\
&\le &\lk| \int_{B_\ep }\sum_{k=1}^d\Pi_{j=1}^d \lk(y_k\sigma_{jk}(x)\rk)e^{i(y,\sigma(x)^ T\xi)} \nu(dy) \rk|+\lk|C(x)\rk|\int_{B_\ep }\lk| ( y,y)\rk|\nu(dy)
\\
&\le & C\int_{B_\ep }\lk|\sum_{j=1}^d y_j\rk|^{\lk|\gamma\rk| }\nu(dy) +\lk|C(x)\rk|\int_{B_\ep }\lk| ( y,y)\rk|\nu(dy)
+ C\int_{B_\ep }\lk| y\rk|^{\lk|\gamma\rk| }\nu(dy) +\lk|C(x)\rk|\int_{B_\ep }\lk| y\rk|^ 2\nu(dy)
\\
&\le& C \int_{B_\ep }\lk|y\rk|^{\lk|\gamma\rk| }\nu(dy)\le C \int_{\mathbb{S}}\lk|\eta\rk|^{\lk|\gamma\rk| }\lambda(d\eta)\int_0^\ep r^{\lk|\gamma\rk| }\frac{dr}{r^{1+\delta}}\le C \ep^{^{\lk|\gamma\rk| }-\delta},
\EEQS
where $C(x)=f\lk(\sigma_{jk}(x)\rk)$ if $\lk|\gamma\rk| =2$ and $C(x)=0$ if $\lk|\gamma\rk| \geq 3$ .
   Observe, due to \cite[ Proposition 3.1.2, in p.77 ]{Pazy:83}   we can write
\DEQS
\lk[ \CP_t  -  \hat \CP_t^ \ep\rk] u  = \int_0^ t \hat \CP^ \ep_{t-s} \lk[ a(x,D)-a_\ep(x,D)\rk] \CP_s u\, ds.
\EEQS
Now let $r_1,r_2\in(0,1)$ such that $r_1+r_2>1$ and $2r_1>r_2$. Then for $u\in B^{-\delta(r_1-\frac{ r_2}{2})}_{\infty,\infty}(\RR^d)$  we get
\DEQS
\lqq{ \lk| \lk[ \CP_t  -  \hat \CP_t^ \ep\rk] u \rk|_{ B^{\frac{\delta r_2}{2}}_{\infty,\infty}} }
&&
\\
&\le& \int_0^ t \lk|  \hat \CP^ \ep_{t-s}\lk[ a(x,D)-a_\ep(x,D)\rk] \CP_s u\rk| _{ B^{\frac{\delta r_2}{2}}_{\infty,\infty}}\, ds
\\
&\le& \int_0^ t \lk\|  \hat\CP^ \ep_{t-s}\rk\|_{L\big( { B^{-\frac{\delta r_2}{2}}_{\infty,\infty}},{ B^{\frac{\delta r_2}{2}}_{\infty,\infty}} \big)}
\lk\| a(x,D)-a_\ep(x,D)\rk\|_{L\big( { B^{\frac{\delta r_2}{2}}_{\infty,\infty}},{ B^{-\frac{\delta r_2}{2}}_{\infty,\infty}} \big)}
\\
&\times & \lk\| \CP_s \rk\|_{L\big( { B^{-\delta(r_1-\frac{ r_2}{2})}_{\infty,\infty}},{ B^{\frac{\delta r_2}{2}}_{\infty,\infty}} \big)} \lk|u\rk| _{B^ {-\delta(r_1-\frac{ r_2}{2})}_{\infty,\infty}}\, ds
\EEQS
\DEQS
\lqq{\le\lk\|a(x,D)-a_\ep(x,D)\rk\|_{L\big( { B^{\frac{\delta r_2}{2}}_{\infty,\infty}},{ B^{-\frac{\delta r_2}{2}}_{\infty,\infty}} \big)}\lk|u\rk| _{B^ {-\delta(r_1-\frac{ r_2}{2})}_{\infty,\infty}} \int_0^ t   (t-s)^ { -\delta r_2}
 s^ {-\delta r_1}
\, ds}
&&
\\
&\le& t^{\delta( r_1+ r_2)-1} \lk\| a(x,D)-a_\ep(x,D)\rk\|_{L\big( { B^{\frac{\delta r_2}{2}}_{\infty,\infty}},{ B^{-\frac{\delta r_2}{2}}_{\infty,\infty}} \big)}\lk|u\rk| _{B^ {-\delta(r_1-\frac{ r_2}{2})}_{\infty,\infty}} \int_0^ 1   (1-s)^ { -\delta r_2}
 s^ {-\delta r_1}
\, ds
\\
&\le &t^{\delta( r_1+ r_2)-1}\lk\| a(x,D)-a_\ep(x,D)\rk\|_{L\big( { B^{\frac{\delta r_2}{2}}_{\infty,\infty}},{ B^{-\frac{\delta r_2}{2}}_{\infty,\infty}} \big)}\lk|u\rk| _{B^ {-\delta(r_1-\frac{ r_2}{2})}_{\infty,\infty}} \, B(1-\delta r_1,1-\delta r_2)
\\
&\le & Ct^{\delta( r_1+ r_2)-1}\lk|u\rk| _{B^ {-\delta(r_1-\frac{ r_2}{2})}_{\infty,\infty}} .
\EEQS
For the last inequality we used the fact that  we have already proven \eqref{firstres}. Therefore,  it follows from   Theorem 6.19 in \cite{pseudo2} that for any  $m\in\RR$
$$( a(x,D)-a_\ep(x,D)):B_{\infty,\infty}^{2 +m}(\RR^d)\to B_{\infty,\infty}^{m}(\RR^d)
$$
is a bounded operator. Therefore we have
$$
\lk\| a(x,D)-a_\ep(x,D)\rk\|_{L\big( { B^{\frac{\delta r_2}{2}}_{\infty,\infty}},{ B^{-\frac{\delta r_2}{2}}_{\infty,\infty}} \big)}\le C \ep ^ {(2-\delta)}.
$$
Let $\delta_1=\frac{\delta r_2}{2}$ and $\delta_2=\delta(r_1-\frac{ r_2}{2})$. Rewriting above gives
\DEQS
 \lk|  \lk[ \CP_t  -  \hat \CP_t^ \ep\rk] u \rk|_{\cal^ 0_b} \le  \lk|  \lk[ \CP_t  -  \hat \CP_t^ \ep\rk] u \rk|_ {B^ {\delta_1}_{\infty,\infty} }
\le t^{\delta( r_1+ r_2)-1}\ep ^ {(2-\delta)} |u|_{ B^ {-\delta_2}_{\infty,\infty} }\le  t^{\delta( r_1+ r_2)-1} \ep^  {(2-\delta)} |u|_{\cal^ 0_b}.
\EEQS
\del{
Due to the Taylor expansion one gets
\DEQS
\lk| a(x,\xi)-a_\ep(x,\xi)\rk|\le \xi^ 2 \ep ^ {2-\alpha}, \quad \xi >0.
\EEQS
Hence, since $\alpha>1$ we can find some $r_1,r_2\in(0,1)$ such that $r_1+r_2>1$, $\alpha (r_1+r_2-\delta)\ge 2$. Therefore,
\DEQS
\lk| a(x,D)-a_\ep(x,D)\rk|_{L( B^{\alpha r_2}_{\infty,\infty}, B^ {-\alpha r_1}_{\infty,\infty}) } \le C\, \ep^ {2-\alpha}.
\EEQS
Therefore,
\DEQS
 \lk|  \lk[ \CP_t u -  \hat \CP_t^ \ep\rk] u \rk|_{\cal^ 0_b} \le  \lk|  \lk[ \CP_t u -  \hat \CP_t^ \ep\rk] u \rk|_ {B^ {\delta/2}_{\infty,\infty} }
\\
\le \ep^ {2-\alpha} |u|_{ B^ {-\delta/2}_{\infty,\infty} }\le  \ep^ {2-\alpha} |u|_{\cal^ 0_b}.
\EEQS
}
\end{proof}

\section{Invertibility of  pseudo--differential operators}
\label{nsec}
In this section, we study under which conditions the pseudo-differential operator is invertible.
To investigate the inverse of a pseudo-differential operator one has to introduce the set of elliptic and hypoelliptic symbols. For the reader's convenience, we define the elliptic and hypoelliptic symbols in this section.


\medskip

\del{\begin{rem}\label{hypadjoint}
		Let $a(x,\xi)\in \Hyp_{\rho,\delta}^{m,m_0}(X\times \RR^ d)$ be a symbol and $a^\ast(x,\xi)$ the symbol of the formal adjoint operator $a^\ast(x,D)$. Then, one can see from the expansion in
		\eqref{expadj}, that if $a(x,\xi)\in\Hyp_{\rho,\delta}^{m,m_0}(X\times \RR^ d) $ then $a^\ast(x,\xi)\in \Hyp_{\rho,\delta}^{m,m_0}(X\times \RR^ d) $.
	\end{rem}
}
We are now interested, under which condition an operator  $a(x,D)$ is invertible.
To be more precise, %
we aim to answer  the following questions. Given $f\in B_{p,r}^s(\RR^d)$, does there exists an element  $u\in\CSS'(\RR^d)$ such that
\DEQSZ\label{probelm01all}
a(x,D) u(x) = f(x),\quad x\in\RR^d,
\EEQSZ
and to which Besov space belongs 
$u$?
\medspace

The invertibility is used for giving bounds of the resolvent of an operator $a(x,D)$.
Here, one is interested not only in the invertibility of $a(x,D)$ but also  in the invertibility of  $\lambda+a(x,D)$, $\lambda \in \rho(a(x,D))$. In particular, we are interested in the  norm of the operator $[\lambda+a(x,D)]^{-1}$ uniformly for all $\lambda$ belonging to the set of resolvents.
However, executing a careful analysis, we can see that certain constants depend only on the first or second derivative on
the symbol of $\lambda+a(x,D)$, which has the effect that this norm is independent of $\lambda$. Hence,  it is necessary to introduce the following class.%
\begin{defn} 
	Let 
	$\rho,\delta$ be two real numbers such that $0\le \rho\le 1$ and $0\le \delta\le 1$. Let $m\in\RR$ and $\kappa\in\NN_0$.
	Let $\tilde \CA^{m,\kappa}_{k_1,k_2;\rho,\delta}(\RR^ d\times \RR^ d)$ be the set of all functions $a:\RR ^d \times \RR^ d \to \mathbb{C}$, where
	\begin{itemize}
		\item  $a(x,\xi)$ is infinitely often differentiable, i.e.\ $a\in \cal_b^\infty(\RR^d \times\RR^d)$;
		\item
		for any two multi-indices $\alpha$ and $\beta$, with $|\alpha|\ge \kappa$, there exists $C_{\alpha,\beta}$  such that 
		$$
		\lk| \partial ^ \alpha_{\xi'} \partial ^ \beta_x a(x,\xi')\Big|_{\xi'=\gamma\xi} \rk|\leq C_{\alpha,\beta}  
		\gr| \gamma\xi|\gl ^ {m-\rho|\alpha|}\ggx^{\delta|\beta |},\quad x\in \RR^d,\xi\in \RR^ d .
		$$
	\end{itemize}
\end{defn}
For $k_1,k_2\in\NN_0$, we also introduce the following  semi--norm  for $ a\in \tilde{\CA} ^{m,\kappa}_{k_1,k_2;\rho,\delta}(\RR^d,\RR^d)$ by
\DEQS
\lqq{ \| a\| _{\tilde{\CA}_{k_1,k_2;{\rho,\delta}}^{m,\kappa} }
} &&
\\
&=& \sup_{\kappa\le|\alpha|\le k_1,|\beta|\le k_2} \, \sup_{(x,\xi,\gamma)\in \RR^d
	\times \CU_1\times [1,\infty) } \lk| \partial ^ \alpha_{\xi'} \partial ^ \beta_x a(x,\xi')\Big|_{\xi'=\gamma\xi} \rk|
\gr|\gamma\xi|\gl  ^ {\rho|\alpha|-m}\ggx^{\delta|\beta|} \,.
\EEQS
Now we are ready to state the main result of this section.
\begin{rem}
The outline of the proof of following theorem, i.e. Theorem \eqref{invert22} is quite similar to the proof of Theorem 5.4 in \cite{kumanago}, however there is an important difference. We have to introduced a symbol class $\tilde \CA^{m,\kappa}_{k_1,k_2;\rho,\delta}(\RR^ d\times \RR^ d)$, since we needed to construct the parametrix of the resolvent of $[\lambda+a(x,D)]$ of an operator $a(x,D)$, where $\lambda$ belongs to $\sigma(A)$ and can be quite large.	
\end{rem}
\begin{thm}\label{invert22}
	Let $ k \ge 0$, $m\in\RR$, $1\le p,r<\infty$. 
	Let $a(x,\xi)$ be a symbol such that $a\in \tilde \CA^{1,-1}_{2d+4,d+3;1,0}(\RR^d\times \RR^d)\cap \Hyp_{d+1,0;1,0}^{\kappa}(\RR^d\times\RR^d)$ for $\kappa=[k]$.
	Let $R\in\NN$ such that 
\DEQSZ\label{r1cond}
R &\ge& 10\times d \times  \lk\| a\rk\|_{\tilde \CA_{2d+1,d+1;1,0}^{1,-1}}\, \|a\|_{\tiny \Hyp_{2d+1,0;1,0}^{\kappa}}
\EEQSZ
	and
\DEQSZ\label{r2cond}
\gr|\gamma\xi|\gl^{\kappa}&\le& \frac{|a(x,\gamma\xi)|}{\,\,\,\,|a|_{\CA^\kappa_{0,0;1,0}}}\quad \mbox{for all} \quad x\in\RR^d\quad \mbox{and} \quad\xi\in \CU_1\quad \mbox{with}\quad \gamma\ge R.
\EEQSZ
	Then, there exists  
a  bounded pseudo--differential operator $B:B_{p,r}^{m}(\RR^d)\to B_{p,r}^{m}(\RR^d)$ with symbol $b(x,D)$, such that
	\begin{itemize}
		\item $\{ (\xi,x)\in \RR^d\times \RR^d: \sup_{x\in\RR^d}    b(x,\xi)>0 \}\subset \{ |\xi|\le 2 R\}$,
		\item $B$ has norm $R^k$ on $B_{p,r}^{m}(\RR^d)$ into itself,
		\item $
		a(x,D)=A+B
		$,
	\end{itemize}
	and, given $f\in B_{p,r}^{m}(\RR^d)$, the problem
	\DEQSZ\label{probelm01}
	A u(x) = f(x),\quad x\in\RR^d
	\EEQSZ
	has a unique solution $u$ belonging to $ B_{p,r}^{m+\kappa}(\RR^d)$.
	In addition,
	there exists a constant $C_1>0$ such that for all $f\in B_{p,r}^{m}(\RR^d)$ and $u$ solving \eqref{probelm01} we have
	$$
	|u|_{B_{p,r}^{m+\kappa}} \le C_1\,\|a\|_{\tiny \Hyp_{d+1,0;1,0}^{\kappa}}
	|f|_{ B_{p,r}^{m}},\quad f\in B_{p,r}^{m}(\RR^d).
	$$
\end{thm}
\begin{rem}
Since $a(x,\xi)$ is elliptic, we can find a number $R>0$ satisfying \eqref{r1cond} and \eqref{r2cond}.
\end{rem}

\begin{rem}
	In fact, analyzing the resolvent $[\lambda+a(x,D)]^{-1}$  of an operator $a(x,D)$, it will be important that we have here in the Theorem the norm of $\tilde \CA^{m,1}_{\rho,\delta}(\RR^ d\times \RR^ d)$ and
	not the norm of  $\CA^{m}_{\rho,\delta}(\RR^ d\times \RR^ d)$. The reason is that calculating the norm in $\tilde \CA^{m,1}_{\rho,\delta}(\RR^ d\times \RR^ d)$
	the first derivative has to be taken. Therefore, the norm in $\tilde \CA^{m,1}_{\rho,\delta}(\RR^ d\times \RR^ d)$ is independent of $\lambda$.
\end{rem}


\begin{proof}
Note, that, for convenience for the reader, we summerized several definition and results necessary for the proof in the appendix \ref{pseudo-app}.
	For simplicity, let $E=B_{p,r}^{m}(\RR^d)$.
	Let $\chi\in \cal^\infty_b(\RR_0^+)$ such that
	$$
	\chi(\xi)=\bcase 0 &\mbox{ if } |\xi|\le 1,\\1 &\mbox{ if } |\xi|\ge 2,
\\ \in (0,1) &\mbox{ if } |\xi|\in (1,2). \ecase
	$$
\del{	A good choice is e.g.\
	\DEQSZ\label{goodchoice}
	\chi (\xi) = \bcase 0 &\quad \mbox{for} \quad |\xi|\le 1
	\\  1-g(|\xi|) &\quad \mbox{for} \quad 1\le |\xi|\le 2
	\\1 &\quad \mbox{for} \quad |\xi|\ge 2
	\ecase
	\EEQSZ
	where
	$$
	g(x):= \exp\lk[-\frac{(1-x)^2)}{(4-x^2)}\rk].
	$$}
	Let us put $\chi_R(\xi):=\chi(\xi/R)$, $\xi\in\RR^d$.
	%
	%
%
In addition, let us set
$$ p_R(x,\xi):= a(x,\xi)\chi_R(\xi), \quad b(x,\xi):= a(x,\xi)(1-\chi_R(\xi)),  \quad \mbox{and} \quad
q_R(x,\xi):=\frac{1}{a(x,\xi)}\chi_R(\xi).
$$
Due to the condition on $R$,
	the function $q_R(x,\xi)$ is bounded and  as a symbol, it is well defined.

\medskip
	Let us  consider the following problem:
Given $f\in B_{p,r}^m(\RR^d)$, find  an element  $u\in\CSS'(\RR^d)$ such that we have
	\DEQSZ\label{shortq}
	p_R(x,D) u(x) = f(x),\quad x\in\RR^d.
	\EEQSZ
	Observe, that on one hand for a solution $u$ of \eqref{shortq} we have
	\DEQS
	\lk[ q(x,D)  p_R(x,D) \rk] u = q(x,D) f, 
	\EEQS
	and, on the other hand, by Remark \ref{compositionremark},  the symbol for
	$q(x,D) \, p_R(x,D)$ is given by
	\DEQS\lqq{
	(q\, \sharp\,  p_R)(x,\xi)=q(x,\xi) p_R(x,\xi)} &&
\\
& & {}+
C(x,\xi)+
\sum_{|\gamma |=\max(d-k+2,k)}
	\OS\iint 
	e^{- i\la y,\eta\ra }  \eta ^\gamma
	r_\gamma (x,\xi,y,\eta)\, dy\, d\eta,
\EEQS
	where
\DEQSZ\label{naechstersummnd}
C(x,\xi)=\sum_{1\le |\rho|\le \max(d-k+1,k-1)} \partial _\xi^\rho q(x,\xi)  \partial _x^\rho  p_R(x,\xi)
\EEQSZ
and
	\DEQSZ\label{gammar}
	r_\gamma (x,\xi,y,\eta)=\int_0^1 \lk[ \partial^\gamma   _{\xi'} q(x',\xi')\Big|_{\xi'=\xi-\theta \eta\atop x'=x}
	\partial^\gamma   _{x'} p_R(x',\xi')\Big|_{\xi'=\xi\atop x'=x-y}   \rk]\, d\theta.
	\EEQSZ
Observe, firstly, for $\xi\in\gamma\CU_1$ we have 
\DEQSZ\label{careful}
\lqq{ | \partial _\xi^\rho q(x,\xi)  \partial _x^\rho  p_R(x,\xi)|} &&\nonumber
\\ &\le &2|a|_{\tiny \Hyp_{|\rho|,0;1,0}^{\kappa}} \, \lgxi^{-\kappa-|\rho|} \, |p|_{\tilde \CA^{\kappa,1}_{0,\rho;0,0}} \lgxi^{\kappa}
\le 2|a|_{\tiny \Hyp_{|\rho|,0;1,0}^{\kappa}} \, |p|_{\tilde \CA^{\kappa,1}_{0,\rho;0,0}} \lgxi^{-1}
\EEQSZ
	Observe, secondly, that by integration by part 
	we have
	\DEQS
	\OS\iint 
	e^{- i\la y,\eta\ra }  \eta ^\gamma
	r_\gamma (x,\xi,y,\eta)\, dy\, d\eta
	=-\,\OS\iint 
	e^{- i\la y,\eta\ra } 
	\partial_\zeta^\gamma r_\gamma (x,\xi,\zeta,\eta)\big|_{\zeta=y}\, dy\, d\eta
	.
	\EEQS
	Putting
	$$
	m_R(x,\xi):=  \sum_{|\gamma| =1} \OS\iint e^{- i\la y,\eta\ra } 
	\partial y^\gamma r_\gamma  (x,\xi,y,\eta)\, dy\, d\eta,
	$$
	one can verify that  $m_R(x,\xi)\in \tilde \CA^{-1,1}_{d+1,0;1,0}(\RR^d\times \RR^d)$.
	In fact, since $a\in \tilde  \CA^{\kappa,1}_{2d+4,d+3;1,0}(\RR^d\times \RR^d)\cap
	\mbox{Hyp}^{\kappa}_{d+1,0;1,0}(\RR^d\times \RR^d)$,
	we know by Theorem \ref{theoremb1} that 
	$$
	\partial_y^\gamma r_\gamma(\cdot,\cdot,x,\xi)\in \CA_{d+1,d+3;1,0}^{-1}(\RR^d\times \RR^d)\cap \CA_{d+1,0;1,0}^{0}(\RR^d\times \RR^d).
	$$
	This can be seen by straightforward calculations.
	First, \del{we have
	$$
	\partial_\xi ^\gamma\lk[\frac 1 {p_R(x,\xi)} \rk] = \frac 1 {[p _R(x,\xi)]^2} \,\,\partial_\xi ^\gamma p_R(x,\xi)\lesssim \lgxi^{-2\kappa } \lgxi^{\kappa-|\gamma|}
	\lesssim \lgxi^{-\kappa-|\gamma|},\quad \xi \in \RR^d.
	$$
	In particular,}
 by the definition of the hypoelliptic norm we have for any multiindex $\alpha$ and $\xi\in\delta\CU_1$, $\delta>1$ 
	$$
	\partial_\xi ^{(\gamma,\alpha)}\lk[\frac 1 {p_R(x,\xi)} \rk]\le \|a\|_{\tiny \Hyp_{|\alpha|+|\gamma|,0;1,0}^{\kappa}}\lgxi^{-\kappa-|\alpha|-|\gamma| }. 
	$$
	Next,  by the definition of the norm in $\tilde \CA^{\kappa,1}_{1,1;1,0}$ we have for any multiindex $\alpha$
	$$
	\partial_x ^{(\gamma,\alpha)} {p_R(x,\xi)} \le 
\|a\|_{\tilde \CA^{\kappa,1}_{|\alpha|+|\gamma|,1;1,0}}\,  \lgxi^{\kappa},\quad \xi\in\delta\CU_1, \,\delta>1.
	$$
	Going back to the operator $m_R(x,D)$. By Theorem 3.13 in \cite[p.\ 50]{pseudo2},
	we can interchange the derivatives with the oscillatory integral. That is
	\DEQS
	\partial^\alpha_\xi m_R(x,\xi) =
	\iint e^{- i\la y,\eta\ra} \int_0^1\partial^\alpha_\xi \lk[ \partial^\gamma   _{\xi'} q(x',\xi')\mid_{\xi'=\xi+\theta \eta\atop x'=x}
	\partial^\gamma   _{x'} p_R(x',\xi')\mid_{\xi'=\xi\atop x'=x+y}   \rk]\, d\theta
	\, dy\, d\eta,
	\EEQS
Secondly, \del{straightforward calculations gives for two symbols $p_1$ and $p_2$ the Leibniz formula
$$
\partial ^\alpha (p_1(x,\xi)p_2(x,\xi))=\sum_{|\beta_1|+|\beta_2|=n} \lk( {|\alpha|\atop \beta_1!\beta_2!}\rk) \partial ^{\beta_1}p_1(x,\xi)\partial^{\beta_2} p_2(x,\xi).
$$
	%
	Thirdly,}
by the Young inequality for a product, we know   for $s>0$
	$$\gr \xi+\theta \eta \gl^{-2s}\le \gr \xi\gl^{-s}\gr \theta \eta\gl^{-s}, 
	$$
and
by the Peetre inequality (see \cite[Lemma 3.7, p.\ 44]{pseudo2}),  we know for $s>0$
	$$\gr \xi+\theta \eta \gl^{s}\le \gr \xi\gl^{s}\gr \theta \eta\gl^{s}. 
	$$
Next,
straightforward calculations gives for $s>d$
\DEQS
\int_{\RR^d}\gr \eta\gl^{-s}\,d\eta\,\le C.
\EEQS
\del{In addition, we have for $s>d$ and $0<\ep<\frac 1d$
\DEQS
\lqq{ \int_0^1 \int_{\RR^d}\gr \theta \eta\gl^{-s}\,d\eta\, d\theta }&&
\\
&\le& \int_0^1 \int_{B(\theta^{-\ep })}\gr \theta \eta\gl^{-s}\,d\eta\, d\theta+\int_0^1 \int_{\RR^d\setminus B(\theta^{-\ep })}\gr \theta \eta\gl^{-s}\,d\eta\, d\theta
\\
&\le& \int_0^1\lk(
C\, \theta^{-d\ep }+ \frac 1 \theta \int_{\RR^d\setminus B(\theta \theta^{-\ep })}\gr y\gl^{-s}\,dy\,\rk)d\theta
\le \int_0^1\lk(
C\, \theta^{-d\ep }+ \frac 1 \theta \,\theta^{(1-\ep )(d-s)}\rk)d\theta\le C.
\EEQS}
Using $\partial_\eta^{\rho} e^{-i\la y,\eta\ra} =(-y)^\rho e^{-i\la y,\eta\ra}$, where $\rho$ is a multiindex, and integration by parts, gives
	\DEQS
\lqq{	\partial^\alpha_\xi m_R(x,\xi)}&&
\\
&=& \sum	\iint (-y)^{-\rho} e^{- i\la y,\eta\ra} \int_0^1\partial^\alpha_\xi \partial^{\rho}_\eta \lk[ \partial^\gamma   _{\xi'} q(x',\xi')\mid_{\xi'=\xi+\theta \eta\atop x'=x}
	\partial^\gamma   _{x'} p_R(x',\xi')\mid_{\xi'=\xi\atop x'=x+y}   \rk]\, d\theta
	\, dy\, d\eta
\\
&=& \sum	\iint (-y)^{-\rho} e^{- i\la y,\eta\ra} \int_0^1\partial^\alpha_\xi \theta^{d+1}\lk[ \partial^{\gamma+\rho}   _{\xi'} q(x',\xi')\mid_{\xi'=\xi+\theta \eta\atop x'=x}
	\partial^\gamma   _{x'} p_R(x',\xi')\mid_{\xi'=\xi\atop x'=x+y}   \rk]\, d\theta
	\, dy\, d\eta,
	\EEQS
Here the sum runs over all multiindex of the form $(d+1,0,\ldots,0)$, $(0,d+1, \ldots,0)$, $\cdots$, $(0,\ldots,d+1)$.
Analysing the proof of Theorem 3.9 \cite{pseudo2} we see that we have to estimate
	\DEQS
|	\partial^\alpha_\xi m_R(x,\xi) |
&\le&
	\iint \int_0^1 \theta^{d+1}
|y|^{-(d+1)} \lk[   \gr \xi+\theta \eta \gl^{-(|\gamma|+|\alpha|+\kappa+(d+1))}  
 \gr \xi \gl^{\kappa-|\alpha|} 
    \rk]\, d\theta
	\, dy\, d\eta
\\
&\le&
	\int \int_0^1 \theta^{d+1}
\lk[ \gr \xi   \gl^{-|\alpha|}   \gr \xi\gl^{-\frac 12 (|\gamma|+\kappa+(d+1))}   \gr  \theta \eta  \gl^{-\frac 12 (|\gamma|+\kappa+(d+1))}  
 \gr \xi \gl^{\kappa-|\alpha|} 
    \rk]\, d\theta
\, d\eta
\\
&\le&\gr \xi   \gl^{-2|\alpha|-(d+1)}
	\int \int_0^1 \theta^{d+1}
   \gr  \theta \eta  \gl^{-\frac 12 (|\gamma|+\kappa+(d+1))}  
\, d\theta
\, d\eta
\\
&\le&\gr \xi   \gl^{-2|\alpha|-(d+1)}
	\int \int_0^1 \theta
   \gr  \eta  \gl^{-\frac 12 (|\gamma|+\kappa+(d+1))}  
\, d\theta
\, d\eta.
	\EEQS
The calculation above gives that for $|\gamma|>d-k+1$, the integration with respect to $\eta$ and $\theta$ is finite.
	Taking into account Theorem \ref{theoremb1}, we can verify for $\xi\in\delta\CU_1, \,\delta>1$ that
	\DEQS
	\lqq{
		\sup_{|\alpha|\le d+1}\lk|\partial_\xi^\alpha  m_R(x,\xi) \rk|\gr |\xi|\gl ^{|\alpha |+1}  }
	&&
	\\
	&\lesssim&
	\sup_{1\le |\alpha|\le d+1\atop 1\le |\beta|\le 2d+1}
	\lk|\partial_\xi^\alpha  \partial_\xi^\beta  \lk[{q_R(x,\xi)}  \rk]\rk|
	\, \sup_{ |\delta|\le d+1 }\lk|\partial_x^\delta p_R(x,\xi)\rk|.
	\EEQS
	Hence, by the generalized Leibniz  rule (see \cite[p.\ 200, (A.1)]{pseudo2}) we have
	$$
	\lk\| m_R\rk\|_{\CA_{d+1,0;1,0}^{-1}} \le \lk\| q_R \rk\|_{\mbox{\rm\tiny Hyp}_{2d+1,0;1,0}^{\kappa}}\lk\| p_R\rk\|_{\tilde \CA_{2d+1,2d+1;1,0}^{\kappa,1}},
	$$
	from what it follows that
	$m_R(x,D)$ is a bounded operator with from $B_{p,r}^m(\RR^d)$ to $B_{p,r}^{m+1}(\RR^d)$.
	In addition, by the same analysis, we get
	$$
	\lk\| m_R\rk\|_{\CA_{d+1,0;1,0}^{0}} \le \lk\| q \rk\|_{\mbox{\rm\tiny Hyp}_{2d+1,0;1,0}^{\kappa}}\lk\| p_R\rk\|_{\tilde \CA_{2d+1,2d+1;1,0}^{\kappa-1,1}}.
	$$
	Observe, that  
	$$\lk\| p_R\rk\|_{\tilde \CA_{2d+4,d+3;1,0}^{\kappa-1,1}}\le \frac 1 R  \lk\| p_R\rk\|_{\tilde \CA_{2d+1,2d+1;1,0}^{\kappa,1}}.
	$$
	Therefore, analysing the symbols, $m_R(x,D)$  is a bounded operator from $B_{p,r}^{m}(\RR^d)$ to $B_{p,r}^{m}(\RR^d)$
	having norm
	$$\|m_R\|_{\CA_{d+1,0;1,0}^{0}}\le \frac 1R\, 
	\|m_R\|_{\CA_{d+1,0;1,0}^{-1}}.
	$$
	
	\medskip

	Now, let us go back to a slightly modified problem to verify for a given $f\in B_{p,r}^m(\RR^d)$, the regularity of $u$, where $u$ solves
	\DEQSZ\label{problem1}
p_R(x,D) u(x) = f(x),\quad x\in\RR^d.
	\EEQSZ
	From before, we know that 
	\DEQS
(q\sharp p_R)(x,\xi) &=&
q(x,\xi)p_R(x,\xi) + C(x,\xi) + 
m_R(x,\xi)
=I + C(x,\xi) + 
m_R(x,\xi).
	\EEQS
A careful analysis (see \eqref{careful}) shows that $C\in \CA^{-1}_{d+1;0;1,0}$, and
$$
\|C\|_{\CA^1_{d+1;0,1,0}}\le \frac 1R |C|_{\CA^{-1}_{d+1;0;1,0}}.
$$
	Due to the assumption on $R$ we know that  $R$ is such large that
	$$\Big(\|C\|_{\CA^1_{d+1;0,1,0}}+\| m_R\|_{\CA_{d+1,0;1,0}^{0} }\Big)\le \frac 16.
	$$

	First, we will show that if $f\in B_{p,r}^m(\RR^d)$, then it follows that  $u\in B_{p,r}^m(\RR^d)$.
	Suppose 
	for any $M\in\NN$ we have $|u|_{B_{p,r}^m}\ge M$. Since,  from before we know that
	\DEQS
	( I+ C(x,D)+m_R(x,D)) u= q(x,D)  f,
	\EEQS
	we get
	\DEQS
	| (I+C(x,D)+ m_R(x,D)) u|_{B_{p,r}^m} \ge \lk||u|_{B_{p,r}^m}-\frac 1R |u|_ {B_{p,r}^m}\rk|\ge \frac 56 |u|_ {B_{p,r}^m}.
	\EEQS
	On the other side,
	\DEQS
	\lk| (I+ C(x,D)+m_R(x,D)) u\rk| _{B_{p,r}^m}=\lk|  q(x,D)  f\rk|_  {B_{p,r}^m}\le \|q\|_{\CA^{0}_{d+1,0;1,0}} |f|_ {B_{p,r}^m}<\infty,
	\EEQS
	which leads to a contradiction, since we assumed that  for any $M\in\NN$ we have $|u|_{B_{p,r}^m}\ge M$. Hence, we know that $u\in B_{p,r}^m(\RR^d)$.
	In the next step, we will show that we have even $u\in B_{p,r}^{m+\kappa}(\RR^d)$ and calculate its norm in this space.
	Using similar arguments as above, we know 	by Theorem
\ref{bound1} and Remark \ref{bound11} that
	\DEQS
	|q(x,D) f |_{B^{m+\kappa}_{p,r}}\le \| q\|_{\CA^{-\kappa}_{d+1,0;1,0}} |f|_{B^{m}_{p,r}}
	\EEQS
	Similar as in the proof of Theorem 3.24 in \cite[p.\ 59]{pseudo2} we
	define
	$$\tilde q(x,D) :=\sum_{j=0}^{k}
	(-1)^j(C(x,D)+ m_R(x,D))^jq(x,D),
	$$
	where
	$$
(C(x,D)+	m_R(x,D))^j=\underbrace{(C(x,D)+m_R(x,D))\ldots (C(x,D)+m_R(x,D))}_{j \quad \mbox{times}}.
	$$
	Since the right hand side is an alternating sum, it follows by the identity
	$$
	q(x,D)p_R(x,D)=I+C(x,D)+m_R(x,D)u(x)
	$$
	that
	\DEQSZ\label{itera}
	\tilde q(x,D) p_R(x,D)=  I+(-1)^{k+1}(C(x,D)+ m_R(x,D))^{k+1}.
	\EEQSZ
	On the other side, since
	$$
	u(x)=q(x,D)f(x)-(C(x,D)+m_R(x,D))u(x),
	$$
	we have
	\DEQS
	q(x,D)f(x)
	&=& q(x,D)p_R(x,D)u(x)
	=\lk(I+C(x,D)+m_R(x,D)\rk)u(x).
	\EEQS
	Since $\lk|m_R(x,\xi)\rk|_{\CA_{d+1,0;1,0}^{0} }\le \frac 16$,
	the sequence $\{u_N:n\in\NN\}$
	defined by
	$$
	u_N(x)=\lk(I+\sum_{k=1}^N (-1)^k(C(x,D)+m_R(x,D))^k \rk)q(x,D)f(x),
	$$
	is bounded and a Cauchy sequence.
	Therefore, there exists a $u$ with $u_N\to u $ strongly,
	and, therefore,
	\DEQS
	\lqq{|u|_{B_{p,r}^{m+\kappa} }}
	&&
	\\
	&\lesssim & \|  q\|_{\mbox{\rm \tiny Hyp}
		_{d+1,0;1,0}^{\kappa,\kappa} } \lk(1+\sum_{k=1}^\infty \|C(x,D)+m_R\|_{\tilde\CA_{d+1,0;1,0}^{-1,1} }^k \rk)  |f|_{B_{p,r}^m}
	\\
	&\lesssim & \|  q\|_{\mbox{\rm \tiny Hyp}
		_{d+1,0;1,0}^{\kappa,\kappa} } \lk(1+\sum_{k=1}^\infty \lk(\frac{1}{6}\rk)^k \rk)  |f|_{B_{p,r}^m}
	\\
	&\lesssim& \frac 65\, \|  q\|_{\mbox{\rm \tiny Hyp}
		_{d+1,0;1,0}^{\kappa,\kappa} }  |f|_{B_{p,r}^m}.
	\EEQS
	This gives the assertion.
	\del{$$
		|u|_{B_{p,r}^{m+\kappa} }\lesssim \|   a\|_{\mbox{\rm \tiny Hyp}
			_{d+1,0;1,0}^{\kappa,\kappa} } \lk( 1+ |   a|_{\mbox{\rm \tiny Hyp} _{2d+4,d+3;1,0}^{\kappa} }|a|_{\CA_{d+1,0;1,0}^\kappa}\rk)  |f|_{B_{p,r}^m}
		$$
	}
\end{proof}

\appendix

\section{Symbol classes and pseudo--differential operators}\label{pseudo-app}

In this section, we shortly introduce pseudo--differential operators and their symbols. Also, we introduce the definitions and Theorems which are necessary for our purpose. For a detailed introduction on pseudo--differential operators and their symbols in the context of partial differential equations
we recommend the books \cite{pseudo2,pseudo,shubin,wong}, or the monograph of Kumano-go \cite{kumanago}, in the context of Markov processes we recommend the books \cite{Jacob-I,Jacob-II,Jacob-III} or the survey \cite{levymatters}.

In order to treat pseudo-differential operators, different classes of symbols have been introduced.
Here, we closely follow the definition of \cite{pseudo2}.

\begin{defn} 
Let 
 $\rho,\delta$ two real numbers such that $0\le \rho\le 1$ and $0\le \delta\le 1$.
  Let $S^m_{\rho,\delta}(\RR^ d\times \RR^ d)$ be the set of all functions $a:\RR ^d \times \RR^ d \to \mathbb{C}$, where
 \begin{itemize}
   \item  $a(x,\xi)$ is infinitely often differentiable, i.e.\ $a\in \cal^\infty_b(\RR^d \times\RR^d)$;
   \item
for any two multi-indices $\alpha$ and $\beta$ there exists $C_{\alpha,\beta}>0$  such that 
$$
\lk| \partial ^ \alpha_{\xi'} \partial ^ \beta_x a(x,\xi')\Big|_{\xi'=\gamma\xi} \rk|\leq C_{\alpha,\beta}  
\gr|\gamma\xi|\gl ^ {m-\rho|\alpha|}\ggx^{\delta| \beta|},\quad x\in \RR^d,\xi\in \CU_1,\,\gamma\ge 1 .
$$
 \end{itemize}
\end{defn}
We call any function $a(x,\xi)$ belonging to  $\cup_{m\in\RR} S^m_{0,0}(\RR^d ,\RR^d )$ a {\sl symbol}.
For many estimates, one does not need that the function is infinitely often differentiable.
It is often only necessary to know the estimates with respect to $\xi$ and $x$ up to a particular order.
For this reason, one also introduces the following classes. 
\begin{defn}(compare \cite[p. 28]{wong})
Let $m\in\RR$.  Let $\CA ^m_{k_1,k_2;\rho,\delta}(\RR ^d,\RR ^d)$ be the set of all functions $a:\RR ^d \times \RR^ d \to \mathbb{C}$, where
 \begin{itemize}
   \item  $a(x,\xi)$ is $k_1$--times differentiable in $\xi$ and $k_2$ times differentiable in $x$;
   \item 
for any two multi-indices $\alpha$ and $\beta$ with $|\alpha|\le k_1$ and $|\beta |\le k_2$,
there exists a  constant $C_{\alpha,\beta}>0$ depending only on $\alpha$ and $\beta$ such that
$$
\lk| \partial ^ \alpha_{\xi'} \partial ^ \beta_x a(x,\xi')\Big|_{\xi'=\gamma\xi} \rk|\leq C_{\alpha,\beta}  
\gr|\gamma\xi|\gl ^ {m-\rho|\alpha|}\ggx^{\delta| \beta|},\quad x\in \RR^d,\xi\in \CU_1,\,\gamma\ge 1 .
$$
\end{itemize}
\end{defn}
\noindent
Moreover, one can introduce a  semi--norm in $\CA^m_{k_1,k_2;\rho,\delta}(\RR^d,\RR ^d)$ by
\DEQS
\lqq{ \| a\| _{\CA_{k_1,k_2;{\rho,\delta}}^m }
} &&
\\
&=& \sup_{|\alpha|\le k_1,|\beta|\le k_2} \, \sup_{(x,\xi,\gamma)\in \RR^d
\times \CU_1\times[1,\infty)} \lk| \partial ^ \alpha_\xi \partial ^ \beta_x a(x,\xi) \rk|
\gr|\gamma\xi|\gl  ^ {\rho|\alpha|-m}\ggx^{\delta|\beta|}, \quad a\in \CA ^m_{k_1,k_2;\rho,\delta}(\RR^ d\times \RR^ d).
\EEQS
\begin{rem}
For $m_1\le m_2$ it follows that  $S^{m_1}_{\rho,\delta}(\RR^d ,\RR ^d)\supseteq S^{m_2}_{\rho,\delta}(\RR^d ,\RR ^d)$ and $\CA^{m_1}_{k_1,k_2;\rho,\delta}(\RR^d ,\RR ^d)\supseteq \CA^{m_2}_{k_1,k_2;\rho,\delta}(\RR^d ,\RR ^d)$, $k_1,k_2\in\NN$.
\end{rem}

\begin{defn}(compare \cite[p.28, Def. 4.2]{wong})
Let $a(x,\xi)$ be a symbol. 
Then, to  $a(x,\xi)$ corresponds an operator $a(x,D)$ defined by
$$
a(x,D) \,u (x): = \int_{\RR^d} e^ {i\la x,\xi\ra } a(x,\xi)\,\hat u(\xi)\, d\xi,\quad u\in \CSS( 
\RR^ d)
$$
and called pseudo--differential operator.
\end{defn}

Clearly, $a(x,D) $ is bounded from $\CSS(\RR^d)$  into $\CSS'(\RR^d)$.

\begin{cor}\label{limitting}
Let $u\in H^m_2(\RR^d)$ for all  $m\in\RR$. Then 	$$
a(x,D) \,u (x): = \int_{\RR^d} e^ {i\la x,\xi\ra } a(x,\xi)\,\hat u(\xi)\, d\xi,
$$
is well defined with $a(x,D)$ being a pseudo-differential operator.
\end{cor}
\begin{proof}
Let $v,\phi\in \CSS(\RR^d)$. Then consider,
\DEQS
\lqq{ (a(x,D)v,\phi)_{L^2(\RR^d)}= \int_{\RR^d} \int_{\RR^d} e^ {i\la x,\xi\ra } a(x,\xi)\,\hat v(\xi)\, d\xi\overline{\phi(x)}\,dx
} &&
\\
&=&  \int_{\RR^d} \int_{\RR^d} e^ {ix^T\xi} a(x,\xi )\overline{\phi(x)}\,dx\,\hat v(\xi)\, d\xi.
\\
&=&  \int_{\RR^d} \hat v(\xi)\overline{\int_{\RR^d} e^ {i\la x,\xi\ra } \overline{a(x,\xi )}\phi(x)\,dx}\,\, d\xi,
\EEQS
where we used Fubini theorem and the fact that $\phi, \hat{v}\in \CSS(\RR^d)$. In Lemma 3.31 in \cite{pseudo2} showed that
$$w(\xi)=\int_{\RR^d} e^ {i\la x,\xi\ra } \overline{a(x,\xi )}\phi(x)\,dx\in \CSS(\RR^d),$$
where
 $a(x,\xi)\in S^{m}_{1,0}(\RR^ d\times \RR^ d)$ with $m\in \RR$. Therefore we have,
 $$(a(x,D)v,\phi)_{L^2(\RR^d)}=(v,a^{*}(x,D)\phi)_{L^2(\RR^d)},$$
 such that $ a^{*}(x,D)\phi \in \CSS(\RR^d)$. Now let $u\in H^m_2(\RR^d)$.
 There exist $\{u_n\}_{n\in\NN}\subset\CSS(\RR^d)$ such that (see Corollary 3.42 in \cite{pseudo2}), $$\lim_{n\to\infty}\langle u_n-u,\phi\rangle=0,$$ for any $\phi\in\CSS(\RR^d)$. Therefore due to the above facts we have  $$\lim_{n\to\infty}\langle a(x,D)u_n,\phi\rangle=\lim_{n\to\infty}\langle u_n,a^{*}(x,D)\phi\rangle= \langle u,a^{*}(x,D)\phi\rangle =\langle a(x,D)u,\phi\rangle<\infty. $$
 Then we could conclude that the Fourier integral representation of $a(x,D)u$ is well defined in $H^m_2(\RR^d)$ with $m\in\RR$.
 See Theorem 3.41 in \cite{pseudo2} as well.%
\end{proof}

One can easily see under which conditions  $a(x,D)$ is also bounded from $L^p(\RR^d)$ into $L^p(\RR^d)$, $1\le p\le \infty$.
To see it, first, observe that the operator can also be represented by a kernel of the form
$$
a(x,D)  f(x) = \int_{\RR^d} k(x,x-y) f(y)\, dy, \quad x\in\RR^d,
$$
where the kernel is given by the inverse Fourier transform
$$
k(x,z) = \CF_{\xi\to z} \lk[ a(x,\xi)\rk] (z)\footnote{$ \CF_{\xi\to z}[f(x,\xi)](z)=\int_{\RR^d}e^{ -2\pi i \la\xi, z\ra }a(x,\xi)\, d\xi$.}.
$$
Differentiation gives the following estimate 
$$
\lk| k(x,z)\rk| \le C\, \lk| \partial _\xi ^\alpha p(x,\xi)\rk|\, |z|^{-\alpha}.
$$
By this estimate and the Young inequality for convolutions one can calculate bounds of the operator between Lebesgue spaces, like
$$
|a(x,D) f|_{L^q}\le \lk\|a\rk\|_{\CA^0_{\gamma,0;1,0}} |f|_{L^q},
$$
for $\gamma\ge  {d+1}$. In case, we have additional regularity of the functions, or the function is a distribution, it is not that obvious.
The next Theorem gives characterize the action of a pseudo--differential operator on Besov spaces.
\begin{thm}\label{bound1}
(compare  \cite[Theorem 6.19, p.\ 164]{pseudo2}) Let $\kappa,m\in\RR$, $a(x,\xi)\in S^{\kappa}_{1,0}(\RR^ d\times \RR^ d)$ and $1\le p,r\le \infty$. Then, $a(x,D):B^{\kappa +m}_{p,r}(\RR^d)\to B^{m}_{p,r}(\RR^d)$ is a linear and bounded operator.
\end{thm}

\begin{rem}\label{bound11}
Tracing step by step of the proof of Theorem 6.19 in \cite[p.\ 164]{pseudo2}, one can see that for all
$\kappa,m\in\RR$, $a(x,\xi)\in S^{\kappa}_{1,0}(\RR^ d\times \RR^ d)$ and $1\le p,r\le \infty$ and any $k\ge d+1$
the following inequality holds 
$$
\lk|a(x,D) f\rk|_{B^{m}_{p,r}} \le \lk\| a\rk\|_{\CA^\kappa_{k,0;\delta,0}}\, \lk| f\rk|_{B^{\kappa+m}_{p,r}}.
$$
\end{rem}

\medskip

To analyze the composition of two operators of given symbols, one has to evaluate
a so-called oscillatory integral. In particular, for any  $\chi\in\CSS(\RR^d\times \RR^d)$ with $\chi(0,0)=1$
and $a\in \CSS(\RR^d\times \RR^d)$,
we define the oscillatory integral by
\DEQS
\OS \iint e ^{ -iy\eta} a(y,\eta) \, dy\, d\eta :=\lim_{\ep\to 0}  \iint_{\RR^d\times\RR^d} \chi(\ep  y,\ep\eta) \, e ^{ -i\la y, \eta\ra } a(y,\eta) \, dy\, d\eta.
\EEQS
To calculate the oscillatory integral, the following Theorem is essential.
\begin{thm}\label{theoremb1}(compare \cite[Theorem 3.9,  p.\ 46]{pseudo2})
Let $m\in \RR$, $a\in\CA_{(d+1+m)\wedge 0,d+1;1,0}^m(\RR^d\times\RR^d)$,  and let $\chi\in\CSS(\RR^d\times \RR^d)$ with $\chi(0,0)=1$. Then the oscillatory integral
\DEQS
\OS \iint e ^{ -i\la y, \eta\ra } a(y,\eta) \, dy\, d\eta
\EEQS
exists and
\DEQS
\lk| \OS\iint e ^{ -i\la y, \eta\ra } a(y,\eta) \, dy\, d\eta \rk|\le C_{m,d} \lk\|a\rk\|_{\CA^m_{(d+1+m)\wedge 0,d+1;1,0}}.
\EEQS
\end{thm}

\begin{cor}\label{cor3.10}
(compare \cite[Corollary 3.10,  p.\ 48]{pseudo2}) Let  $a_j\in S^m_{1,0}(\RR^d\times\RR^d)$ be a bounded sequence in
$\CA_{d+1+m,d+1;\rho,\delta}^m(\RR^d\times \RR^d)$ such that there exists some $a\in \CA_{d+1+m,d+1;\rho,\delta}^m(\RR^d\times\RR^d)$
$$
\lim_{j\to\infty} \partial ^\alpha _\eta\partial^\beta _ya_j(y,\eta) =  \partial_\eta ^\alpha \partial_y^\beta a(y,\eta),
$$
for any $|\alpha|\le d+m+1$, $|\beta|\le d+1$, $y\in \RR^d$ and $\eta\in \RR^d$. Then
\DEQS
\lim_{j\to\infty} \OS\iint e^{-i \la y, \eta\ra  } a_j(y,\eta)\, dy\, d\eta=  \OS  \iint e^{-i \la y, \eta\ra  } a(y,\eta)\, dy\, d\eta.
\EEQS
\end{cor}

\del{Note, that  for $a\in\CSS^m_{1,0}(\RR^d\times\RR^d)$ we have
$$
a(x,D) u(x) =\OS \iint e^{-i\la(x-y),\xi\ra } a(x,\xi) u(y)\, dy\, d\eta.
$$}

With the help of the oscillatory integral, one can show that the composition of two pseudo-differential operators
is
again a pseudo-differential operator.
Using formal calculations, an application of the Taylor formula leads to the
following characterization.

\begin{thm}\label{product1}
(compare  \cite[Theorem 3.16, p.\ 55]{pseudo2})
Let $a_1(x,\xi)\in S ^{m_1}_{1,0}(\RR^ d\times \RR^ d) $ and $a_2(x,\xi)\in S^{m_2}_{1,0}(\RR^ d\times \RR^ d)$. Then the composition  $a_1(x,D)\,a_2(x,D)$
is again a pseudo--differential operator,  whose symbol we denote by $[a_1\,\sharp\, a_2](x,\xi)$, and we have 
$$[a_1\,\sharp\, a_2](x,\xi)\in S^{m_1+m_2}_{1,0}(\RR^ d\times \RR^ d).$$
 Moreover, it can be expanded asymptotically as follows 
\DEQSZ\label{prodcomp}
[a_1\,\sharp\, a_2](x,\xi)\sim \sum_{\alpha} \frac {1\del{(-i)^ {|\alpha|}}} {\alpha !} \lk( \partial_\xi^ \alpha a_1(x,\xi)\rk) \lk( \partial_x^ \alpha a_2(x,\xi)\rk).
\EEQSZ
To be more  precise,  equation \eqref{prodcomp} means that
\DEQSZ \label{prodsym}
[a_1\,\sharp\, a_2](x,\xi)-\sum_{|\alpha|\le N} {1\del{(-i)^ {|\alpha|}}\over \alpha !} \lk( \partial_\xi^ \alpha a_1(x,\xi)\rk) \lk( \partial_x^ \alpha a_2(x,\xi)\rk)
\EEQSZ
belongs to $S^ {m_1+m_2-N}_{1,0}(\RR^ d\times \RR^ d)$ for every positive integer $N$.
\end{thm}
\begin{rem}\label{compositionremark}
Following the proof of Theorem 3.16 \cite[p. 53]{pseudo2},
one observes that
\DEQSZ \label{prodsym}
\lqq{ [a_1\,\sharp\, a_2](x,\xi)-\sum_{|\alpha|\le N} {1\del{(-i)^ {|\alpha|}}\over \alpha !} \lk( \partial_\xi^ \alpha a_1(x,\xi)\rk) \lk( \partial_x^ \alpha a_2(x,\xi)\rk)}
\\
&=&(N+1) \sum_{|\alpha|=N+1} \frac 1
 {\alpha !} \OS \iint e^{-i\la y, \eta\ra } \eta ^\alpha
  r_\alpha (x,\xi,y,\eta)\, dy\, d\eta
\\
&=&(N+1) \sum_{|\alpha|=N+1} \frac 1 
 {\alpha !} \OS \iint e^{-i\la y, \eta\ra } D_y^\alpha
  r_\alpha (x,\xi,y,\eta)\, dy\, d\eta
 \nonumber
\EEQSZ
with
$$r_\alpha (x,\xi,y,\eta)=\int_0^1 \lk[ \partial^\alpha  _{\xi'} p_1(x',\xi')\mid_{\xi'=\xi+\theta \eta\atop x'=x}
 \partial^\alpha  _{x'} p_2(x',\xi')\mid_{\xi'=\xi\atop x'=x+y}   (1-\theta)^N\rk]\, d\theta.
$$
\end{rem}
\del{\begin{defn}
Let $a(x,\xi)$ be a symbol in $S^m_{0,0}(\RR^d\times\RR^d)$ and $a(x,D)$ the associated pseudo--differential operator.
Suppose there exists a linear operator $a^\ast(x,D):\CSS(\RR^d)\to\CSS(\RR^d)$ such that
$$
(a(x,D)f,g) = (f,a^\ast(x,D) g),\quad f,g\in \CSS(\RR^d).
$$
Then we call $a^\ast(x,D)$ the formal adjoint operator of the operator $a(x,D)$.
\end{defn}}
\del{
The next point of interest is to classify the continuous operator and the exact Besov spaces on which they are operating.
\begin{thm}\label{bound1}
(compare  \cite[Theorem 6.19, p.\ 164]{pseudo2}) Let $\kappa,m\in\RR$, $a(x,\xi)\in S^{\kappa}_{1,0}(\RR^ d\times \RR^ d)$ and $1\le q,r\le \infty$. Then $a(x,D):B^{\kappa +m}_{p,r}(\RR^d)\to B^{m}_{p,r}(\RR^d)$ is a linear and bounded operator.
\end{thm}

\begin{rem}\label{bound11}
Tracing step by step of the proof of Theorem 6.19 in \cite{pseudo2}, one can see that for any $\kappa\in\RR$ and  $k\ge d+1+\kappa$ there exists a constant $C>0$ such that
we have  for any $a \in \CA^\kappa_{k,0;\delta,0}(\RR^d\times \RR^d)$ and $f\in B^{\kappa+m}_{p,r}(\RR^d)$
$$
\lk|a(x,D) f\rk|_{B^{m}_{p,r}} \le C \lk| a\rk|_{\CA^\kappa_{k,0;\delta,0}}\, \lk| f\rk|_{B^{\kappa+m}_{p,r}} .
$$
\end{rem}

}

\def\polhk#1{\setbox0=\hbox{#1}{\ooalign{\hidewidth\lower1.5ex\hbox{`}\hidewid%
th\crcr\unhbox0}}}
  \def\polhk#1{\setbox0=\hbox{#1}{\ooalign{\hidewidth\lower1.5ex\hbox{`}\hidew%
idth\crcr\unhbox0}}}

\end{document}